\documentclass[a4paper,oneside,reqno]{amsart}

\usepackage[T1]{fontenc}
\usepackage[utf8]{inputenc}
\usepackage[english]{babel}
\usepackage[osf,sc]{mathpazo}
\usepackage{eucal}
\normalfont

\usepackage{fullpage}
\usepackage[dvipsnames]{xcolor}
\usepackage{hyperref}
\usepackage{bookmark}
\usepackage{enumitem}
\usepackage[normalem]{ulem}

\newcommand\hmmax{0}
\newcommand\bmmax{0}
\hypersetup{
	colorlinks=true,
	linktocpage=true,
	pdfstartpage=1,
	pdfstartview=FitV,
	breaklinks=true,
	pdfpagemode=UseNone,
	pageanchor=true,
	plainpages=false,
	bookmarksnumbered,
	bookmarksopen=true,
	bookmarksopenlevel=1,
	hypertexnames=false,
	pdfhighlight=/O,
	urlcolor=BrickRed,
	linkcolor=RoyalBlue,
	citecolor=ForestGreen
}

\usepackage{graphicx,float}
\usepackage{caption}
\usepackage{pgfplots}
\usepackage{tikz}
\usetikzlibrary{calc,decorations.markings}
\pgfplotsset{compat=1.16}

\usepackage[textsize=footnotesize]{todonotes}
\presetkeys{todonotes}{color=Apricot}{}
\usepackage{amsmath,amssymb,amsthm}
\usepackage{mathtools,thmtools} 
\usepackage{bm,braket,stmaryrd,dsfont,mathrsfs,esint}
\numberwithin{equation}{section}

\usepackage[english]{varioref}
\usepackage[noabbrev,capitalise]{cleveref}

\Crefname{thm}{theorem}{theorems}
\Crefname{thm}{Theorem}{Theorems}

\Crefname{lem}{lemma}{lemmata}
\Crefname{lem}{Lemma}{Lemmata}

\Crefname{prop}{proposition}{propositions}
\Crefname{prop}{Proposition}{Propositions}

\Crefname{defn}{definition}{definitions}
\Crefname{defn}{Definition}{Definitions}

\Crefname{rem}{remark}{remarks}
\Crefname{rem}{Remark}{Remarks}

\newtheorem{theoremA}{Theorem}
\renewcommand{\thetheoremA}{A}

\newtheorem{theoremB}{Theorem}
\renewcommand{\thetheoremB}{B}

\newcommand{\dd}{\mathrm{d}}
\newcommand{\ii}{\mathrm{i}}
\newcommand{\bbraket}[1]{\llbracket #1 \rrbracket}

\DeclarePairedDelimiter{\card}{\lvert}{\rvert}
\DeclareMathOperator*{\Res}{Res}

\newcommand{\MM}{\mathbb{M}}
\newcommand{\Tutte}{\mathsf{Tut}}
\newcommand{\Aut}{\textnormal{Aut}}

\newcommand{\nodisc}{\textup{no }(0,1)}
\newcommand{\std}{\textnormal{std}}
\newcommand{\spl}{\textnormal{spl}}
\newcommand{\ind}{\textnormal{ind}}

\newcommand{\barsuperscript}[2]{%
  \overline{\vphantom{\overline{#1}}\smash{{#1}^{#2}}}%
}
\newcommand{\bbarsuperscript}[2]{%
  \overline{\vphantom{\overline{#1}}\smash{\overline{#1}^{#2}}}%
}

\theoremstyle{plain}
\newtheorem{thm}{Theorem}[section]
\newtheorem{prop}[thm]{Proposition}
\newtheorem{lem}[thm]{Lemma}

\theoremstyle{definition}
\newtheorem{defn}[thm]{Definition}
\newtheorem{rem}[thm]{Remark}

\usepackage[
	style=alphabetic,
	sorting=nyt,
	maxbibnames=99,
	maxcitenames=4,
	maxalphanames=4,
	minalphanames=4,
	giveninits=true,
	backend=biber
]{biblatex}
\renewcommand*{\bibfont}{\small}
\renewbibmacro{in:}{}
\usepackage{csquotes}
\title{A bijective topological recursion for maps}

\author{Ga\"etan Borot}
\address[G. Borot]{Institut f\"ur Mathematik und Institut f\"ur Physik, Humboldt-Universit\"at zu Berlin, 10099 Berlin, Germany}
\email{gaetan.borot@hu-berlin.de}

\author{Alessandro Giacchetto}
\address[A. Giacchetto]{Department Mathematik, ETH Z\"urich, CH-8006 Z\"urich, Switzerland}
\email{alessandro.giacchetto@math.ethz.ch}

\author{Lasse Merkens}
\address[L. Merkens]{Institut f\"ur Physik, Humboldt-Universit\"at zu Berlin, 10099 Berlin, Germany}
\email{lasse.merkens@hu-berlin.de}


\begin{document}

\begin{abstract}
	We prove bijectively a recursive excision formula for maps of arbitrary topology. This yields the topological recursion formulae governing their enumeration, already at the level of formal generating series and without any analyticity assumption. We extend the result to maps carrying self-avoiding loop models and to stuffed maps, i.e. a variant of maps allowing faces of arbitrary topology. Our results give a precise combinatorial meaning to all terms of the (blobbed) topological recursion known for (stuffed) maps, including the recursion kernel, which had so far been missing. The construction relies on the simple idea of iterating Tutte's algorithm until the topology changes. Equivalently, it can be interpreted as a pair-of-pants decomposition driven by a path issuing from the root of the first boundary, providing an analogue for maps of the Mirzakhani--McShane identity for hyperbolic surfaces.
\end{abstract}

\maketitle

\section{Introduction}

\subsection{Motivation}
A map is an isotopy class of proper embeddings of a finite graph into an oriented compact surface such that the complement of the embedded graph is a union of topological discs. Maps may be viewed as discretised surfaces obtained by gluing polygons, and have been studied combinatorially since Tutte \cite{Tut63}. They also arise as the Feynman diagrams of random Hermitian matrix models \cite{tH74,BIPZ78}. The connection with random matrices was rediscovered in \cite{HZ86} and used to compute the Euler characteristic of moduli spaces of Riemann surfaces. In physics, maps provide a discretisation of two-dimensional quantum gravity, and this connection led Witten to his celebrated conjecture \cite{Wit90} on the enumerative geometry of the moduli space of curves, later proved by Kontsevich \cite{Kon92}.

Building on earlier work \cite{ACM92,ACKM93}, it was shown in \cite{Eyn04,Eyn11,Eyn16} that Hermitian matrix models, and hence generating series of maps of arbitrary topology, satisfy the topological recursion in the sense of Eynard and Orantin \cite{EO07}. This extends to richer classes of maps, such as maps carrying the Ising model \cite{EO05,EO08bndry,EO08mixed}, self-avoiding loop models \cite{BE11,BEO15}, stuffed maps \cite{Bor14}, and constellations with internal faces \cite{BCCG24}. For stuffed maps, which underlie the most general unitarily invariant random Hermitian matrix models, including those with multi-trace interactions, the relevant recursive structure is the blobbed topological recursion \cite{BS17}.

However, the known derivations are not fully \emph{bijective}. They rely on analytic continuation and on a detailed study of the analytic properties of generating series. Although the resulting formulae are reminiscent of pairs-of-pants excision, the individual terms of topological recursion, the recursion kernel and, in the stuffed case, the blob terms, do not a priori come with a direct combinatorial meaning. This naturally leads to the following question:
\begin{center}
\textsf{
	Can one derive topological-recursion-type formulae for maps \\ 
	by a purely combinatorial, bijective procedure?
}
\end{center}
The purpose of this work is to answer this question positively. We give bijective derivations of topological-recursion-type formulae, valid already at the level of formal generating series and without analyticity assumptions. We do so for models of increasing complexity: maps; maps with tubes, in which annular faces are allowed and which are combinatorially equivalent to maps carrying configurations of self-avoiding loops; and stuffed maps, in which the faces can have arbitrary topology.

\subsection{Main idea}
The starting point of our approach is Tutte's decomposition, which consists of erasing the root-edge of the first boundary. Instead of applying it only once, we iterate it as long as the topology is preserved, and stop just before the first topology-changing step. This produces two objects: a \emph{core map}, on which one further application of Tutte's procedure would change topology, and a \emph{skin map}, namely an annular remainder assembling the pieces removed along the way. We call this procedure \emph{skinning}.

A slight variant of this decomposition leads not only to an annular remainder, but to the excision of a pair of pants embedded in the original map. This gives a direct combinatorial origin for the recursive structure that later appears in topological recursion. In the case of maps, it already yields the desired formula. For maps with tubes and stuffed maps, the same mechanism persists, but additional terms appear, reflecting the presence of more complicated internal faces.

\subsection{Main results}
We formulate our main results for the enumeration of stuffed maps, which contains maps and maps with tubes as special cases. We denote by $W_{g,n}^{\ast}(x_1,\ldots,x_n)$ the generating series of stuffed maps of genus $g$ with $n$ labelled rooted boundaries. This series includes formal Boltzmann weights for vertices and for each type of internal face, together with a factor $x_i^{-\ell_i-1}$ when the $i$-th boundary has degree $\ell_i$.

Besides $W_{g,n}^{\ast}$, the formulae involve two auxiliary series. The first is $S^{\ast}(x_1,x)$, which enumerates skin maps, namely the annular remainder produced by iterating Tutte's procedure until topology changes. The second is $V_{g,n}(x_1;x_2,\ldots,x_n)$, which collects the exceptional indecomposable contributions arising in the excision procedure; it does not appear for maps or maps with tubes, but only for stuffed maps. The series $S^{\ast}$ plays the role of the recursion kernel, while $V_{g,n}$ gives rise to the blob term.

Our first main result is an excision formula expressing $W_{g,n}^{\ast}$ recursively in terms of lower-complexity generating series, together with the auxiliary series $S^{\ast}$ and $V_{g,n}$.

\begin{theoremA}[Excision formulae]
\label{intro:excision}
	We have the excision relations
	\begin{multline}
		W_{g,n}^{\ast}(x_1,x_2,\ldots,x_n)
		=
		\Bigg\langle
			S^{\ast}(x_1,x)
			\Bigg(
				\mathcal{M}_{y}W_{g-1,n+1}^{\ast}(x,y,x_2,\ldots,x_n)\big|_{y=x} \\
				\qquad\qquad\qquad
				+
				\sum_{\substack{h' + h'' = g \\ J' \sqcup J'' = \set{2,\ldots,n}}}^{\nodisc}
					W_{h',1+\card{J'}}^{\ast}(x,\bm{x}_{J'})
					\mathcal{M}_{x} W_{h'',1+\card{J''}}^{\ast}(x,\bm{x}_{J''})
			\Bigg)
			+ W_{0,2}^{\ast}(x_1,x)V_{g,n}(x;x_2,\ldots,x_n)
		\Bigg\rangle_x
	\end{multline}
	for $(g,n) \neq (0,1),(0,2),(1,1)$, and the low-topology cases
	\begin{equation}
	\begin{aligned}
		W_{0,2}^{\ast}(x_1,x_2)
		& =
		\left\langle S^{\ast}(x_1,x) \frac{W_{0,1}^{\ast}(x)}{(x_2 - x)^2} \right\rangle_x, \\
		W_{1,1}^{\ast}(x_1)
		& =
		\left\langle
			S^{\ast}(x_1,x) 
				\left.\mathcal{M}_{y}\widetilde{W}^{\ast}_{0,2}(x,y)\right|_{y=x}
			+
			W_{0,2}^{\ast}(x_1,x)V_{1,1}(x)\right\rangle_x.  
	\end{aligned}
	\end{equation}
	Here $\langle F(x) \rangle_x$ denotes the coefficient of $x^{-1}$ in a formal series $F(x)$, $\bm{x}_{I}$ denotes the tuple $(x_i)_{i \in I}$ for any subset $I \subseteq \set{1,\ldots,n}$, and the ``no $(0,1)$'' superscript indicates that we discard $(h',1+\card{J'}) = (0,1)$ and $(h'',1+\card{J''}) = (0,1)$ in the sum. Furthermore, $\mathcal{M}_{x_i}W_{g,n}^{\ast}$ is the sum of $W_{g,n}^{\ast}$ and the generating series of stuffed maps for which the $i$-th boundary is fully glued to a (boundary or internal) face; $\mathcal{M}_{x_i}\widetilde{W}_{0,2}^{\ast}(x_1,x_2)$ is defined similarly, but excluding annuli in which the two boundary faces are fully glued to each other.
\end{theoremA}

A more explicit expression for $\mathcal{M}_{x_i}W_{g,n}^{\ast}$ and $V_{g,n}$ in terms of lower-complexity generating series is given later in \Cref{lem:R1W:stable,lem:V:explicit}. In particular, \Cref{intro:excision} is genuinely recursive in the topological complexity $2g-2+n$.

Our second main result identifies the generating series of skin maps in terms of the disc and annuli data, in a form that makes its role as recursion kernel manifest.

\begin{theoremB}[Skin enumeration]
\label{intro:skin:enum}
	There is an explicit choice of antiderivative for which
	\begin{equation}
		S^{\ast}(x_1,x)
		=
		\frac{
			{\displaystyle\int^{x}}
			\mathcal{D}_{y} W_{0,2}^{\ast}(x_1,y)\dd y
		}{
			-\mathcal{D}_{x}W_{0,1}^{\ast}(x)
		},
	\end{equation}
	where $\mathcal{D}_{x_i}W_{g,n}^{\ast} \coloneqq \frac{1}{2}(W_{g,n}^{\ast} + \mathcal{M}_{x_i}W_{g,n}^{\ast} )$.
\end{theoremB}

Simpler bijective functional relations for maps and maps with tubes are proved in \Cref{sec:ord:maps,sec:maps:tubes}, while stuffed maps are treated in \Cref{sec:stuffed:maps}. For instance, in the case of maps, \Cref{intro:skin:enum} takes the concrete form (see~\Cref{skin:enum:maps})
\begin{equation}
	S(x_1,x)
	=
	\frac{
		{\displaystyle\int^{x}_{\infty}}
		\left(
			2W_{0,2}(x_1,y)
			+
			\frac{1}{(x_1-y)^2}
		\right)\dd y
		-
		\partial_q W_{0,1}(x_1)
	}{
		x - \sum_{k \ge 1} t_k x^{k-1} - 2W_{0,1}(x)
	},
\end{equation}
where $q$ is the vertex-weight and $t_k$ is the weight for internal faces of degree $k$.

\subsection{More on the strategy}
Let us explain the strategy in more detail in the case of maps. In his seminal work \cite{Tut63}, Tutte described a bijection giving a recursive construction of maps with disc topology, i.e. of topology $(g,n)=(0,1)$, by erasing the root-edge and examining the resulting map. We apply the same procedure to a map $\mathfrak{m}$ of general topology $(g,n)$, and iterate it until the resulting map no longer contains a component of topology $(g,n)$. This yields a bijection
\begin{equation}
	\mathfrak{m} \longmapsto (\mathfrak{s},\hat{\mathfrak{m}}),
\end{equation}
where the skin $\mathfrak{s}$ is an annular map formed by the pieces removed during the iteration and $\hat{\mathfrak{m}}$ is a core map on which one further application of Tutte's procedure would change topology, see~\Cref{fig:pants:decomp}. This skinning bijection already yields a first recursive formula. Moving a superficial piece of the core map to the skin map then reveals an embedded pair of pants and gives a bijection
\begin{equation}
	\mathfrak{m} \longmapsto (\mathfrak{p},\check{\mathfrak{m}}),
\end{equation}
where $\mathfrak{p}$ is the excised pair of pants and $\check{\mathfrak{m}}$ is the remaining map of simpler topology, see~\Cref{fig:pants:decomp} again. This leads to the excision formula of \Cref{intro:excision}.

\begin{figure}
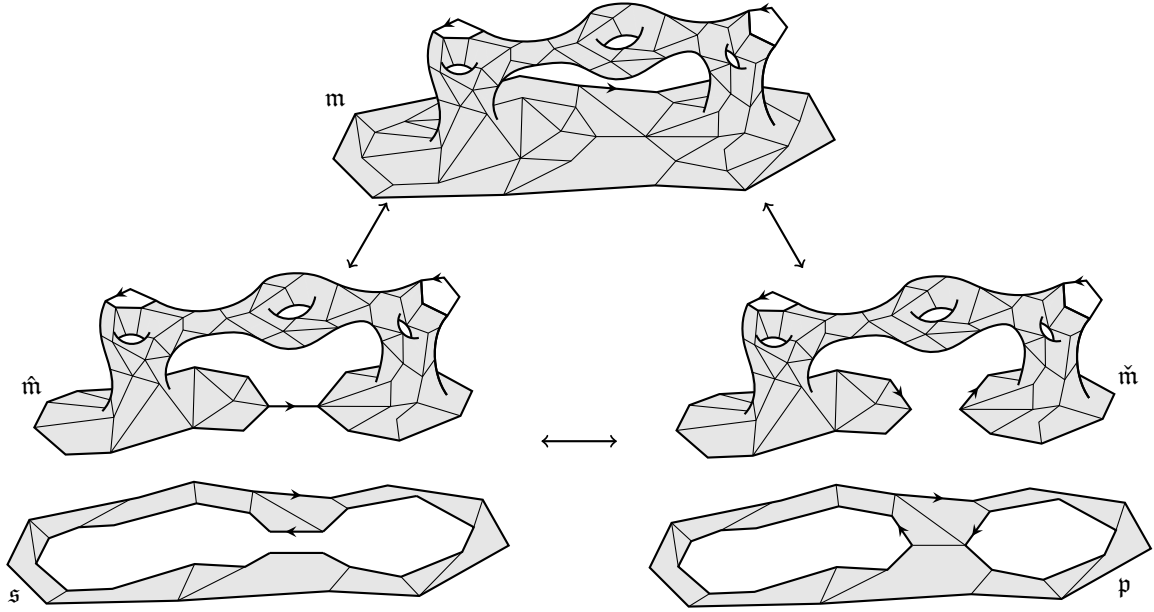

	\centering

	\caption{Decomposition of a map $\mathfrak{m}$ (top) into its skin $\mathfrak{s}$ and the core map $\hat{\mathfrak{m}}$ (bottom left) and a pair of pants $\mathfrak{p}$ and a map of simpler topology $\check{\mathfrak{m}}$ (bottom right).}
	\label{fig:pants:decomp}
\end{figure}

It remains to enumerate the possible skins, or equivalently the possible excised pairs of pants. For this, one more ingredient is needed. It is provided by pointed discs: although ordinary discs cannot be skinned, pointed discs can be treated by the same procedure. Their enumeration gives the additional relation needed to fix the integration constant in the formula for skin maps in \Cref{intro:skin:enum}.

For maps with tubes and stuffed maps, the guiding ideas are the same, but Tutte's procedure gives rise to additional situations. In particular, there are stuffed maps in which excising a pair of pants starting from the first boundary does not reduce the complexity $2g - 2 + n$, due to internal faces with negative Euler characteristic. These indecomposable maps are exactly those counted by the last term in \Cref{intro:excision}, namely the term involving $V_{g,n}$. The factor $W_{0,2}^{\ast}$ accounts for an annulus propagating the first boundary to the first internal face with negative Euler characteristic that is met while skinning.

\subsection{Connection with topological recursion and Mirzakhani's recursion}
In most applications of Eynard--Orantin theory to enumerative geometry, the local involution and the specific form of the recursion kernel arise from analytic manipulations, and the individual terms of the recursion do not themselves carry a direct enumerative meaning. The structure of our formulae shows that the enumeration of maps, maps with tubes, and stuffed maps belongs instead to the class of ``bijectively meaningful'' topological recursions.

More precisely, our formulae match the universal pattern of the topological recursion \cite{EO07}, as well as its blobbed extension \cite{Bor14,BS17}. Up to simple normalisations, the correspondence is as follows:
\begin{itemize}
	\item \emph{Formal monodromy and formal discontinuity.} The operation $\mathcal{M}_{x}$ plays the role of the local involution, sending a point to its counterpart on the other sheet of the spectral curve. The operation $\mathcal{D}_{x}$ corresponds to the odd part with respect to this involution, that is, the discontinuity across the cut given by the difference between the two values on the two sheets.
	
	\item \emph{Recursion kernel.} The generating series $S^*$ of skin maps plays the role of the recursion kernel, being expressed as the quotient of an integrated odd part of $W_{0,2}^{\ast}$ by the odd part of $W_{0,1}^{\ast}$.
	
	\item \emph{Blob.} In the stuffed case, the term involving $V_{g,n}$ is the analogue of the blob term.
\end{itemize}
The operations $\mathcal{M}_{x}$ and $\mathcal{D}_{x}$ are not operators stricto sensu at the level of generating series, since they combine in a single series the enumeration of two different objects. They become genuine operators after analytic continuation.

When the face and vertex weights are specialised to real values for which the relevant generating series enjoy suitable analytic properties, we show in \Cref{sec:analytic} that our excision formula implies the usual (blobbed) topological recursion on the associated spectral curve, thereby recovering the results of \cite{Eyn04,BE11,Bor14,BEO15,Eyn16}. The combinatorial formulae of Theorems~\labelcref{intro:excision} and~\labelcref{intro:skin:enum} are stronger, in the sense that they already hold formally, without any assumption about analyticity.

We conclude by remarking that there are a handful of cases in which topological recursion is known to admit a direct geometric meaning. A notable example is Mirzakhani's recursion for Weil--Petersson volumes and its non-orientable analogue \cite{Mir07,GGO}, obtained by integrating Mirzakhani--Mc\-Shane identities. Other examples of the same flavour include volumes and lattice-point counts on moduli spaces of metric ribbon graphs and their non-orientable analogues \cite{And+26,CGGO}, as well as recursions governing statistics of hyperbolic or combinatorial lengths of multicurves \cite{ABO,And+23,And+26}. Our situation is analogous, and in fact the parallel is stronger than a mere formal resemblance. In Mirzakhani's work, the recursion is obtained by excising pairs of pants determined by geodesic arcs emanating from the first boundary. In our setting, the re-rooting process underlying skinning determines a path emanating from the first boundary, and the homotopy class of the excised pair of pants is encoded by this path. As we explain in \Cref{sec:pop}, this makes our decomposition a combinatorial analogue of the Mirzakhani--McShane identities.

\subsection{Organisation of the paper}
For pedagogical reasons, we proceed by increasing complexity. We first treat maps in \Cref{sec:ord:maps}, then maps with tubes in \Cref{sec:maps:tubes}, and finally stuffed maps in full generality in \Cref{sec:stuffed:maps}. In each case, we first develop the skinning procedure. This removes the skin adjacent to the first boundary and leaves a core map, but it stops just before the change of topology. In \Cref{sec:pop}, we apply Tutte's procedure once more to the core. This additional step reveals the excised pair of pants and, in the stuffed case, the non-excisable configurations that produce the blob term.

The progressive treatment of maps, maps with tubes, and stuffed maps makes visible the universal structure underlying the excision formula and skin enumeration of Theorems~\labelcref{intro:excision} and~\labelcref{intro:skin:enum}. Finally, in \Cref{sec:analytic}, under suitable analytic assumptions, we compare our formulae with the blobbed topological recursion and recover the usual residue expressions on the associated spectral curve.

\subsection{Conventions for formal series}
\label{ssec:notation}
If $F(x)=\sum_{\ell\in\mathbb{Z}} F_\ell x^\ell$ is a formal bilateral series, we set
\begin{equation}
	[F(x)]_{\pm,x} \coloneqq \sum_{\pm \ell > 0} F_\ell x^\ell,
	\qquad
	[F(x)]_{0,x} \coloneqq F_0,
	\qquad
	\langle F(x) \rangle_x \coloneqq F_{-1}
\end{equation}
for the positive/negative part, the constant term, and the formal residue, respectively. More generally, for a formal series $F(x_1,\ldots,x_n)$ and a decomposition $I \sqcup J = \set{1,\ldots,n}$, we write $\langle F(x_1,\ldots,x_n)\rangle_{\bm{x}_I}$ for the series in the variables $\bm{x}_J$ obtained by extracting the coefficient of $x_i^{-1}$ for each $i \in I$.

Finally, we adopt the formal expansions
\begin{equation}
\label{x1:x2:expns}
	\frac{1}{x_1-x_2}
	\coloneqq
	\sum_{k\geq 0}\frac{x_2^k}{x_1^{k+1}},
	\qquad\qquad
	\frac{1}{(x_1-x_2)^2}
	\coloneqq
	\sum_{k\geq 1} k \frac{x_2^{k-1}}{x_1^{k+1}}.
\end{equation}
The order of the variables matters: by convention, the leftmost variable in the denominator is expanded first near $\infty$.

\subsection*{Acknowledgments}
This work was initiated when G.B. and A.G. were affiliated with the Max-Planck-Institut für Mathematik, Bonn, and benefited at this time from the support of the Max-Planck-Ge\-sell\-schaft. A.G. was supported by an ETH Fellowship (22-2 FEL-003) and a Hermann-Weyl-Instructorship from the Forschungsinstitut für Mathematik at ETH Zürich. L.M. was supported by the Deutsche Forschungsgemeinschaft (DFG, German Research Foundation) -- Projektnummer 417533893/GRK2575 ``Rethinking Quantum Field Theory''.

\section{Bijective functional relations for maps}
\label{sec:ord:maps}

This section is devoted to the case of maps, which already contains the main combinatorial ideas in their simplest form. We revisit Tutte's decomposition, define the skinning procedure, and derive the resulting bijective functional relations. These will serve as the model for the more general constructions developed in the next sections.

\subsection{Definition and generating series}
We begin by recalling the definition of maps and fixing the notation used throughout this section. A \emph{map} is an isotopy class of proper embeddings of a non-empty finite graph into a compact connected oriented surface such that the complement of the embedded graph is a disjoint union of topological discs. We consider maps with $n \geq 1$ labelled distinguished faces, called \emph{boundaries} or \emph{external faces}, denoted $\partial_1,\ldots,\partial_n$; the remaining faces are called \emph{internal faces}. Each \emph{half-edge} belongs to a unique face, and we orient it so that this face lies to its left with respect to the orientation of the surface. If $\epsilon$ is a half-edge, we denote by $\overline{\epsilon}$ the half-edge with the opposite orientation; an \emph{edge} is an unordered pair $e=\set{\epsilon,\overline{\epsilon}}$. We denote by $\epsilon^+$ and $\epsilon^-$ the half-edges immediately after and before $\epsilon$, respectively, along the orientation of the face to which it belongs. Unless specified otherwise, we assume that the $i$-th boundary carries a distinguished half-edge $\rho_i$, called the \emph{$i$-th root}. The edge supporting $\rho_i$ is denoted by $r_i$.

The topology of a map $\mathfrak{m}$ is encoded by the pair $(g,n)$, where $g$ is the genus and $n$ is the number of boundaries. Its \emph{Euler characteristic} is
\begin{equation}
\label{Euler}
	\chi(\mathfrak{m})
	=
	\card{\set{ \text{vertices} }} - \card{\set{ \text{edges} }} + \card{\set{ \text{internal faces}}}
	=
	2 - 2g - n.
\end{equation}
In particular, discs, annuli, and pairs of pants have respective topologies $(0,1)$, $(0,2)$, and $(0,3)$.
 
The degree of a face is the number of half-edges it contains. For fixed topology $(g,n)$ and boundary degrees $\ell_1,\ldots,\ell_n$, let $\MM_{g;\ell_1,\ldots,\ell_n}$ denote the set of maps whose boundaries have respective degrees $\ell_1,\ldots,\ell_n$. All boundary faces have positive degree, except for the map consisting of a single vertex on a sphere, which we regard as having one boundary of degree $0$ and no internal face; we call it the \emph{trivial map}. Thus $\MM_{0;0}$ is a singleton. In all other cases the sets $\MM_{g;\ell_1,\ldots,\ell_n}$ are declared to be empty. For a map $\mathfrak{m}$ we set
\begin{equation}
	v(\mathfrak{m}) \coloneqq \card{\set{ \text{vertices} }},
	\qquad
	N_k(\mathfrak{m}) \coloneqq \card{\set{ \text{internal faces of degree } k }}.
\end{equation}
The generating series of maps of topology $(g,n)$ with fixed boundary degrees is then defined by
\begin{equation}
	W_{g;\ell_1,\ldots,\ell_n}
	\coloneqq
	\sum_{\mathfrak{m} \in \MM_{g;\ell_1,\ldots,\ell_n}}
	q^{v(\mathfrak{m})}
	\prod_{k \geq 1} t_k^{N_k(\mathfrak{m})},
\end{equation}
where $q$ and $t_1,t_2,\ldots$ play the role of Boltzmann weights. The corresponding generating series with free boundary degrees is defined by
\begin{equation}
\label{Wgn:ord:maps}
	W_{g,n}(x_1,\ldots,x_n)
	\coloneqq
	\delta_{g,0}\delta_{n,1} \frac{q}{x_1}
	+
	\sum_{\ell_1,\ldots,\ell_n \geq 1}
	\frac{W_{g;\ell_1,\ldots,\ell_n}}{x_1^{\ell_1 + 1} \cdots x_n^{\ell_n + 1}}.
\end{equation}
The latter is a formal power series in $q,t_1,t_2$ whose coefficients are polynomials in $t_3,t_4,\ldots$ (omitted from the notation) and $x_1^{-1},\ldots,x_n^{-1}$:
\begin{equation}
	W_{g,n} \in \mathbb{Z}\bigl[x_1^{-1},\ldots,x_n^{-1},t_3,t_4,\ldots\bigr]\bbraket{q,t_1,t_2}.
\end{equation}
This polynomiality follows from the fact that, for fixed topology, boundary degrees, number of vertices, and numbers of $1$- and $2$-gons, there are only finitely many maps; this follows from a simple edge-counting argument and the Euler relation \eqref{Euler}. We also stress that our generating series are not weighted by automorphism factors: the presence of labelled rooted boundaries removes all automorphisms. With the notation of \Cref{ssec:notation}, we have
\begin{equation}
	W_{g;\ell_1,\ldots,\ell_n}
	=
	\left\langle W_{g,n}(x_1,\ldots,x_n) \prod_{i=1}^n x_i^{\ell_i} \right\rangle_{x_1,\ldots,x_n}.
\end{equation}
We will also need the \emph{disc potential}, which collects the weights $(t_k)_{k \ge 1}$ of the internal faces:
\begin{equation}
	T(x) \coloneqq \sum_{k \geq 1} \frac{t_k}{k} x^k.
\end{equation}

\subsection{Tutte's procedure}
\label{ssec:Tutte}
In \cite{Tut63}, Tutte described a bijection yielding a recursive construction of maps with disc topology. Its extension to arbitrary topology appears in \cite{WL72} and \cite{Eyn11,Eyn16}, where the corresponding relation for generating series is identified with the Dyson--Schwinger equation of formal Hermitian matrix models. Although this bijection and the resulting functional relations are well known, we recall them here in order to fix notations and prepare for the iterative construction that follows. Since our main interest is in higher topologies, we assume $(g,n)\neq(0,1)$.

In the following, we set $I=\set{2,\ldots,n}$, and for every subset $J\subseteq I$ we write $\bm{\ell}_J=(\ell_j)_{j\in J}$. For $i\in I$ we also set $I_i=\set{2,\ldots,\widehat{i},\ldots,n}$.

We now recall Tutte's procedure. Given a map $\mathfrak{m}$ of topology $(g,n)$, it consists in erasing the edge\footnote{
	After erasing the edge, we must reroot the new boundary. Provided that the choice is Markovian, the rerooting rule is somewhat arbitrary: it affects the resulting bijection, but not its properties. We choose to ``turn left'' whenever possible.
} $r_1$ supporting the root $\rho_1$. There are three possible geometric situations:
\begin{itemize}
	\item[(I)] $r_1$ bounds $\partial_1$ and an internal face;

	\item[(R)] $r_1$ bounds $\partial_1$ and another boundary;

	\item[(D)] $r_1$ bounds $\partial_1$ on both sides.
\end{itemize}
We call these cases internal-face removal, boundary reduction, and degeneration, respectively. In view of the skinning construction, we also distinguish them according to whether they preserve or change topology: case (I) is topology-preserving, case (R) is topology-changing, while case (D) splits into a topology-preserving case (D${}^{=}$), where a disc is separated off, and a topology-changing case (D${}^{>}$), where no disc component is produced. We now describe these cases in detail.

(I) If $r_1$ bounds an internal face of degree $k$, erasing it is equivalent to removing this internal face, and yields a map of genus $g$ with $n$ boundaries of respective degrees $\ell_1+k-2,\bm{\ell}_I$, see~\Cref{fig:TutteI}. We root the new first boundary at $\overline{\rho_1}{}^-$ if $k \geq 2$, or $\rho_1^+$ if $k = 1$ and $\ell_1 \geq 2$. If $k = \ell_1 = 1$, the resulting map is a trivial disc and does not need a root.

\begin{figure}[H]
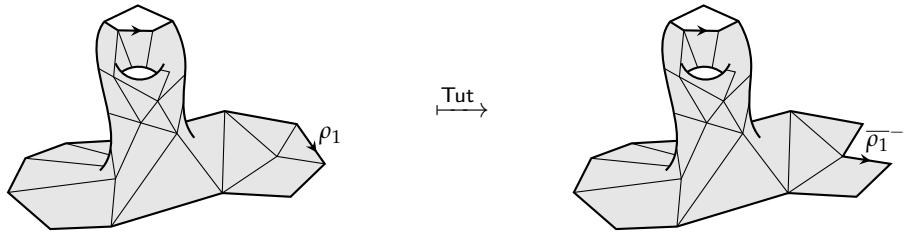

	\centering

	\caption{Case (I): the first root separates $\partial_1$ from an internal face.}
	\label{fig:TutteI}
\end{figure}

(R) If $r_1$ bounds $\partial_i$ for some $i\in\set{2,\ldots,n}$, erasing it merges the two boundary faces and yields a map of genus $g$ with $n-1$ boundaries of degrees $\ell_1+\ell_i-2,\bm{\ell}_{I_i}$, see~\Cref{fig:TutteR}. We root the first new boundary at $\overline{\rho_1}^-$ if $\ell_i \geq 2$. The case $\ell_i = 1$ can only occur if $(g,n) = (0,2)$ and $\ell_1 = 1$. In that case, the resulting map is trivial and does not need a root. We also remember the number $j_i\in\set{0,\ldots,\ell_i-1}$ of half-edges between $\overline{\rho_1}$ and $\rho_i$ around $\partial_i$, following its orientation.
\begin{figure}[H]
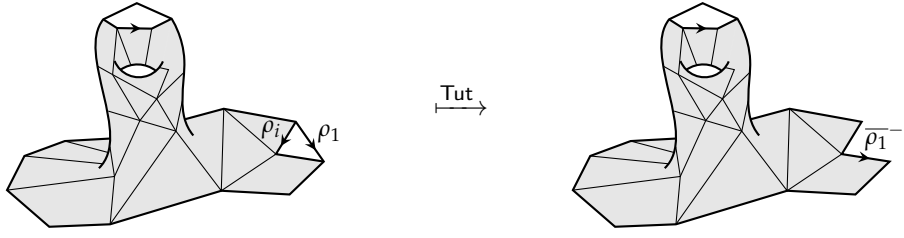

	\centering

	\caption{Case (R): the first root bounds $\partial_1$ and $\partial_i$.}
	\label{fig:TutteR}
\end{figure}

(D) If $r_1$ bounds $\partial_1$ on both sides, erasing it splits the first boundary into two new boundaries of respective degree $k'$ and $k''$ such that $k' + k'' = \ell_1 - 2$. We order them by saying that the one containing the endpoint of $\rho_1$ is first, and that the one containing the starting point of $\rho_1$ is second. We root these new boundaries at the half-edges $\overline{\rho_1}-$ and $\rho_1^-$ unless they coincide with $\overline{\rho_1}$. If $\overline{\rho_1}^-$ or $\rho_1^-$ coincides with $\overline{\rho_1}$, we regard the corresponding component of the resulting map as a trivial disc, and it does not need a root. To go further, we distinguish two cases, according to whether the resulting map has a disc component or not.

\begin{enumerate}
	\item[(D${}^{>}$)]
	No disc component is produced. This forces $k',k'' \geq 1$. There are two possibilities.
	
	\noindent
	\emph{Connected case.} If the result remains connected, then a handle connects the two sides, see~\Cref{fig:TutteDconn}. Thus, we obtain a map of genus $g-1$ with $n+1$ boundaries of respective degrees $k',k'',\bm{\ell}_I$.
	\begin{figure}[H]
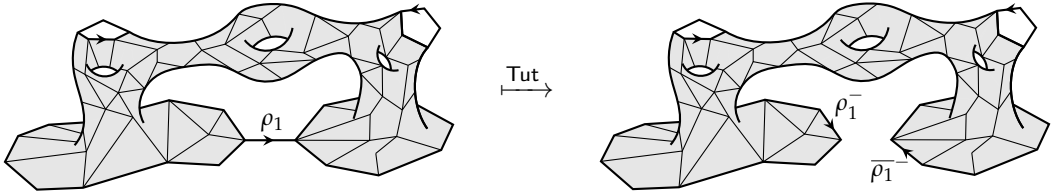

		\centering

		\caption{Case (D${}^{>}$), connected sub-case: the first root bounds $\partial_1$ on both sides, and removing it leaves the map connected.}
		\label{fig:TutteDconn}
	\end{figure}

	\noindent
	\emph{Disconnected case.} If the result is disconnected, then it consists of two maps, neither of which is a disc, see~\Cref{fig:TutteDdisc}. The first component contains $\overline{\rho_1}^-$, it has genus $h'$, and its remaining boundaries are $\bm{\partial}_{J'}$ with respective degrees $k',\bm{\ell}_{J'}$. The second component contains $\rho_1^-$, it has genus $h''$, and its remaining boundaries are $\bm{\partial}_{J''}$ with respective degrees $k'',\bm{\ell}_{J''}$. We have $h'+h''=g$ and $J'\sqcup J''=I$. Moreover, $(h',J')$ and $(h'',J'')$ are both different from $(0,\emptyset)$.	
 	
	\begin{figure}[H]
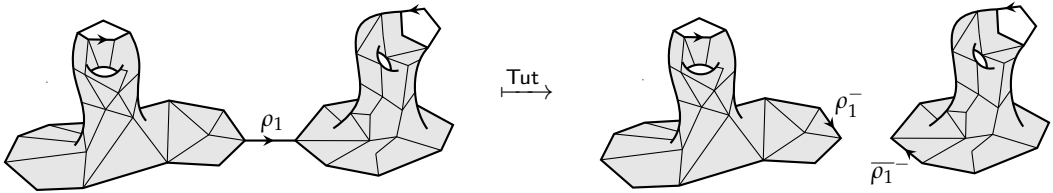

		\centering

		\caption{Case (D${}^{>}$), disconnected sub-case: $r_1$ bounds $\partial_1$ on both sides, and removing it disconnects the map into two maps, neither of which is a disc.}
		\label{fig:TutteDdisc}
	\end{figure}

	\item[(D${}^{=}$)]
	Similar to the disconnected case above, but one of the two resulting maps is a disc: $(h',J')=(0,\emptyset)$ or $(h'',J'')=(0,\emptyset)$, see~\Cref{fig:TutteD01}. The degree $k'$ or $k''$ is allowed to be $0$, meaning that the corresponding component is a trivial disc.
	\begin{figure}[H]
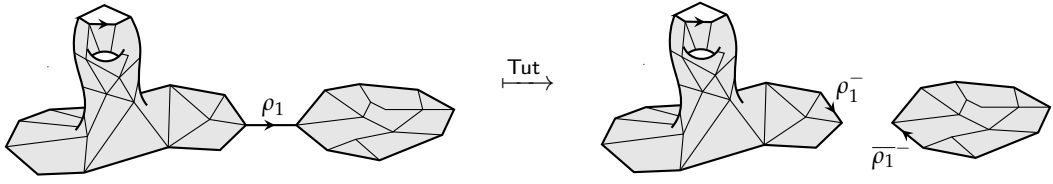

		\centering
		%
		\caption{Case (D${}^{=}$): $r_1$ is bounded by $\partial_1$ on both sides, and removing it disconnects the map into two maps, one of which is a disc.}
		\label{fig:TutteD01}
	\end{figure}
\end{enumerate}

The resulting maps, together with the auxiliary information given by the integer $j_i$ in case (R), uniquely determine the original map. We have therefore described a bijection
\begin{equation}
	\Tutte
	\colon
	\MM_{g;\ell_1,\bm{\ell}_I}
	\overset{\simeq}{\longrightarrow}
	\mathbb{I}_{g;\ell_1,\bm{\ell}_I}
	\sqcup
	\mathbb{R}_{g;\ell_1,\bm{\ell}_I}
	\sqcup
	\mathbb{D}_{g;\ell_1,\bm{\ell}_I}^{>}
	\sqcup
	\mathbb{D}^{=}_{g;\ell_1,\bm{\ell}_I}.
\end{equation}
The sets appearing on the right-hand side are given by
\begin{equation}
\begin{split}
	\mathbb{I}_{g;\ell_1,\bm{\ell}_I}
	& \coloneqq
		\bigsqcup_{k \geq 1} \MM_{g;\ell_1+k-2,\bm{\ell}_I}, \\
	\mathbb{R}_{g;\ell_1,\bm{\ell}_I}
	& \coloneqq
		\bigsqcup_{i=2}^n \MM_{g;\ell_1+\ell_i-2,\bm{\ell}_{I_i}} \times \set{0,\ldots,\ell_i-1}, \\
	\mathbb{D}_{g;\ell_1,\bm{\ell}_I}^{>}
	& \coloneqq
	\bigsqcup_{\substack{k',k'' \geq 1 \\ k'+k''=\ell_1-2}}
		\Biggl( \MM_{g-1;k',k'',\bm{\ell}_I}
			\sqcup
			\bigsqcup_{\substack{h'+h''=g, \\ J'\sqcup J'' = I}}^{\nodisc}
				\MM_{h';k',\bm{\ell}_{J'}} \times \MM_{h'';k'',\bm{\ell}_{J''}}
		\Biggr), \\
	\mathbb{D}_{g;\ell_1,\bm{\ell}_I}^{=}
	& \coloneqq
	\bigsqcup_{k=0}^{\ell_1-3}
	\bigl(
		\MM_{0;k} \times \MM_{g;\ell_1-k-2,\bm{\ell}_I}
	\bigr)
	\sqcup
	\bigl(
		\MM_{g;\ell_1-k-2,\bm{\ell}_I} \times \MM_{0;k}
	\bigr).
\end{split}
\end{equation}
Here ``no $(0,1)$'' means that we exclude from the disjoint union the cases in which $(h',1+\card{J'})$ or $(h'',1+\card{J''})$ is equal to $(0,1)$. We will use a similar notation for sums. If $(g,n)=(0,2)$, the case (D${}^{>}$) is absent. Translating this bijection into generating series yields the following relations.

\begin{lem}
\label{Tutte:lemma}
	Let $g\geq 0$, $n\geq 1$, and $\ell_1,\ldots,\ell_n\geq 1$, and assume $(g,n)\neq(0,1)$. Then
	\begin{equation}
	\begin{split}
		W_{g;\ell_1,\bm{\ell}_I}
		&=
		\sum_{k \geq 1} t_k \,W_{g;\ell_1+k-2,\bm{\ell}_I}
		+
		\sum_{i=2}^n \ell_i \, W_{g;\ell_1+\ell_i-2,\bm{\ell}_{I_i}} \\
		&\quad
		+
		\sum_{\substack{k',k'' \geq 1 \\ k'+k''=\ell_1-2}}
		\Biggl(
		W_{g-1;k',k'',\bm{\ell}_I}
		+
		\sum_{\substack{h'+h''=g \\ J'\sqcup J'' = I}}^{\nodisc}
		W_{h';k',\bm{\ell}_{J'}} W_{h'';k'',\bm{\ell}_{J''}}
		+ W_{0;k'}W_{g;k'',\bm{\ell}_I}
		+ W_{g;k',\bm{\ell}_I}W_{0;k''}
		\Biggr).
	\end{split}
	\end{equation}
	Equivalently, in terms of generating series with free boundary degrees, we have
	\begin{equation}
	\label{Tutte:rec}
	\begin{split}
		\mathcal{L}_{x_1} W_{g,n}(x_1,\bm{x}_I)
		&=
		\sum_{i=2}^n
		\partial_{x_i}\left(
		\frac{W_{g,n-1}(x_1,\bm{x}_{I_i}) - W_{g,n-1}(x_i,\bm{x}_{I_i})}{x_1-x_i}
		\right) \\
		&\qquad
		+
		W_{g-1,n+1}(x_1,x_1,\bm{x}_I)
		+
		\sum_{\substack{h'+h''=g \\ J'\sqcup J'' = I}}^{\nodisc}
		W_{h',1+\card{J'}}(x_1,\bm{x}_{J'})W_{h'',1+\card{J''}}(x_1,\bm{x}_{J''}).
	\end{split}
	\end{equation}
	Here $\frac{1}{x_1-x_i}$ is expanded as in~\eqref{x1:x2:expns}, and the operator $\mathcal{L}_{x}$, acting on a formal Laurent series $f(x)$, is defined by $\mathcal{L}_{x}f(x) \coloneqq \bigl[ (x-T'(x))f(x) \bigr]_{-,x} - 2\,W_{0,1}(x)f(x)$.
\end{lem}

\begin{proof}
	The relation for fixed boundary degrees is immediate from the description of the bijection $\Tutte$. Multiplying both sides by $x_1^{-\ell_1}x_2^{-\ell_2-1}\cdots x_n^{-\ell_n-1}$ and summing over $\ell_1,\ldots,\ell_n \geq 1$ (for any fixed monomial in $q,\bm{t}$ the sum is finite), we obtain
	\begin{multline}
	 \bigl[
			(x_1-T'(x_1))W_{g,n}(x_1,\bm{x}_I)
		\bigr]_{-,x_1}
		-
		2\,W_{0,1}(x_1)W_{g,n}(x_1,\bm{x}_I)
		=
		\sum_{i=2}^n
		\partial_{x_i}\left(
		\frac{W_{g,n-1}(x_1,\bm{x}_{I_i}) - W_{g,n-1}(x_i,\bm{x}_{I_i})}{x_1-x_i}
		\right) \\
		+
		W_{g-1,n+1}(x_1,x_1,\bm{x}_I)
		+
		\sum_{\substack{h'+h''=g \\ J'\sqcup J'' = I}}^{\nodisc}
		W_{h',1+\card{J'}}(x_1,\bm{x}_{J'})W_{h'',1+\card{J''}}(x_1,\bm{x}_{J''}),
	\end{multline}
	where $T'(x)=\sum_{k \geq 1} t_k x^{k-1}$. This is precisely \eqref{Tutte:rec}. 
\end{proof}

\subsection{Skinning maps}
\label{ssec:skin}
We call \emph{skinning} the process of iterating Tutte's procedure as long as the topology is preserved, and stopping just before the first topology-changing step. The resulting map still has topology $(g,n)$, but would change topology if Tutte's procedure were applied once more; we call it the \emph{core map}. The pieces removed along the way can be assembled into a map, which we call the \emph{skin map}. Gluing the core map and its skin back together recovers the original map, see~\Cref{fig:skin:core}. Since this gluing does not change topology and the Euler characteristic is additive under gluing, the skin map must have vanishing Euler characteristic, that is, it must be a topological annulus.

With the notations of \Cref{ssec:Tutte}, the topology-preserving cases are precisely (I) and (D${}^{=}$), while the topology-changing cases are (R) and (D${}^{>}$). Skinning consists in applying Tutte's procedure repeatedly as long as the outcome lies in one of the topology-preserving cases, and stopping just before the first step whose outcome would lie in one of the topology-changing cases. The map at which the process stops is the core map, and the pieces removed during the previous steps assemble into the skin map.

\begin{figure}
	\centering

	\caption{Decomposition of a map $\mathfrak{m}$ into a skin map $\mathfrak{s}$ a core map $\hat{\mathfrak{m}}$.}
	\label{fig:skin:core}
\end{figure}

Let us first make the notion of core map precise.

\begin{defn}
	Let $g \geq 0$, $n \geq 1$, and $\ell_1,\ldots,\ell_n \geq 1$, with $(g,n) \neq (0,1)$. We denote by $\hat{\MM}_{g;\ell_1,\bm{\ell}_I}$ the set of maps $\hat{\mathfrak{m}}$ of topology $(g,n)$ and boundary degrees $\ell_1,\bm{\ell}_I$ such that
	\begin{equation}
		\Tutte(\hat{\mathfrak{m}})
		\in
		\mathbb{R}_{g;\ell_1,\bm{\ell}_I}
		\sqcup
		\mathbb{D}_{g;\ell_1,\bm{\ell}_I}^{>}.
	\end{equation}
	We call these maps \emph{core maps} and introduce their generating series
	\begin{equation}
	\label{Wgn:core}
		\hat{W}_{g;\ell_1,\bm{\ell}_I}
		\coloneqq
		\sum_{\hat{\mathfrak{m}} \in \hat{\MM}_{g;\ell_1,\bm{\ell}_I}}
		q^{v(\hat{\mathfrak{m}})}
		\prod_{k \geq 1} t_{k}^{N_k(\hat{\mathfrak{m}})},
		\qquad
		\hat{W}_{g,n}(x_1,\ldots,x_n)
		\coloneqq
		\sum_{\ell_1,\ldots,\ell_n \geq 1}
		\frac{\hat{W}_{g;\ell_1,\ldots,\ell_n}}
		{x_1^{\ell_1}x_2^{\ell_2+1}\cdots x_n^{\ell_n+1}}.
	\end{equation}
\end{defn}

Notice that, unlike for maps, our convention on the generating series for core maps with free boundary degrees is that the first boundary appears with weight $x_1^{-\ell_1}$ instead of $x_1^{-\ell_1-1}$. This will make the gluing formulae simpler at the level of generating series.

By definition, we immediately obtain
\begin{equation}
	\hat{W}_{g;\ell_1,\bm{\ell}_I}
	=
	\sum_{i=2}^n \ell_i\, W_{g;\ell_1+\ell_i-2,\bm{\ell}_{I_i}}
	+
	\sum_{\substack{k',k'' \geq 1 \\ k'+k''= \ell_1-2}}
	\Biggl(
	W_{g-1;k',k'',\bm{\ell}_I}
	+
	\sum_{\substack{h'+h''=g \\ J'\sqcup J''=I}}^{\nodisc}
	W_{h';k',\bm{\ell}_{J'}}\,W_{h'';k'',\bm{\ell}_{J''}}
	\Biggr).
\end{equation}
In terms of generating series with free boundary degrees, the same relation becomes
\begin{equation}
\label{Wcore:rec}
\begin{split}
	\hat{W}_{g,n}(x_1,\bm{x}_I)
	&=
	\sum_{i=2}^n
	\partial_{x_i}
	\left(
		\frac{W_{g,n-1}(x_1,\bm{x}_{I_i}) - W_{g,n-1}(x_i,\bm{x}_{I_i})}{x_1-x_i}
	\right) \\
	&\quad
	+
	W_{g-1,n+1}(x_1,x_1,\bm{x}_I)
	+
	\sum_{\substack{h'+h''=g \\ J'\sqcup J''=I}}^{\nodisc}
	W_{h',1+\card{J'}}(x_1,\bm{x}_{J'})\,W_{h'',1+\card{J''}}(x_1,\bm{x}_{J''}).
\end{split}
\end{equation}

We now turn to the pieces removed during skinning. The key observation is that cutting according to Tutte's procedure, while recording what is cut off, is inverse to a suitable gluing operation.

Let us define precisely what we mean by gluing. Gluing two half-edges $\epsilon$ and $\epsilon'$ means identifying them with opposite orientations. If $\epsilon$ (respectively $\epsilon'$) was part of a larger map where it had an opposite half-edge, the latter would disappear after gluing and be replaced by $\epsilon'$ (respectively, $\epsilon$). If $\partial$ and $\partial'$ are two distinct boundaries (possibly belonging to the same map) of common degree $\ell$ and with roots $\rho$ and $\rho'$, the operation of gluing $\partial$ and $\partial'$ consists in gluing $\rho$ to $\rho'$, then for all $j \in \set{1,\ldots,\ell - 1}$ gluing the $j$-th half-edge after $\rho$ to the $j$-th half-edge before $\rho'$, and eventually forgetting the roots. A sufficient condition for this operation to produce again a map---in which $\partial$ and $\partial'$ merge to form a cycle---is that $\partial$ or $\partial'$ is simple. Recall that a face $\mathfrak{f}$ is called \emph{simple} if every vertex in $\mathfrak{f}$ is adjacent to at most two half-edges of $\mathfrak{f}$, see for instance \cite{BG17}. In particular, a face in which two half-edges are glued to each other is not simple.

Starting from a map $\mathfrak{m}$ of topology $(g,n)\neq(0,1)$, we erase $r_1$ and, referring to the cases of \Cref{ssec:Tutte}, we associate with the removed part an \emph{elementary skin map} $\mathfrak{e}$ determined by $\Tutte(\mathfrak{m})$. This elementary skin map has genus $0$ and at least two boundaries. We denote by $\partial_i\mathfrak{e}$ the $i$-th boundary of $\mathfrak{e}$. For $i\geq 2$, the boundary $\partial_i\mathfrak{e}$ is simple, while $\partial_1\mathfrak{e}$ may or may not be simple, depending on the case.

(I) $\mathfrak{e}$ is an annulus with one internal face, see~\Cref{fig:gluingI}. The boundary $\partial_1\mathfrak{e}$ has root $\rho_1$ and degree $\ell_1$, and happens to be simple. The boundary $\partial_2\mathfrak{e}$ has degree $\ell_1+k-2$ and is rooted at $\bbarsuperscript{\rho_1}{-}$ if $k \geq 2$, or $\barsuperscript{\rho_1}{+}$ if $k = 1$. The $\ell_1-1$ half-edges after $\rho_1$ are glued to the $\ell_1-1$ half-edges before the root of $\partial_2\mathfrak{e}$. The internal face has degree $k$; one of its half-edges is glued to $\rho_1$ and the remaining $k - 1$ half-edges before it are glued to the root of $\partial_2\mathfrak{e}$ and the $k - 2$ half-edges after it.
\begin{figure}[H]
	\centering

	\caption{Case (I): $\mathfrak{e}$ is an annulus with one internal face.}
	\label{fig:gluingI}
\end{figure}

(R) $\mathfrak{e}$ is a \emph{theta map}, that is, a pair of pants without internal faces in which the root-edge of the first boundary bounds the third boundary, see~\Cref{fig:gluingR}. This case is analogous to (I), except that the internal face is now marked and forms the third boundary face. The boundary $\partial_1\mathfrak{e}$ has root $\rho_1$ and degree $\ell_1$, and is simple. The boundary $\partial_2\mathfrak{e}$ has degree $\ell_1+\ell_i-2$ and root $\bbarsuperscript{\rho_1}{-}$ if $\ell_i \geq 2$. The boundary $\partial_3\mathfrak{e}$ has degree $\ell_i$ and root $\rho_i$. We obtain the map $\mathfrak{e}$ after the following gluings. We glue $\rho_1$ to the $(j_i + 1)$-th half edge before $\rho_i$. We glue the $\ell_1 - 1$ half-edges after $\rho_1$ to the $\ell_1 - 1$ half-edges before the root of $\partial_2\mathfrak{e}$. We glue the half-edges in position $\ell_i - j_i - 1,\ell_i - j_i - 2,\ldots,-j_i$ compared to $\rho_i$ to the root of $\partial_2\mathfrak{e}$ and the $\ell_i - 1$ half-edges after it.
\begin{figure}[H]
	\centering

	\caption{Case (R): $\mathfrak{e}$ is a theta map.}
	\label{fig:gluingR}
\end{figure}

(D${}^{>}$) $\mathfrak{e}$ is a pair of glasses, that is, a pair of pants without internal faces such that a root-edge bounds the first boundary on both sides, see~\Cref{fig:gluingC1}. The boundary $\partial_1\mathfrak{e}$ has root $\rho_1$ and degree $\ell_1$, and is not simple. The boundary $\partial_2\mathfrak{e}$ has degree $k'$, root $\bbarsuperscript{\rho_1}{-}$, and all its half-edges are glued to $\partial_1\mathfrak{e}$, starting with the root which is glued to $\overline{\rho_1}^-$. Lastly, $\partial_3\mathfrak{e}$ has  degree $k''$, root $\barsuperscript{\rho_1}{-}$, and all its half-edges are glued to $\partial_1\mathfrak{e}$, starting with its root which is glued to $\rho_1^-$. The connected and disconnected sub-cases of (D${}^{>}$) give the same elementary skin map, but leave different complements.
\begin{figure}[H]
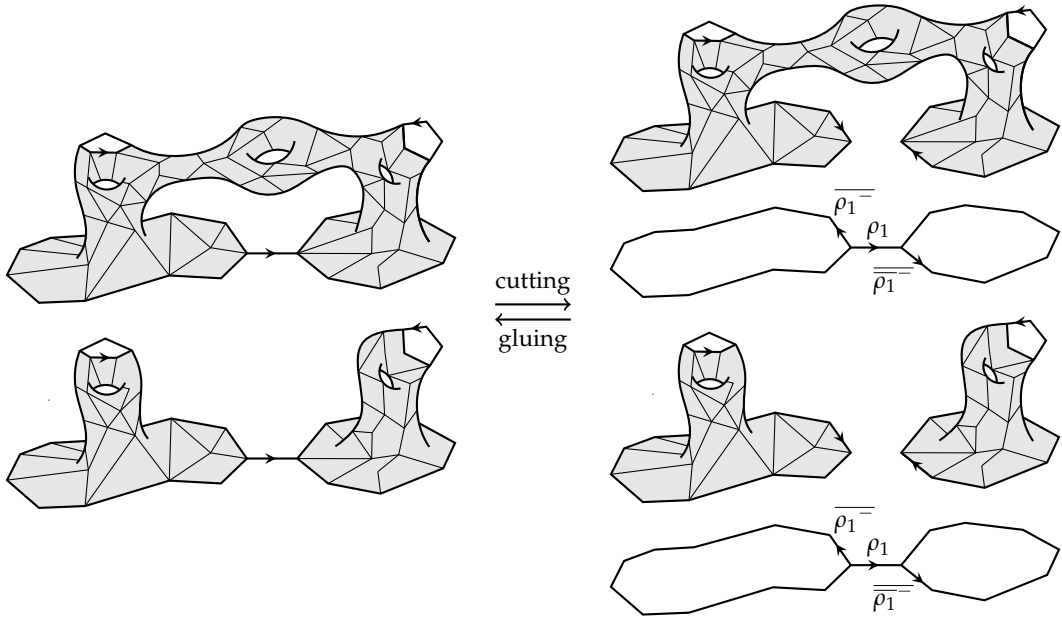

	\centering

	\caption{\label{fig:gluingC1} Case (D${}^{>}$), connected and disconnected sub-cases, top and bottom, respectively: $\mathfrak{e}$ is a pair of glasses.}
\end{figure}

(D${}^{=}$) From the point of view of Tutte's procedure, this is similar to the disconnected sub-case of (D${}^{>}$), with one of the components of  $\Tutte(\mathfrak{m})=(\mathfrak{m}',\mathfrak{m}'')$ being a disc; they cannot both be discs, since $(g,n)\neq(0,1)$. However, we treat it differently, as we wish $\mathfrak{e}$ to include this disc, see~\Cref{fig:gluingC2}. And there is a novelty, as this disc could be trivial. More precisely, $\mathfrak{e}$ is an annulus constructed as follows. If $\mathfrak{m}'$ or $\mathfrak{m}''$ is a trivial disc, $\mathfrak{e}$ is an annulus without internal faces. Its first boundary has degree $\ell_1$, is rooted at $\rho_1$, and $\rho_1$ is glued to $\overline{\rho_1}$. Its second boundary has degree $\ell_1 - 2$ and is rooted at $\bbarsuperscript{\rho_1}{-}$ if $\mathfrak{m}'$ is the trivial disc, or at $\barsuperscript{\rho_1}{-}$ if $\mathfrak{m}''$ is the trivial disc. If $\mathfrak{m}'$ and $\mathfrak{m}''$ are not trivial discs, we first construct an auxiliary pair of glasses $\mathfrak{a}$ defined like in (D${}^{>}$). Then, if $\mathfrak{m}'$ (respectively, $\mathfrak{m}''$) is a disc we glue its boundary face to $\partial_2\mathfrak{a}$ (respectively to $\partial_3\mathfrak{a}$). The result is our elementary skin $\mathfrak{e}$.

\begin{figure}[H]
	\centering

	\caption{\label{fig:gluingC2} Case (D${}^{=}$): $\mathfrak{e}$ is an annulus.}
\end{figure}

For the skinning process itself, we only need to keep track of the elementary skin maps arising from the topology-preserving cases. Indeed, if the outcome of Tutte's procedure lies in one of the topology-changing cases (R) or (D${}^{>}$), then the original map is already a core map and the process stops. By contrast, in the topology-preserving cases (I) and (D${}^{=}$), the elementary skin map is part of the skin being built, and it is an annulus.

\begin{defn}
	Let $\ell_1,\ell \geq 1$. We denote by $\mathbb{E}_{\ell_1,\ell}$ the set of elementary skin maps $\mathfrak{e}$ with boundary degrees $\ell_1$ and $\ell$ arising from the topology-preserving cases (I) and (D${}^{=}$). We define their generating series by
	\begin{equation}
	\label{eskin}
		E_{\ell_1,\ell}
		\coloneqq
		\sum_{\mathfrak{e} \in \mathbb{E}_{\ell_1,\ell}} q^{v(\mathfrak{e}) - \ell} \prod_{k \geq 1} t_k^{N_k(\mathfrak{e})}.
	\end{equation}
	Note that we do not count the vertices of $\partial_2\mathfrak{e}$ in the weight. Indeed, simple boundaries have as many vertices as their degree, here $\ell$, and $\partial_2\mathfrak{e}$ is meant to be glued to the boundary of another map; omitting to count these vertices prevents overcounting after the identification. 
\end{defn}

By definition, we have
\begin{equation}
\label{eskin:01}
	E_{\ell_1,\ell} = t_{\ell - \ell_1 + 2} + 2W_{0;\ell_1 - \ell - 2}.
\end{equation}
Here and below we use the convention $W_{0;k}=0$ for $k<0$, and $W_{0;0} = q$ counts the trivial skin map. The first term in~\eqref{eskin:01} comes from case (I), while the second comes from case (D${}^{=}$); the factor $2$ records the two possible orientations for the root, pointing towards the disc component or opposite to it.

\begin{defn}
\label{def:skin}
	A \emph{skin map} is an annulus that either has two simple boundaries of the same positive degree, rooted at the same edge, and no internal faces (\emph{trivial skin}, see~\Cref{fig:trivial:skin}), or is obtained from a sequence of $a \geq 1$ elementary skin maps $\mathfrak{e}_1,\ldots,\mathfrak{e}_{a}$ such that $\mathfrak{e}_j \in \mathbb{E}_{\ell_j,\ell_{j + 1}}$ for some $\ell_1,\ldots,\ell_{a + 1} \geq 1$, by gluing $\partial_2\mathfrak{e}_j$ to $\partial_1\mathfrak{e}_{j + 1}$ for every $j \in \set{1,\ldots,a - 1}$. We denote by $\mathbb{S}_{\ell_1,\ell}$ the set of skin maps with boundary degrees $\ell_1,\ell$, and define their generating series by
	\begin{equation}
	\label{skin:maps}
		S_{\ell_1,\ell}
		\coloneqq
		\sum_{\mathfrak{s} \in \mathbb{S}_{\ell_1,\ell}}
			q^{v(\mathfrak{s}) - \ell}
			\prod_{k \geq 1} t_k^{N_k(\mathfrak{s})},
		\qquad
		S(x_1,x)
		\coloneqq
		\sum_{\ell_1,\ell \geq 1} S_{\ell_1,\ell} \frac{x^{\ell - 1}}{x_1^{\ell_1 + 1}}.
	\end{equation}
\end{defn}

As in the generating series of elementary skin maps, we do not count the vertices of the second boundary component in the generating series of skin maps. We also stress the special way in which the degree of the second boundary is incorporated, namely with $x^{\ell - 1}$ instead of $x^{-\ell - 1}$. There is no negative power of $x$ in $S(x_1,x)$ because the degree $\ell$ is positive. This convention matches the one adopted in \cite{BG17} for simple boundaries, and it will make the gluing formulae in \Cref{ssec:skinning} cleaner.

As we will see in the next section, the significance of the trivial skin is that it acts as the identity for gluing: if the original map is already a core map, we can still regard it as glued to the trivial skin.

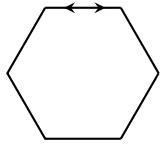
\begin{figure}[H]
	\centering
	\begin{tikzpicture}[baseline=-.5ex]
		\draw[thick] (0:1) -- (60:1);
		\draw[thick,postaction={decorate, decoration={markings, mark= at position 0.75 with {\arrow{stealth}}}},postaction={decorate, decoration={markings, mark= at position 0.35 with {\arrow{stealth reversed}}}}]
		(60:1) -- (120:1);
		\draw[thick] (120:1) -- (180:1) -- (240:1) -- (300:1);
		\draw[thick] (-60:1) -- (0:1);
	\end{tikzpicture}
	\caption{The trivial skin of degree $6$.}
	\label{fig:trivial:skin}
\end{figure}

We conclude by stressing that, although it was convenient to introduce elementary skin maps and skin maps by cutting along Tutte's procedure, these notions stand on their own and do not depend on the topology $(g,n)$ one started with. More precisely, using the fact that gluing is inverse to cutting, one could equivalently define skin maps as those annular maps whose second boundary is simple and that yield a trivial skin after finitely many applications of Tutte's procedure.

\subsection{Recursion for skin maps}
\label{ssec:rec:skin}
Our next goal is to compute the generating series of skin maps $S(x_1,x)$. It follows from \Cref{def:skin} that
\begin{equation}
\label{skin:elementary}
	S_{\ell_1,\ell}
	=
	\delta_{\ell_1,\ell}
	+
	\sum_{a \geq 1}
	\sum_{\substack{\ell_2,\ldots,\ell_{a} \geq 1 \\ \ell_{a + 1} = \ell}}
	\prod_{j = 1}^{a} E_{\ell_{j},\ell_{j + 1}}.
\end{equation}
The first term corresponds to the trivial skin and may be viewed as the $a=0$ term of the sum. An equivalent way to encode this formula is to introduce the transfer operators $\mathcal{E}$ and $\mathcal{S}$ acting on sequences $(f_{\ell})_{\ell \geq 1}$ by
\begin{equation}
	\mathcal{S}[f]_{\ell_1} \coloneqq \sum_{\ell \geq 1} S_{\ell_1,\ell} f_{\ell},
	\qquad\qquad
	\mathcal{E}[f]_{\ell_1} \coloneqq \sum_{\ell \geq 1} E_{\ell_1,\ell} f_{\ell}.
\end{equation}
Then \eqref{skin:elementary} becomes $\mathcal{S} = \sum_{a \geq 0} \mathcal{E}^{a} = (\textnormal{Id}-\mathcal{E})^{-1}$.
In particular, we have
\begin{equation}
\label{transfer:id}
	\mathcal{S}
	=
	\textnormal{Id} + \mathcal{E}\circ \mathcal{S},
	\qquad\qquad
	\mathcal{S}
	=
	\textnormal{Id} + \mathcal{S}\circ \mathcal{E}.
\end{equation}
Written in terms of matrix elements, these two identities yield relations for the generating series.

The first one in~\eqref{transfer:id} reads
\begin{equation}
\label{rec:skin1}
	S_{\ell_1,\ell}
	=
	\delta_{\ell_1,\ell}
	+
	\sum_{\ell' \geq 1} E_{\ell_1,\ell'} S_{\ell',\ell}
	=
	\delta_{\ell_1,\ell}
	+
	\sum_{k \geq 1} t_k\,S_{\ell_1 + k - 2,\ell}
	+
	2 \sum_{k = 0}^{\ell_1 - 3} W_{0;k}\,S_{\ell_1 - k - 2,\ell}.
\end{equation}
where in the second equality we used \eqref{eskin} and changed variables to $k=\ell'-\ell_1+2$ in the first sum and to $k=\ell_1-\ell'-2$ in the second. Since $W_{0;k}=O(q)$, this equation uniquely determines $S_{\ell_1,\ell}$ as an element in $\mathbb{Q}[t_3,t_4,\ldots]\bbraket{q,t_1,t_2}$. Combinatorially, the first equation of \eqref{rec:skin1} states that a skin map is either the trivial skin or is obtained by gluing the second boundary of an elementary skin map arising from a topology-preserving Tutte step to the first boundary of another skin map.

The second equality in~\eqref{transfer:id} gives, in terms of generating series,
\begin{equation}
	S_{\ell_1,\ell}
	=
	\delta_{\ell_1,\ell}
	+
	\sum_{\ell' \geq 1} S_{\ell_1,\ell'} E_{\ell',\ell}
	=
	\delta_{\ell_1,\ell}
	+
	\sum_{k = 1}^{\ell+1} t_k\,S_{\ell_1,\ell-k+2}
	+
	2\sum_{k\geq 0} W_{0;k}\,S_{\ell_1,\ell+k+2}.
\end{equation}
This equation also uniquely determines $S_{\ell_1,\ell}\in \mathbb{Q}[t_3,t_4,\ldots]\bbraket{q,t_1,t_2}$. Combinatorially, it means that a skin map is either the trivial skin or is obtained by gluing the first boundary of an elementary skin map to the second boundary of a skin map. In terms of generating series, the above equation reads
\begin{equation}
\label{rec:skin2}
	\left[
		S(x_1,x) \bigl(x- \partial_x T(x)-2W_{0,1}(x)\bigr)
	\right]_{+,x}
	=
	\frac{1}{x_1-x} - \frac{1}{x_1},
\end{equation}
where $\frac{1}{x_1-x}$ is expanded in non-negative powers of $x$, and subtracting $\frac{1}{x_1}$ extracts its positive part.

\subsection{Skinning formula and skin enumeration}
\label{ssec:skinning}
To recover an arbitrary map $\mathfrak{m}$ of topology $(g,n)\neq(0,1)$ and boundary degrees $\ell_1,\dots,\ell_n$, we may choose an arbitrary skin map $\mathfrak{s} \in \mathbb{S}_{\ell_1,\ell}$, an arbitrary core map $\hat{\mathfrak{m}} \in \hat{\MM}_{g;\ell,\bm{\ell}_I}$, and glue the second boundary of $\mathfrak{s}$ to the first boundary of $\hat{\mathfrak{m}}$. By construction, this defines a bijection
\begin{equation}
	\begin{array}{ccc}
		\MM_{g;\ell_1,\bm{\ell}_I} & \overset{\simeq}{\longrightarrow} & \bigsqcup_{\ell \geq 1} \mathbb{S}_{\ell_1,\ell} \times \hat{\MM}_{g;\ell,\bm{\ell}_I} \\
		\mathfrak{m} & \longmapsto & (\mathfrak{s},\hat{\mathfrak{m}}).
	\end{array}
\end{equation}
Moreover, the number of vertices of $\mathfrak{m}$ is $v(\mathfrak{m}) = v(\hat{\mathfrak{m}}) + v(\mathfrak{s}) - \ell$. Therefore, the generating series for fixed boundary degrees satisfy the skinning relation
\begin{equation}
\label{skinning:rel}
	W_{g;\ell_1,\bm{\ell}_I} = \sum_{\ell \ge 1} S_{\ell_1,\ell} \hat{W}_{g;\ell,\bm{\ell}_I}.
\end{equation}
Multiplying by $x_1^{-\ell_1-1} \cdots x_n^{-\ell_n-1}$ and summing over $\ell_1,\ldots,\ell_n \geq 1$, we obtain
\begin{equation}
	W_{g,n}(x_1,\bm{x}_I) = \big\langle S(x_1,x)\hat{W}_{g,n}(x,\bm{x}_I) \big\rangle_x.
\end{equation}
Here the coefficient extraction is taken term by term in the formal Laurent expansion in $x$. The conventions for the degree-counting variable of the second boundary of a skin map and of the first boundary of a core map were chosen precisely so that extracting the coefficient of $x^{-1}$ matches equal degrees on the two boundaries being glued.

Using \eqref{Wcore:rec} to express $\hat{W}_{g,n}$ in terms of generating series for maps of simpler topology, i.e. with higher Euler characteristic, we arrive at the following recursion.

\begin{prop}[Skinning formula for maps]
\label{skinning:maps}
	Let $g \geq 0$, $n \geq 1$ be such that $(g,n) \neq (0,1)$. Then
	\begin{multline}
	\label{preTR}
		W_{g,n}(x_1,\bm{x}_I)
		=
		\Bigg\langle
		S(x_1,x)
		\Bigg(
			\sum_{i=2}^n
				\partial_{x_i} \left(
					\frac{W_{g,n-1}(x,\bm{x}_{I_i}) - W_{g,n-1}(x_i,\bm{x}_{I_i})}{x - x_i}
				\right) \\
		+
		W_{g-1,n+1}(x,x,\bm{x}_I)
		+
		\sum_{\substack{h' + h'' = g \\ J' \sqcup J'' = I}}^{\nodisc}
			W_{h',1+\card{J'}}(x,\bm{x}_{J'})W_{h'',1+\card{J''}}(x,\bm{x}_{J''})
		\Bigg)\Bigg\rangle_x .
	\end{multline}
\end{prop}

The disc case is exceptional: applying Tutte's procedure to a disc never produces a genuine core map of smaller complexity. Nevertheless, there is an analogue of the skinning formula, obtained by marking a vertex at which the procedure stops. More precisely, let $\dot{\MM}_{0;\ell_1}$ be the set of \emph{pointed discs}, that is discs with boundary degree $\ell_1$ and one marked vertex. Assigning vertex-weight $q^{v(\mathfrak{m})-1}$, that is, excluding the marked vertex from the counting, we find that the generating series of pointed discs is $\partial_q W_{0;\ell_1}$.

Given $\mathfrak{m}\in\dot{\MM}_{0;\ell_1}$, we apply Tutte's procedure erasing $r_1$ exactly as before, except that the marked vertex now plays the role of the boundary at which the process stops. Case (I) is unchanged, but we need to handle with care\footnote{
	This did not occur previously since we had $(g,n)\neq(0,1)$.
} the situation where $\ell_1=1$ and there is a single internal face of degree $1$. We call $\circ$ this map: as $r_1$ separates off the marked vertex, we declare it to be a stopping case. If the marked vertex is univalent and adjacent to $r_1$, then erasing the latter  separates off the marked vertex and leaves an unpointed disc of degree $\ell_1-2$, see~\Cref{fig:Rpointed}. This is the analogue of case (R), with the marked vertex replacing the second boundary, and we declare it to be a stopping case.

\begin{figure}[H]
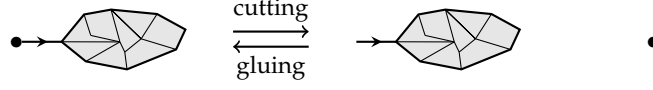

	\centering

	\caption{Case (R) for pointed discs.}
	\label{fig:Rpointed}
\end{figure}

The remaining unpointed disc may lie on either side of $r_1$. If $ r_1= \set{\rho_1,\overline{\rho_1}}$ and is not adjacent to the marked vertex, then we are in the analogue of case (D${}^{=}$): after erasing $r_1$, one component contains the marked vertex and the other does not, and we include the latter in the elementary skin map, see~\Cref{fig:Dpointed}. The analogue of case (D${}^{>}$) is absent, since a pointed disc cannot produce a topology-changing degeneration before the marked vertex is separated.

\begin{figure}[H]
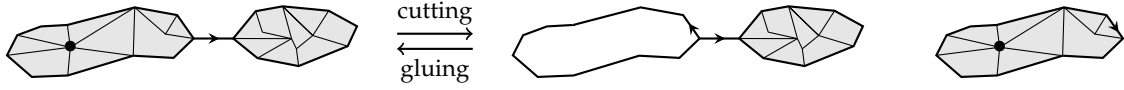

	\centering

	\caption{Case (D${}^{=}$) for pointed discs.}
	\label{fig:Dpointed}
\end{figure}

Iterating the procedure until the stopping case occurs, we obtain: on the one hand, a skin map in $\mathbb{S}_{\ell_1,\ell}$; on the other hand, either the disc $\circ$ (this forces $\ell = 1$) or an (unpointed) disc of degree $\ell-2$ which can lie on either side of the root-edge. This results in the bijection, for $\ell_1 \geq 1$,
\begin{equation}
	\dot{\MM}_{0;\ell_1}
	\overset{\simeq}{\longrightarrow}
	\mathbb{S}_{\ell_1,1} \times \set{\circ} \,\,\sqcup \,\,\bigsqcup_{\ell \geq 1} \mathbb{S}_{\ell_1,\ell} \times (\MM_{0;\ell - 2} \sqcup \MM_{0;\ell - 2}).
	\end{equation}
It is important to observe that the recursive definition of $\mathbb{S}_{\ell_1,\ell}$ is identical to that in \Cref{ssec:skin}; in particular, the elementary skin maps are counted by \eqref{eskin:01}. Therefore, the skinning relation for pointed discs takes the form
\begin{equation}
	\partial_q W_{0;\ell_1} = \sum_{\ell \geq 1} S_{\ell_1,\ell}(\delta_{\ell,1}t_1 + 2W_{0;\ell - 2})
\end{equation}
In terms of generating series,
\begin{equation}
\label{skinning:discs}
	\partial_q W_{0,1}(x_1)-\frac{1}{x_1}
	=
	\left[ S(x_1,x)(T'(x) + 2W_{0,1}(x))\right]_{0,x},
\end{equation}
where the subtraction of $\frac{1}{x_1}$ comes from the trivial disc term $\partial_q W_{0;0} = 1$, and $T'(0) = t_1$ accounts for the disc $\circ$. This relation can also be proved by applying the pointing operator $\partial_q$ to Tutte's recursion for discs and using the first recursion \eqref{rec:skin1} for $S$.

The pointed-disc relation above provides the remaining ingredient needed to complete the computation of the skin generating series $S(x_1,x)$. This series will play the role of the recursion kernel in the topological recursion, and we would like to express it in terms of generating series of maps.

\begin{prop}[Skin enumeration for maps]
\label{skin:enum:maps}
	We have
	\begin{equation}
		S(x_1,x)
		=
		\frac{
			{\displaystyle\int^{x}_{\infty}}
			\left(
				2W_{0,2}(x_1,y)
				+
				\frac{1}{(x_1-y)^2}
			\right)\dd y
			-
			\partial_q W_{0,1}(x_1)
		}{
			x - T'(x) - 2W_{0,1}(x)
		},
	\end{equation}
	where $\int^{x}_{\infty} \frac{1}{(x_1-y)^2}\dd y = \frac{1}{x_1-x}$, and the result is expanded as in~\eqref{x1:x2:expns}.
\end{prop}

Before giving the proof, let us comment on the shape of the formula. Roughly speaking, it says that skin maps are obtained from annulus maps whose second boundary is unrooted (accounted by first term in the numerator), up to a few exceptional cases (the other two terms in the numerator) and up to the bubbling of internal faces and discs (the denominator, expanded as geometric series). Such an interpretation dates back to \cite{EO09}.

In a precise way for us, it reflects the recursive definition of skin maps in terms of elementary ones, together with the skinning relations for annuli and pointed discs. Indeed, the denominator is the generating series of the operator obtained from Tutte for skin maps: the term $T'(x)$ comes from case (I), while the term $2W_{0,1}(x)$ comes from case (D${}^{=}$). Dividing by $x - T'(x) - 2W_{0,1}(x)$ amounts to solving the second recursion for $S(x_1,x)$. The numerator is obtained by reconstructing the full series of $S(x_1,x)(x - T'(x) - 2W_{0,1}(x))$ in three parts. The positive powers of $x$ are fixed by the recursive construction of skin maps and give the positive part of $\frac{1}{x_1-x}$. The negative powers of $x$ are fixed by the skinning formula for annuli, in which we unroot the second boundary by taking the antiderivative. Finally, the constant term in $x$ is fixed by the skinning relation for pointed discs, and gives the contribution $-\partial_q W_{0,1}(x_1)$, together with the constant term of $\frac{1}{x_1-x}$. While the non-trivial annulus contribution appears with a factor of $2$, the pointed-disc contribution does not; this difference is due to the orientation of the root in the pointed-disc case.
 
\begin{proof}
	We determine separately the positive, negative, and constant parts, as a series in $x$, of $S(x_1,x)(x-T'(x)-2W_{0,1}(x))$. First, the recursive definition of skin maps, \Cref{rec:skin2}, gives the positive part:
	\begin{equation}
		\left[
			S(x_1,x)\bigl(x-T'(x)-2W_{0,1}(x)\bigr)
		\right]_{+,x}
		=
		\frac{1}{x_1-x}-\frac{1}{x_1}.
	\end{equation}
	Second, the constant term is fixed by the pointed-disc skinning relation \eqref{skinning:discs}:
	\begin{equation}
		\left[
			S(x_1,x)\bigl(x-T'(x)-2W_{0,1}(x)\bigr)
		\right]_{0,x}
		=
		-\left[
			S(x_1,x)(T'(x) + 2W_{0,1}(x))
		\right]_{0,x}
		=
		 -\partial_q W_{0,1}(x_1)+\frac{1}{x_1}.
	\end{equation}
	It remains to determine the negative part. By the skinning relation for annuli, \eqref{preTR}, after integration from $\infty$ to $x$, we get
	\begin{equation}
		\int_\infty^x W_{0,2}(x_1,y)\dd y
		=
		\biggl\langle
			S(x_1,\xi)
			\frac{W_{0,1}(\xi)-W_{0,1}(x)}{\xi-x}
		\biggr\rangle_\xi.
	\end{equation}
	We now use the elementary identity
	\begin{equation}
		\biggl\langle
			F(\xi)
			\frac{G(\xi)-G(x)}{\xi-x}
		\biggr\rangle_\xi
		=
		-\bigl[F(x)G(x)\bigr]_{-,x},
	\end{equation}
	valid whenever $F$ has only non-negative powers and $G$ has only negative powers. Applying it to $F(\xi)=S(x_1,\xi)$ and $G(\xi)=W_{0,1}(\xi)$ gives $\int_\infty^x W_{0,2}(x_1,y)\dd y = -[ S(x_1,x)W_{0,1}(x) ]_{-,x}.$ Since $S(x_1,x)(x-T'(x))$ has no negative powers of $x$, we find
	\begin{equation}
		\left[
			S(x_1,x)\bigl(x-T'(x)-2W_{0,1}(x)\bigr)
		\right]_{-,x}
		=
		\int_\infty^x 2W_{0,2}(x_1,y)\dd y.
	\end{equation}
	Combining the negative, positive, and constant parts, we obtain
	\begin{equation}
	\begin{split}
		S(x_1,x)\bigl(x-T'(x)-2W_{0,1}(x)\bigr)
		&=
		\int_\infty^x 2W_{0,2}(x_1,y)\dd y
		+
		\left(
			\frac{1}{x_1-x}-\frac{1}{x_1}
		\right)
		+
		\left(
			-\partial_q W_{0,1}(x_1)+\frac{1}{x_1}
		\right) \\
		&=
		\int_\infty^{x}
		\left(
			2W_{0,2}(x_1,y)
			+
			\frac{1}{(x_1-y)^2}
		\right)\dd y
		-
		\partial_q W_{0,1}(x_1).
	\end{split}
	\end{equation}
	Dividing by $x-T'(x)-2W_{0,1}(x)$ gives the result.
\end{proof}

\subsection{Example and comments}
Before concluding this section with a few remarks, we illustrate the skinning process step by step in \Cref{fig:stepbystep}.

\begin{figure}
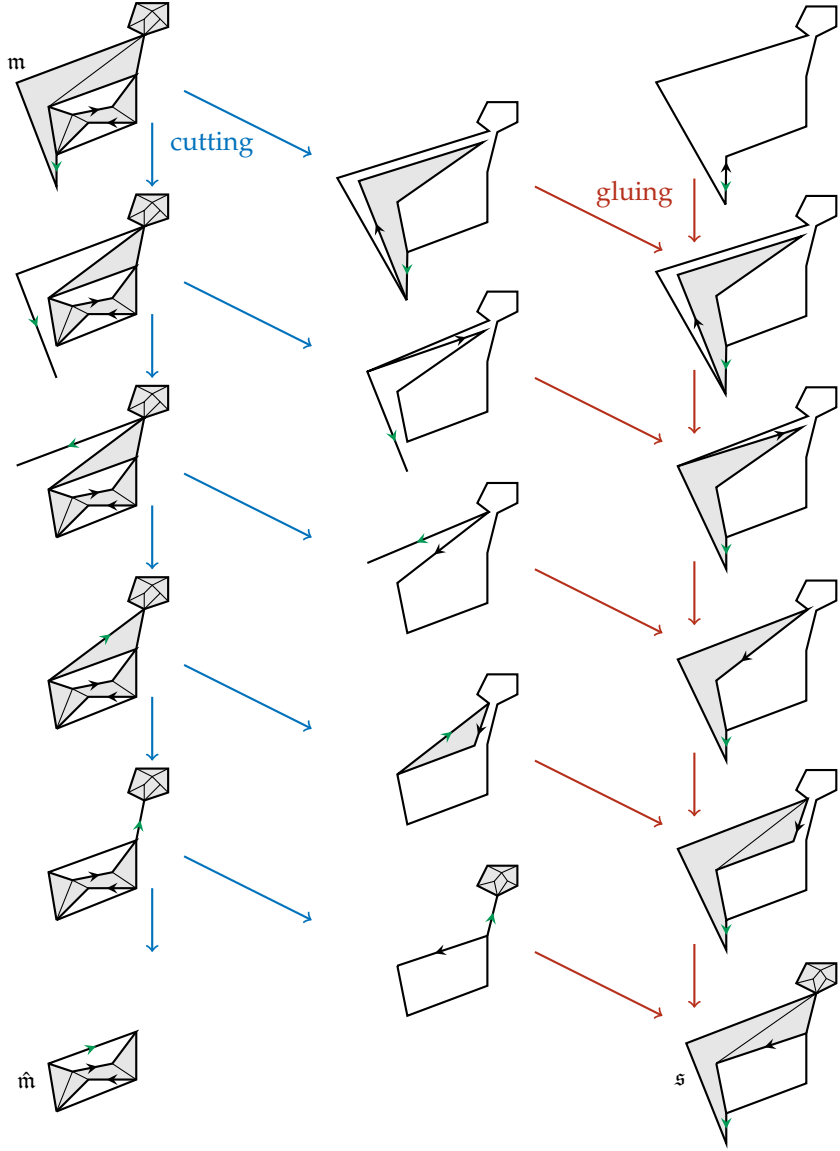

	\centering

	\caption{Example of a step-by-step skinning process. Starting from a map $\mathfrak{m}$ (top left), we repeatedly apply Tutte, producing a map with one fewer edge (left column), an elementary piece (middle column), and the partial skin (right column) at each step. The process terminates with the core map $\hat{\mathfrak{m}}$ (bottom left) and the skin $\mathfrak{s}$ (bottom right). The root of the first boundary is shown in green throughout the construction.}
	\label{fig:stepbystep}
\end{figure}

\subsubsection{Refined enumeration}
The skinning procedure suggests two natural refinements of the enumeration problem. First, one may assign a weight $u^a$ to the number of layers of the skin map. Then
\begin{equation}
	S^{(u)}(x_1,x)
	=
	\frac{
		{\displaystyle \int_{\infty}^{x}}
		\left(
			2W^{(u)}_{0,2}(x_1,y) + \frac{1}{(x_1 - y)^2}
		\right) \dd y
		-
		\partial_q W_{0,1}^{(u)}(x_1)
	}{
		x - u T'(x) - 2uW_{0,1}(x)
	},
\end{equation}
where the superscript $(u)$ indicates the refined generating series. Second, one may assign a weight $\lambda^{\ell}$ in~\eqref{skinning:rel}, i.e. keep track of the length of the curve along which the skin map and the core map are glued. This amounts to replacing $S(x_1,x)$ by $\lambda S(x_1,\lambda x)$.

\subsubsection{Relation to Eynard--Orantin topological recursion}
The formulae in \Cref{skinning:maps,skin:enum:maps} are reminiscent of the Eynard--Orantin topological recursion formula \cite{EO07}. We will show in \Cref{sec:analytic} that our formal-series formulae imply Eynard's analytic form of topological recursion for maps under off-criticality assumptions on the weights of the model \cite{Eyn04}.

In fact, as explained in the introduction, our approach gives a bijective meaning to the topological recursion formula for maps as the excision of a pair of pants embedded into the map itself. This will be explained in greater generality in \Cref{sec:pop}. For the moment, let us indicate where the pair of pants hides. Recall that, by iteratively cutting the root-edge of the first boundary according to Tutte's procedure, one first cut off a skin map, and then, after one further step, the topology changes. This last step is inverse to gluing to the second boundary of the skin map an additional elementary skin map arising this time from one of the topology-changing cases (R) or (D${}^{>}$). The result is a thickened version of the skin map, which is topologically a pair of pants (see~\Cref{fig:pants:decomp}). In case (R), this pair of pants is singular and is glued to the remaining core map only along one boundary, corresponding to the contribution of the derivative term in~\eqref{preTR}. In every case, it is completely determined by the data of the skin map and the core map. Thus one may regard $S(x_1,x)$ as enumerating these pairs of pants.

\subsubsection{Difference from earlier works}
At this point, it is natural to compare our procedure with other decompositions of maps appearing in the literature.

In \cite{BGM22,BGM25}, the authors introduce the \emph{trumpet decomposition} for maps, in which a trumpet is cut off from each boundary so that all boundaries of the remaining map are tight, meaning that they realise the shortest paths in their homotopy class. The similarity with our skinning process is that the trumpets, just like our skin maps, are annuli with one ordinary and one simple boundary. They are, however, different: trumpets are cut off in order to make the remaining boundaries tight, whereas our skins are obtained by following Tutte's procedure until the next step would change the topology.

Another procedure based on Tutte's decomposition is \emph{peeling} \cite{Wat95,Bud15,Bud18,Cur23}, which has mainly been used to study the random geometry of planar maps. One may describe it as a construction of discs with boundary degree $\ell$, starting from what we call an empty skin of degree $\ell$ together with internal faces, and then iteratively gluing pairs of edges. At each step, three types of moves may occur. Gluing an edge of the second boundary of the skin to an edge of a separate internal face leaves the number of boundaries unchanged. Gluing two neighbouring edges of the same boundary also leaves the number of boundaries unchanged, unless they are neighbouring on both sides; in that exceptional case, the boundary must have degree $2$, and the gluing removes it. Finally, gluing two non-neighbouring edges of the same boundary increases the number of boundaries. The process starts with two boundaries, namely those of the empty skin, and stops as soon as only one boundary remains. Just before stopping, it produces a skin map whose second boundary has degree $2$. In this way, peeling defines a Markov chain.

There are, however, important differences between peeling and our procedure. First, our construction applies to maps of semi-stable topology, whereas peeling is carried out for discs. As a consequence, both the role of the skin map and the stopping rule are different: in peeling, the process ends when all boundaries but one have disappeared, while in our setting it ends just before a topology-changing Tutte outcome, namely (R) or (D${}^{>}$). Second, our steps (I), (D${}^=$) and the connected sub-case of (D${}^{>}$) are themselves sequences of peeling steps, and are therefore closer to the lazy peeling of \cite{Bud15}. Third, in peeling the re-rooting is governed by an arbitrary Markovian rule, whereas in our setting it must be chosen so that the skin remains connected. This raises the question whether peeling could be combined with our procedure to study the random geometry of maps of fixed topology through the random geometry of skin maps and their recursive gluing.

\section{Generalisation to maps with tubes or self-avoiding loops}
\label{sec:maps:tubes}

We now extend the previous discussion to the case of \emph{maps with tubes}. They form the first non-trivial extension of maps in which internal faces may have the topology of annuli, rather than only discs. Moreover, they are equivalent to the better-known model of \emph{maps with self-avoiding loops}, as explained in the \Cref{ssec:sef-avoiding}. They also exhibit, in a simplified setting, the new phenomena that arise for general stuffed maps. For this reason, we treat them with less detail than maps, while emphasising the genuinely new features.

\subsection{Definition and generating series}
Maps with tubes are defined in the same way as maps, except that we also allow internal faces with annular topology \cite{BEO15}. The degree of an annular face is the unordered pair $\set{k_1,k_2}$ of the degrees of its two boundary components, and we only allow $\set{k_1,k_2} \neq \set{0,0}$. 
We denote by $\MM_{g;\ell_1,\ldots,\ell_n}^{\circ}$ the set of maps with tubes of topology $(g,n)$ and respective boundary degrees $\ell_1,\ldots,\ell_n$. The trivial map is the only map with a boundary of degree $0$, so $\MM_{0;0}^{\circ}$ is a singleton. For a map with tubes $\mathfrak{m}$ we set
\begin{equation}
	N_{k_1,k_2}(\mathfrak{m}) \coloneqq \card*{\Set{ \text{annular faces of degree $\set{k_1,k_2}$} }}.
\end{equation}
We then define the generating series of maps with tubes of topology $(g,n)$ with fixed boundary degrees $\ell_1,\ldots,\ell_n$ by
\begin{equation}
	W_{g;\ell_1,\ldots,\ell_n}^{\circ}
	\coloneqq
	\sum_{\mathfrak{m} \in \MM_{g;\ell_1,\ldots,\ell_n}^\circ}
	\frac{q^{v(\mathfrak{m})}}{\card{\Aut(\mathfrak{m})}}
	\prod_{k \geq 1} t_k^{N_{k}(\mathfrak{m})}
	\prod_{\substack{k_1 \geq k_2 \geq 0 \\ k_1+k_2>0}} t_{k_1,k_2}^{N_{k_1,k_2}(\mathfrak{m})},
\end{equation}
and, with free boundary degrees, by
\begin{equation}
	W_{g,n}^{\circ}(x_1,\ldots,x_n)
	\coloneqq
	\delta_{g,0}\delta_{n,1}\frac{q}{x_1}
	+
	\sum_{\ell_1,\ldots,\ell_n \geq 1}
	\frac{W_{g;\ell_1,\ldots,\ell_n}^{\circ}}{x_1^{\ell_1 + 1} \cdots x_n^{\ell_n + 1}}.
\end{equation}
Unlike maps, maps with tubes may have non-trivial automorphisms, so
\begin{equation}
	W_{g,n}^{\circ}
	\in
	\mathbb{Q}[x_1^{-1},\ldots,x_n^{-1},t_3,t_4,\ldots,t_{1,0},t_{2,0},\ldots,t_{1,1},t_{2,1},\ldots]\bbraket{q,t_1,t_2}.
\end{equation}
It is convenient to introduce Boltzmann weights $t_{k_1,k_2}=t_{k_2,k_1}$ for all $k_1,k_2\geq 0$, and to set $t_{k_1,k_2}=0$ whenever $(k_1,k_2)=(0,0)$ or one of $k_1,k_2$ is negative. In our constructions we will also encounter faces that are labelled but not rooted; we still refer to them as boundaries, and we will make their special role explicit when needed. Maps are recovered by setting $t_{k_1,k_2}=0$ for all $k_1,k_2$. We introduce the disc and annular potentials
\begin{equation}
	T_{0,1}(x)
	\coloneqq
	\sum_{k \geq 1} \frac{t_k}{k} x^k,
	\qquad
	T_{0,2}(x_1,x_2)
	\coloneqq
	\sum_{\substack{k_1,k_2\geq 0 \\ k_1+k_2>0}} \frac{t_{k_1,k_2}}{[k_1] [k_2]}x_1^{k_1}x_2^{k_2} ,
\end{equation}
where
\begin{equation}
	[k] \coloneqq
	\begin{cases}
		k, & k>0,\\
		1, & k=0.
	\end{cases}
\end{equation}
Comparing to \Cref{sec:ord:maps}, we have $T_{0,1}(x) = T(x)$, and in the notations of \cite{BEO15} the generating series of ``rings'' is $R(x_1,x_2) = e^{T_{0,2}(x_1,x_2)}$. 

\subsection{Maps with self-avoiding loops}
\label{ssec:sef-avoiding}
Maps with tubes are in correspondence with the better-known model of maps with self-avoiding loops. The latter has been extensively studied in the literature, due to its rich phase diagram displaying a continuous family of universality classes \cite{Kos89,KS92,EK95,BBG12a,BBG12b,Bud18,Kho22}. The former has rarely been considered, but it is better adapted to the discussion of skinning and to the further generalisation to stuffed maps in \Cref{sec:stuffed:maps}. Although it will play no role in the subsequent analysis, it may therefore be useful to describe the correspondence.

A map with self-avoiding loops is a map with two types of internal faces of disc topology, normal and special, and two types of half-edges, normal and special. Normal faces have only normal half-edges. Special faces carry exactly two special half-edges, the other half-edges being normal: going around a special face, we see $k_1 \geq 0$ normal edges, one special edge, $k_2 \geq 0$ normal edges, and one special edge. We call the pair $\set{k_1,k_2}$ the \emph{interspaced degree}. We impose that special faces have at least one normal half-edge, that is, $k_1 + k_2 > 0$. Crucially, edges can only be made of two half-edges of the same type. In each special face, we draw a simple path connected the two special edges, and due to the previous edge-rule these paths must assemble into loops, and loops are mutually avoiding. This configuration of self-avoiding loops (up to isotopy in the map minus the vertices) fully determines which faces and which edges are special, namely those crossed by a loop. The degree of a loop is the number of special faces it crosses. In particular, two special edges of a given special face glued to each other indicate a loop of degree $1$. We denote by $\MM_{g;\ell_1,\ldots,\ell_n}^{\alpha}$ the set of maps with self-avoiding loops of topology $(g,n)$ and with $n$ normal boundary faces of respective degrees $\ell_1,\ldots,\ell_n \geq 1$. For a map with self-avoiding loops $\mathfrak{m}$, we set
\begin{equation}
\begin{split}
	L(\mathfrak{m}) & \coloneqq \card*{\Set{\text{loops}}}, \\
	N_{k}(\mathfrak{m}) & \coloneqq \card*{\Set{\text{normal faces of degree $k$}}}, \\
	P_{k_1,k_2}(\mathfrak{m}) & \coloneqq \card*{\Set{ \text{special faces of interspaced degree $\set{k_1,k_2}$} }}.
\end{split}
\end{equation}
We then define the generating series
\begin{equation}
	W_{g;\ell_1,\ldots,\ell_n}^{\alpha}
	\coloneqq
	\sum_{\mathfrak{m} \in \MM_{g;\ell_1,\ldots,\ell_n}^{\alpha}}
		\frac{q^{v(\mathfrak{m})}\mu^{L(\mathfrak{m})}}{\card{\Aut(\mathfrak{m})}}
		\prod_{k \geq 1} t_k^{N_k(\mathfrak{m})}
		\prod_{\substack{k_1 \geq k_2 \geq 0 \\ k_1 + k_2 > 0}}
			s_{k_1,k_2}^{P_{k_1,k_2}(\mathfrak{m})}.
\end{equation}
It is convenient to introduce Boltzmann weights $s_{k_1,k_2} = s_{k_2,k_1}$ for all $k_1,k_2 \geq 0$, and to set $s_{0,0} = 0$.

A map with self-avoiding loops determines a map with tubes, whose annular faces are given by the sequences of special faces crossed by loops, see~\Cref{fig:loops:tubes}. Keeping track of all possible ways in which an annular face can be obtained from a sequence of special faces along a loop, one sees that
\begin{equation}
	W_{g;\ell_1,\ldots,\ell_n}^{\alpha} = W_{g;\ell_1,\ldots,\ell_n}^{\circ}
\end{equation}
after the invertible change of formal variables
\begin{equation}
	\sum_{k_1,k_2 \geq 0} t_{k_1,k_2}x_1^{k_1}x_2^{k_2} = - \mu \log\left(1 - \sum_{k_1,k_2 \geq 0} s_{k_1,k_2}x_1^{k_1}x_2^{k_2}\right),
\end{equation}
where $t_{k_1,k_2}$ are the annular weights in maps with tubes and $s_{k_1,k_2}$ are the weights of special faces in maps with self-avoiding loops. Note the agreement between the conventions $t_{0,0} = 0$ and $s_{0,0} = 0$. The correspondence is stated and proved in general in \cite{BGS}, but its origin and its realisation in special cases can already be found in \cite{BBG12a,BBG12b}.

\begin{figure}
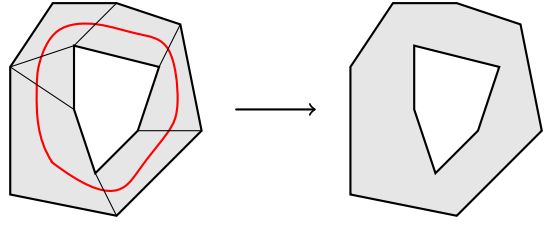


	\caption{An annular face built from a cyclic sequence of special disc faces.}
	\label{fig:loops:tubes}
\end{figure}

\subsection{Tutte's procedure}
\label{ssec:Tutte:tubes}
Tutte's procedure for maps with tubes follows the same prescription as for maps: we erase the edge $r_1$ supporting the root $\rho_1$ and distinguish the cases according to the face or boundary lying on the other side. The cases (R) and (D) are unchanged, except that the remaining maps may now carry tubes. The case (I), however, is richer, because the internal face adjacent to $r_1$ may be annular; as a result, case (I) now has both topology-preserving and topology-changing outcomes. We therefore review only this case, subdividing it into (I${}^{=}$) and (I${}^{>}$).

(I${}^{=}$) This includes two types of topology-preserving internal-face removal. First, if $r_1$ bounds an internal disc face of degree $k\geq 1$, erasing it is equivalent to removing this internal face, and yields a map of genus $g$ with $n$ boundaries of respective degrees $\ell_1+k-2,\bm{\ell}_I$.

Second, suppose that $r_1$ bounds $\partial_1$ and the boundary $\delta'$ of degree $k'\geq 1$ of an annular face whose other boundary $\delta''$ has degree $k'' \geq 0$. When erasing $r_1$, we also erase the interior of the annular face. If one of the two resulting components is a disc, then we treat the operation as topology-preserving, by including this disc in the removed piece. The boundary $\partial'' \coloneqq \delta''$ survives as a boundary of the resulting object, but remains exceptionally unrooted, see~\Cref{fig:I:tubes}.

\begin{figure}[H]
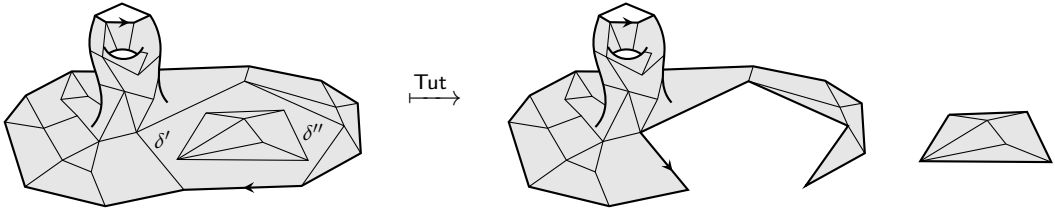

	\centering

	\caption{Case (I${}^=$), generic situation.}
	\label{fig:I:tubes}
\end{figure}

The following special cases are naturally included here. If $\partial_1$ is glued to $\delta'$ and $\ell_1=k'\geq 1$, the first component is a disc with boundary degree $2k'-2$ and no internal faces, whose first $k'-1$ half-edges are glued to the next $k'-1$ ones, see~\Cref{fig:deg1}; this includes the case $k'=1$, with the first component being the trivial map. If $k''=0$, then the second component is the trivial map, see~\Cref{fig:deg2}.

\begin{figure}[H]
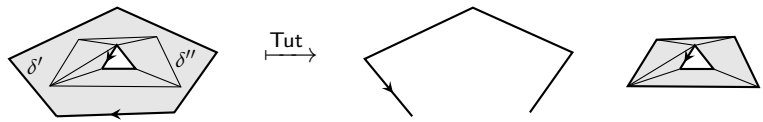

	\centering
	%
	\caption{Case (I${}^=$), degenerate sub-case: the first component is a disc without internal faces, possibly the trivial map.}
	\label{fig:deg1}
\end{figure}

\begin{figure}[H]
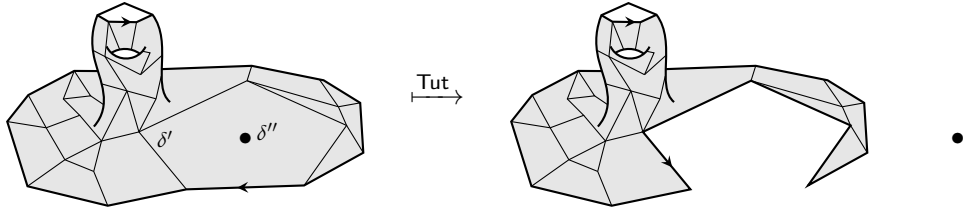

	\centering
	%
	\caption{Case (I${}^=$), degenerate sub-case: the second component is the trivial map.}
	\label{fig:deg2}
\end{figure}

(I${}^{>}$) This is the topology-changing annular-face case. Suppose that $r_1$ bounds $\partial_1$ and the boundary $\delta'$ of degree $k' \geq 1$ of an annular face whose other boundary $\delta''$ has degree $k' \geq 1$. After erasing $r_1$ we get a new boundary $\partial'$ of degree $\ell_1 + k' - 2$ which we root at $\overline{\rho_1}{}^{-}$ (being topology-changing actually forces $k' \geq 2$). As before, the boundary $\partial'' \coloneqq \delta''$ survives as a boundary of the resulting object but remains  unrooted.

As in the (D$^{>}$) case, there are two possibilities.

\emph{Connected case.} If the result remains connected, it is a map of genus $g-1$ with $n+1$ boundaries of respective degrees $\ell_1+k'-2,k'',\bm{\ell}_I$. We declare that the first boundary is rooted at $\overline{\rho_1}{}^{-}$, while the second boundary is unrooted.

\begin{figure}[H]
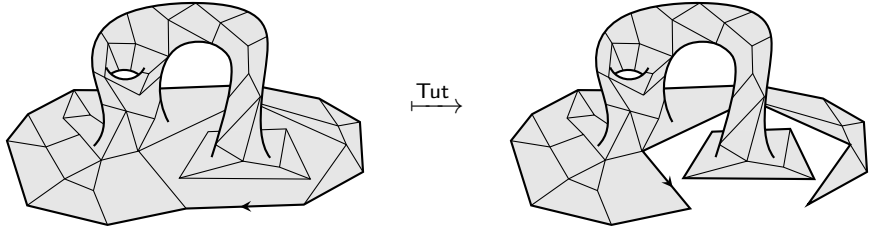

	\centering

	\caption{\label{fig:Tutte:tubes:conn} Case (I${}^>$), connected sub-case.}
\end{figure}

\emph{Disconnected case.} If the result is disconnected and neither component is a disc, then the first component is the one with boundary $\partial'$, its other boundaries are $\bm{\partial}_{J'}$ with respective degrees $\ell_1+k'-2,\bm{\ell}_{J'}$, and it has genus $h'$. The second component has genus $h''$, its first boundary is $\partial''$, and its other boundaries are $\bm{\partial}_{J''}$ with respective degrees $k'',\bm{\ell}_{J''}$. We have $h'+h''=g$ and $J'\sqcup J''=I$, and moreover $(h',J')$ and $(h'',J'')$ are both different from $(0,\emptyset)$.

\begin{figure}[H]
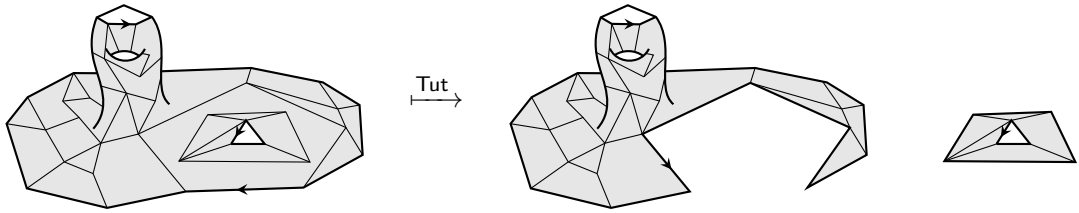

	\centering

	\caption{\label{fig:Tutte:tubes:disc} Case (I${}^>$), disconnected sub-case.}
\end{figure}

This gives rise to a bijection of orbifold sets\footnote{
	By an orbifold set, we mean a set in which each point is equipped with its finite isotropy group, namely the automorphism group of the corresponding map with tubes. A bijection of orbifold sets consists of a bijection of the underlying sets together with an isomorphism between the isotropy groups of each pair of corresponding points.
}
\begin{equation}
	\Tutte
	\colon
	\MM_{g;\ell_1,\bm{\ell}_I}^{\circ}
	\overset{\simeq}{\longrightarrow}
	\mathbb{I}_{g;\ell_1,\bm{\ell}_I}^{=}
	\sqcup
	\mathbb{I}_{g;\ell_1,\bm{\ell}_I}^{>}
	\sqcup
	\mathbb{R}_{g;\ell_1,\bm{\ell}_I}
	\sqcup
	\mathbb{D}_{g;\ell_1,\bm{\ell}_I}^{>}
	\sqcup
	\mathbb{D}_{g;\ell_1,\bm{\ell}_I}^{=},
\end{equation}
where we have defined
\begin{equation}
\begin{split}
	\mathbb{I}_{g;\ell_1,\bm{\ell}_I}^{=}
	&\coloneqq
	\bigsqcup_{k \ge 1}
		\MM^{\circ}_{g;\ell_1 + k - 2,\bm{\ell}_I}
	\sqcup
	\bigsqcup_{k',k'' \geq 1}
		\bigl(\MM_{0;\ell_1 + k' - 2}^{\circ} \times \MM^{\circ}_{g;\underline{k}'',\bm{\ell}_{I}}\bigr)
	\sqcup
	\bigsqcup_{\substack{k' \geq 1 \\ k'' \geq 0}}
		\bigl(\MM_{g;\ell_1 + k' - 2,\bm{\ell}_I}^{\circ} \times \MM^{\circ}_{0;\underline{k}''}\bigr), \\
	\mathbb{I}_{g;\ell_1,\bm{\ell}_I}^{>}
	&\coloneqq
	\bigsqcup_{k',k'' \geq 1} \Biggl(
		\MM^{\circ}_{g-1;\ell_1 + k' - 2,\underline{k}'',\bm{\ell}_I}
		\sqcup
		\bigsqcup_{\substack{h' + h'' = g \\ J' \sqcup J'' = I}}^{\nodisc}
			\MM_{h';\ell_1 + k' - 2,\bm{\ell}_{J'}}^{\circ} \times \MM^{\circ}_{h'';\underline{k}'',\bm{\ell}_{J''}}
	\Biggr), \\
	\mathbb{R}_{g;\ell_1,\bm{\ell}_I}
	&\coloneqq
	\bigsqcup_{i = 2}^{n}
		\MM^{\circ}_{g;\ell_1 + \ell_i - 2,\bm{\ell}_{I_i}} \times \set{0,\ldots,\ell_i - 1}, \\
	\mathbb{D}_{g;\ell_1,\bm{\ell}_I}^{>}
	&\coloneqq
	\bigsqcup_{\substack{k',k'' \geq 1 \\ k' + k'' = \ell_1 - 2}} \Biggl(
		\MM^{\circ}_{g - 1;k',k'',\bm{\ell}_I}
		\sqcup
		\bigsqcup_{\substack{h' + h'' = g \\ J' \sqcup J'' = I}}^{\nodisc}
			\MM^{\circ}_{h';k',\bm{\ell}_{J'}} \times \MM^{\circ}_{h'';k'',\bm{\ell}_{J''}}
	\Biggr), \\
	\mathbb{D}_{g;\ell_1,\bm{\ell}_{I}}^{=}
	&\coloneqq
	\bigsqcup_{k = 0}^{\ell_1 - 3}
		\bigl(\MM^{\circ}_{0;k} \times \MM_{g;\ell_1 - k - 2,\bm{\ell}_I}^{\circ}\bigr)
		\sqcup
		\bigl(\MM^{\circ}_{g;\ell_1  - k - 2,\bm{\ell}_I} \times \MM_{0;k}^{\circ}\bigr).
\end{split}
\end{equation}
An underlined index $\underline{k}''$ means that we consider a quotient set in which maps carry a distinguished boundary of degree $k''$ that is not rooted. This translates into the following relations at the level of generating series, cf.~\cite{BEO15}.
 
\begin{lem}
	Let $g \geq 0$, $n \geq 1$, and $\ell_1,\ldots,\ell_n \geq 1$ be such that $(g,n) \neq (0,1)$. We have
		\begin{equation}
	\begin{split}
		W^{\circ}_{g;\ell_1,\bm{\ell}_I}
		&=
		\sum_{k \geq 1} t_k \, W^{\circ}_{g;\ell_1+k-2,\bm{\ell}_I}
		+
		\sum_{k',k'' \geq 1}
			\frac{t_{k',k''}}{k''}
			W^{\circ}_{0;\ell_1+k'-2}
			W^{\circ}_{g;k'',\bm{\ell}_{I}}
		+
		\sum_{\substack{k'\geq 1 \\ k''\geq 0}}
			\frac{t_{k',k''}}{[k'']}
			W^{\circ}_{g;\ell_1+k'-2,\bm{\ell}_I}
			W^{\circ}_{0;k''} \\
		&\quad+
		\sum_{k',k'' \geq 1}
			\frac{t_{k',k''}}{k''}
			\Biggl(
				W^{\circ}_{g-1;\ell_1+k'-2,k'',\bm{\ell}_I}
				+
				\sum_{\substack{h' + h'' = g \\ J' \sqcup J'' = I}}^{\nodisc}
				W^{\circ}_{h';\ell_1 + k' - 2,\bm{\ell}_{J'}}
				W^{\circ}_{h'';k'',\bm{\ell}_{J''}}
			\Biggr)
		+
		\sum_{i=2}^n \ell_i \, W^{\circ}_{g;\ell_1+\ell_i-2,\bm{\ell}_{I_i}} \\
		&\quad+
		\sum_{\substack{k',k'' \geq 1 \\  k' + k'' = \ell_1 - 2}}
		\Biggl(
			W^{\circ}_{g-1;k',k'',\bm{\ell}_{I}}
			+
			\sum_{\substack{h' + h'' = g \\ J' \sqcup J'' = I}}^{\nodisc}
			W^{\circ}_{h';k',\bm{\ell}_{J'}}
			W^{\circ}_{h'';k'',\bm{\ell}_{J''}}
			+
			W^{\circ}_{0;k'}W^{\circ}_{g;k'',\bm{\ell}_I}
			+
			W^{\circ}_{g;k',\bm{\ell}_I}W^{\circ}_{0;k''}
		\Biggr).
	\end{split}
	\end{equation}
	For generating series with free boundary degrees, this becomes
	\begin{equation}
	\label{Tutte:rec:tubes}
	\begin{split}
		\mathcal{L}^{\circ}_{x_1}W^{\circ}_{g,n}(x_1,\bm{x}_I)
		&=
		\Bigg\langle
			\frac{\partial_{\xi} T_{0,2}(\xi,\eta)}{x_1 - \xi}
			\biggl(
				W^{\circ}_{g-1,n+1}(\xi,\eta,\bm{x}_I)
				+
				\sum_{\substack{h' + h'' = g \\ J' \sqcup J'' = I}}^{\nodisc}
				W^{\circ}_{h',1+\card{J'}}(\xi,\bm{x}_{J'})W^{\circ}_{h'',1+\card{J''}}(\eta,\bm{x}_{J''})
			\biggr)
		\Bigg\rangle_{\xi,\eta} \\
		&\qquad
		+
		\sum_{i=2}^n
		\partial_{x_i}\left(
			\frac{W^{\circ}_{g,n-1}(x_1,\bm{x}_{I_i}) - W^{\circ}_{g,n-1}(x_i,\bm{x}_{I_i})}{x_1 - x_i}
		\right) \\
		&\qquad
		+
		W^{\circ}_{g-1,n+1}(x_1,x_1,\bm{x}_I)
		+
		\sum_{\substack{h' + h'' = g \\ J' \sqcup J'' = I}}^{\nodisc}
			W^{\circ}_{h',1+\card{J'}}(x_1,\bm{x}_{J'})W^{\circ}_{h'',1+\card{J''}}(x_1,\bm{x}_{J''}).
	\end{split}
	\end{equation}
	where $\frac{1}{x_1 - y}$ is expanded as in~\eqref{x1:x2:expns} and $\mathcal{L}_{x_1}$ is the operator
	\begin{equation}
		\mathcal{L}^{\circ}_{x} f(x)
		\coloneqq
		\bigl[
			(x- T_{0,1}'(x))f(x)
		\bigr]_{-,x}
		-
		2\,W^{\circ}_{0,1}(x)f(x)
		-
		\Bigg\langle
			\frac{\partial_{\xi}T_{0,2}(\xi,\eta)}{x-\xi}
			\bigl(
				W^{\circ}_{0,1}(\xi) f(\eta)
				+
				f(\xi) W^{\circ}_{0,1}(\eta)
			\bigr)
		\Bigg\rangle_{\xi,\eta}.
	\end{equation}
\end{lem}

\subsection{Skinning maps}
\label{ssec:skin:tubes}
As for maps, we now separate the topology-changing and topology-preserving cases of Tutte's procedure, and use this distinction to define core maps and skin maps. The new feature is that annular faces may appear in the detached part, producing elementary skin maps with tubes and, in particular, unrooted boundary components.

\begin{defn}
	Let $g \geq 0$, $n \geq 1$, and $\ell_1,\ldots,\ell_n \geq 1$ be such that $(g,n) \neq (0,1)$. We denote by $\hat{\MM}^{\circ}_{g;\ell_1,\bm{\ell}_I}$ the set of maps with tubes $\hat{\mathfrak{m}}$ of topology $(g,n)$ with boundary degrees $\ell_1,\bm{\ell}_I$ such that applying Tutte's procedure changes the topology, i.e.
	\begin{equation}
		\Tutte(\hat{\mathfrak{m}})
		\in
		\mathbb{I}_{g;\ell_1,\bm{\ell}_I}^{>}
		\sqcup
		\mathbb{R}_{g;\ell_1,\bm{\ell}_I}
		\sqcup
		\mathbb{D}_{g;\ell_1,\bm{\ell}_I}^{>}.
	\end{equation}
	These are the \emph{core maps with tubes}. We introduce their generating series
	\begin{equation}
	\begin{split}
		\hat{W}_{g;\ell_1,\bm{\ell}_I}^{\circ}
		&\coloneqq
		\sum_{\hat{\mathfrak{m}} \in \hat{\MM}^{\circ}_{g;\ell_1,\bm{\ell}_I}}
		\frac{q^{v(\hat{\mathfrak{m}})}}{\card{\Aut(\hat{\mathfrak{m}})}}
		\prod_{k \geq 1} t_{k}^{N_k(\hat{\mathfrak{m}})}
		\prod_{\substack{k_1 \geq k_2 \geq 0 \\ k_1+k_2>0}} t_{k_1,k_2}^{N_{k_1,k_2}(\hat{\mathfrak{m}})}, \\
		\hat{W}_{g,n}^{\circ}(x_1,\ldots,x_n)
		&\coloneqq
		\sum_{\ell_1,\ldots,\ell_n \geq 1}
		\frac{\hat{W}_{g;\ell_1,\bm{\ell}_I}^{\circ}}{x_1^{\ell_1}x_2^{\ell_2 + 1} \cdots x_n^{\ell_n + 1}}.
	\end{split}
	\end{equation}
\end{defn}

By retaining only the topology-changing terms in~\eqref{Tutte:rec:tubes}, we obtain the following relation for fixed boundary degrees:
\begin{equation}
\begin{split}
	\hat{W}^{\circ}_{g;\ell_1,\bm{\ell}_I}
	&=
	\sum_{k',k'' \geq 1}
	\frac{t_{k',k''}}{k''}
	\Biggl(
		W^{\circ}_{g-1;\ell_1+k'-2,k'',\bm{\ell}_I}
		+
		\sum_{\substack{h' + h'' = g \\ J' \sqcup J'' = I}}^{\nodisc}
			W^{\circ}_{h';\ell_1 + k' - 2,\bm{\ell}_{J'}}W^{\circ}_{h'';k'',\bm{\ell}_{J''}}
	\Biggr) \\
	&\quad
	+
	\sum_{i=2}^n \ell_i \, W^{\circ}_{g;\ell_1+\ell_i-2,\bm{\ell}_{I_i}}
	+
	\sum_{\substack{k',k'' \geq 1 \\ k' + k'' = \ell_1 - 2}}
	\Biggl(
		W^{\circ}_{g-1;k',k'',\bm{\ell}_{I}}
		+
		\sum_{\substack{h' + h'' = g \\ J' \sqcup J'' = I}}^{\nodisc}
			W^{\circ}_{h';k',\bm{\ell}_{J'}}W^{\circ}_{h'';k'',\bm{\ell}_{J''}}
	\Biggr).
\end{split}
\end{equation}
For free boundary degrees, this becomes
\begin{equation}
\label{Wcore:tubes:rec}
\begin{split}
	\hat{W}^{\circ}_{g,n}(x_1,\bm{x}_I)
	&=
	\Bigg\langle
		\frac{\partial_{\xi} T_{0,2}(\xi,\eta)}{x_1 - \xi}
		\biggl(
			W^{\circ}_{g-1,n+1}(\xi,\eta,\bm{x}_I)
			+
			\sum_{\substack{h' + h'' = g \\ J' \sqcup J'' = I}}^{\nodisc}
			W^{\circ}_{h',1+\card{J'}}(\xi,\bm{x}_{J'})W^{\circ}_{h'',1+\card{J''}}(\eta,\bm{x}_{J''})
		\biggr)\Bigg\rangle_{\xi,\eta} \\
	&\quad
	+
	\sum_{i=2}^n
	\partial_{x_i}\left(
		\frac{W^{\circ}_{g,n-1}(x_1,\bm{x}_{I_i}) - W^{\circ}_{g,n-1}(x_i,\bm{x}_{I_i})}{x_1 - x_i}
	\right) \\
	&\quad
	+
	W^{\circ}_{g-1,n+1}(x_1,x_1,\bm{x}_I)
	+
	\sum_{\substack{h' + h'' = g \\ J' \sqcup J'' = I}}^{\nodisc}
		W^{\circ}_{h',1+\card{J'}}(x_1,\bm{x}_{J'})W^{\circ}_{h'',1+\card{J''}}(x_1,\bm{x}_{J''}).
\end{split}
\end{equation}

We define \emph{elementary skin maps with tubes}  $\mathfrak{e}$ by examining the various cases of Tutte's procedure. All the boundaries of $\mathfrak{e}$ are simple except perhaps the first one. For the cases (R) and (D), the construction is the same as in \Cref{ssec:skin}, so we discuss only the new cases.

(I${}^{>}$) The elementary skin map $\mathfrak{e}$ is a pair of pants with one internal annular face $\mathfrak{f}$. The boundary $\partial_1\mathfrak{e}$ has root $\rho_1$ and degree $\ell_1$, and happens to be simple. The boundary $\partial_2\mathfrak{e}$ has degree $\ell_1 + k' - 2$ and root $\bbarsuperscript{\rho_1}{-}$. This root and the $\ell_1-2$ half-edges before are glued to the $\ell_1-1$ half-edges after $\rho_1$. Finally, $\partial_3\mathfrak{e}$ has degree $k''$ and is unrooted. $\mathfrak{f}$ has boundary components of degree $k'$ and $k''$. The latter is glued to $\partial_3\mathfrak{e}$. The former has a half-edge glued to $\rho_1$, and the $k' - 1$ half-edges before it are glued to the root of $\partial_2\mathfrak{e}$ and the $k' - 2$ half-edges after it.

(I${}^=$) If a disc face is removed, the construction is the same as in \Cref{ssec:skin}. We therefore consider the case of an annular face. From the point of view of Tutte's procedure, this case is similar to the disconnected case (I${}^>$), except that one of the two components in $\Tutte(\mathfrak{m})=(\mathfrak{m}',\mathfrak{m}'')$ is a disc. The elementary skin map $\mathfrak{e}$ is now an annulus. It is obtained by first constructing an auxiliary pair of pants $\mathfrak{a}$ as in (I${}^{>}$), and then, if $\mathfrak{m}'$ (respectively, $\mathfrak{m}''$) is a disc, gluing the boundary face of this disc to $\partial_2\mathfrak{a}$ (respectively, to $\partial_3\mathfrak{a}$). The boundary\footnote{
	In contrast to case (I${}^{=}$), if $\mathfrak{m}'$ is a disc, gluing it to $\mathfrak{a}$ may cause the boundary $\partial_1\mathfrak{e}$ to become non-simple.
} $\partial_1\mathfrak{e}$ has degree $\ell_1$ and root $\rho_1$. If $\mathfrak{m}'$ is a disc, $\partial_2\mathfrak{e}$ has degree $k''$, is unrooted, and is glued to the second boundary $\delta''$ of an annular face of degree $k',k''$. For this gluing, roots must be chosen and matched, but the resulting map is independent of this choice. If $\mathfrak{m}''$ is a disc, $\partial_2\mathfrak{e}$ has degree $\ell_1+k'-2$ and is rooted at $\bbarsuperscript{\rho_1}{-}$. This root and the following $k'-2$ half-edges are glued to $k'-1$ consecutive half-edges of the first boundary $\delta'$ of an annular face of degree $k',k''$, while the last half-edge of $\delta'$ is glued to $\rho_1$.

\begin{defn}
	Let $\ell_1,\ell \geq 1$. We denote by $\mathbb{E}_{\ell_1,\ell}^{\circ}$ the set of elementary skin maps with tubes $\mathfrak{e}$ with boundary degrees $\ell_1$ and $\ell$, arising from the topology-preserving cases (I${}^=$) and (D${}^{=}$). We define their generating series by
	\begin{equation}
		E_{\ell_1,\ell}^{\circ}
		\coloneqq
		\sum_{\mathfrak{e} \in \mathbb{E}_{\ell_1,\ell}^{\circ}}
		\frac{q^{v(\mathfrak{e}) - \ell}}{\card{\Aut(\mathfrak{e})}}
		\prod_{k \geq 1} t_k^{N_k(\mathfrak{e})}
		\prod_{\substack{k_1 \geq k_2 \geq 0 \\ k_1+k_2 > 0}} t_{k_1,k_2}^{N_{k_1,k_2}(\mathfrak{e})}.
	\end{equation}
\end{defn}

Our construction implies
\begin{equation}
\label{eskin:tubes:bij}
	E^\circ_{\ell_1,\ell}
	=
	t_{\ell - \ell_1 + 2}
	+
	2W_{0;\ell_1 - \ell - 2}^\circ
	+
	\sum_{k' \geq 1} \frac{t_{k',\ell}}{\ell}W_{0;\ell_1 + k' - 2}^\circ
	+
	\sum_{k'' \geq 0} \frac{t_{\ell - \ell_1 + 2,k''}}{[k'']} W_{0;k''}^\circ.
\end{equation}
In some sense, the last two terms mix the features of the first two terms already present for maps. As before, we can now define skin maps with tubes by successively gluing elementary skin maps.

\begin{defn}
\label{def:skin:tubes}
	A \emph{skin map with tubes} is an annulus that either has two simple boundaries of the same positive degree rooted at the same edge and no internal faces (\emph{trivial skin}), or is obtained from a sequence of $a \geq 1$ elementary skin maps with tubes $\mathfrak{e}_1,\ldots,\mathfrak{e}_{a}$ such that $\mathfrak{e}_j \in \mathbb{E}^{\circ}_{\ell_j,\ell_{j + 1}}$ (for some $\ell_1,\ldots,\ell_{a + 1} \geq 1$) by gluing $\partial_2\mathfrak{e}_j$ to $\partial_1\mathfrak{e}_{j + 1}$ for every $j \in \set{1,\ldots,a - 1}$. We denote by $\mathbb{S}_{\ell_1,\ell}^{\circ}$ the set of skin maps with tubes with boundaries of respective degrees $\ell_1,\ell$, and define their generating series by
	\begin{equation}
	\label{skin:tubes}
		S_{\ell_1,\ell}^{\circ}
		\coloneqq
		\sum_{\mathfrak{s} \in \mathbb{S}_{\ell_1,\ell}^{\circ}}
			\frac{q^{v(\mathfrak{s}) - \ell}}{\card{\Aut(\mathfrak{s})}}
			\prod_{k \geq 1} t_k^{N_k(\mathfrak{s})}
			\prod_{\substack{k_1 \geq k_2 \geq 0 \\ k_1+k_2>0}} t_{k_1,k_2}^{N_{k_1,k_2}(\mathfrak{s})},
		\qquad
		S^{\circ}(x_1,x)
		\coloneqq
		\sum_{\ell_1,\ell \geq 1} S_{\ell_1,\ell}^{\circ} \frac{x^{\ell - 1}}{x_1^{\ell_1 + 1}}.
	\end{equation}
\end{defn}

The argument of \Cref{ssec:rec:skin} now gives the recursive relations
\begin{equation}
	S_{\ell_1,\ell}^{\circ}
	=
	\delta_{\ell_1,\ell}
	+
	\sum_{\ell' \geq 1} E_{\ell_1,\ell'}^{\circ} S_{\ell',\ell}^{\circ},
	\qquad
	S_{\ell_1,\ell}^{\circ}
	=
	\delta_{\ell_1,\ell}
	+
	\sum_{\ell' \geq 1} S_{\ell_1,\ell'}^{\circ} E_{\ell',\ell}^{\circ}.
\end{equation}
These become explicit in terms of the weights and the disc generating series upon inserting \eqref{eskin:tubes:bij}. In particular, the second equality becomes
\begin{equation}
	\label{second:skinning:tube} S_{\ell_1,\ell}^\circ
	=
	\delta_{\ell_1,\ell}
	+
	\sum_{\ell' \geq 1}
		S_{\ell_1,\ell'}^\circ
		\bigg(
			t_{\ell-\ell'+2}
			+
			2W_{0;\ell'-\ell-2}^\circ
			+
			\sum_{\ell'' \geq 1} \frac{t_{\ell'',\ell}}{\ell}W_{0;\ell'+\ell''-2}^\circ
			+
			\sum_{\ell'' \geq 0} \frac{t_{\ell-\ell'+2,\ell''}}{[\ell'']}W_{0;\ell''}^\circ
		\bigg).
\end{equation}
In generating-series form, this can be written
\begin{multline}
\label{rec:skin:tubes}
	\Bigg[
		S^\circ(x_1,x)
		\biggl(
			x-T_{0,1}'(x)
			-2W^\circ_{0,1}(x)
			-
			\bigl\langle
				\partial_xT_{0,2}(x,\xi)W^\circ_{0,1}(\xi)
			\bigr\rangle_\xi
		\biggr)
	\Bigg]_{+,x} \\
	=
		\left[\frac{1}{x_1 - x} +  \bigl\langle S^{\circ}(x_1,\xi) \partial_{\xi} T_{0,2}(x,\xi) W^{\circ}_{0,1}(\xi) \bigr\rangle_{\xi}\right]_{+,x}.
\end{multline}
In the last term, extracting the positive powers of $x$ amounts to removing the constant term in $x$.

\subsection{Skinning formula and skin enumeration}
\label{ssec:skinning:tubes}
Given a map with tubes $\mathfrak{m}$ of topology $(g,n)\neq(0,1)$, iterating Tutte's procedure until the topology changes produces a pair $(\mathfrak{s},\hat{\mathfrak{m}})$ consisting of a skin map with tubes and a core map with tubes. There is, however, a subtlety in this iteration. In order to apply Tutte again, we need a root on the new boundary produced by the first step, but no such root is provided when an annular face is erased. To address this, we choose a root arbitrarily on the new boundary. Technically, this means replacing in $\mathbb{I}^{=}_{g;\ell_1,\bm{\ell}_I}$ the sets $\mathbb{M}^{\circ}_{g;\underline{k''},\bm{\ell}_I}$ (respectively, $\mathbb{M}^{\circ}_{0;\underline{k''}}$) of maps with unrooted first boundary by arbitrary subsets of $\mathbb{M}^{\circ}_{g;k'',\bm{\ell}_I}$ (respectively, $\mathbb{M}_{0;k''}^{\circ}$) containing a single rooted representative of each map with unrooted first boundary, and making the analogous replacements in $\mathbb{I}^{>}_{g;\ell_1,\bm{\ell}_I}$. The Markovian nature of the skinning process is preserved under any such choice, which is crucial for ensuring that the skin and core can be chosen ``independently''. The erased annular face has a rotational automorphism, which explains why gluing the skin and core back together still yields the inverse bijection described below.

This said, let us return to the description of $\mathfrak{s}$ and $\hat{\mathfrak{m}}$. The boundary $\partial_2\mathfrak{s}$ is simple, but may fail to be rooted, while $\partial_1\hat{\mathfrak{m}}$ is rooted, but may fail to be simple. If $\partial_2\mathfrak{s}$ is rooted, then we can glue it directly to $\partial_1\hat{\mathfrak{m}}$ and recover $\mathfrak{m}$. If $\partial_2\mathfrak{s}$ is not rooted, then it is necessarily a boundary component of an annular face in $\mathfrak{s}$. In that case, after choosing a root on $\partial_2\mathfrak{s}$, we can glue it to $\partial_1\hat{\mathfrak{m}}$; the resulting map is independent of this choice and is again $\mathfrak{m}$. Thus, the map $\mathfrak{m}\mapsto (\mathfrak{s},\hat{\mathfrak{m}})$ is a bijection, and we obtain the skinning relation
\begin{equation}
	W_{g;\ell_1,\bm{\ell}_I}^{\circ}
	=
	\sum_{\ell \geq 1} S^{\circ}_{\ell_1,\ell} \hat{W}^{\circ}_{g;\ell,\bm{\ell}_I}
	\qquad\text{or}\qquad
	W_{g,n}^{\circ}(x_1,\bm{x}_I)
	=
	\Big\langle
		S^{\circ}(x_1,x) \hat{W}^{\circ}_{g,n}(x,\bm{x}_I)
	\Big\rangle_{x}.
\end{equation}
Equivalently, in terms of generating series with free boundary degrees and after inserting \eqref{Wcore:tubes:rec}:

\begin{prop}[Skinning for maps with tubes]
	Let $g \geq 0$, $n \geq 1$ be such that $(g,n) \neq (0,1)$. Then
	\begin{equation}\label{eq:skinning:tubes}
	\begin{split}
		W_{g,n}^{\circ}(x_1,\bm{x}_I)
		&=
		\bigg\langle
			S^{\circ}(x_1,x)
			\frac{\partial_{\xi} T_{0,2}(\xi,\eta)}{x - \xi}
			\biggl(
				W^{\circ}_{g-1,n+1}(\xi,\eta,\bm{x}_I)
				+
				\sum_{\substack{h' + h'' = g \\ J' \sqcup J'' = I}}^{\nodisc}
				W^{\circ}_{h',1+\card{J'}}(\xi,\bm{x}_{J'})W^{\circ}_{h'',1+\card{J''}}(\eta,\bm{x}_{J''})
			\biggr)
		\bigg\rangle_{x,\xi,\eta} \\
		&
		\qquad +
		\Bigg\langle
			S^{\circ}(x_1,x)
			\Biggl(
				\sum_{i=2}^n
				\partial_{x_i}\left(
					\frac{W_{g,n-1}^{\circ}(x,\bm{x}_{I_i}) - W_{g,n-1}^{\circ}(x_i,\bm{x}_{I_i})}{x - x_i}
				\right)
				\\
				&\qquad\qquad\qquad\qquad
				+
				W_{g-1,n+1}^{\circ}(x,x,\bm{x}_I)
				+
				\sum_{\substack{h' + h'' = g \\ J' \sqcup J'' = I}}^{\nodisc}
				W_{h',1+\card{J'}}^{\circ}(x,\bm{x}_{J'})W_{h'',1+\card{J''}}^{\circ}(x,\bm{x}_{J''})
			\Biggr)
		\Bigg\rangle_{x}.
	\end{split}
	\end{equation}
\end{prop}

The skinning of pointed discs works as in \Cref{ssec:skinning}, except that the topology-preserving annulus case (I${}^{=}$) may also occur, adding two stopping cases where the marked vertex is separated off. In the first, the outer boundary contains the marked vertex and has degree $1$: after erasing $r_1$ and the interior of the annular face, only the unrooted and unpointed disc glued to the inner boundary remains. In the second, the inner boundary of the annular face has degree $0$ and consists only of the marked vertex: after erasing $r_1$, we obtain the outer unpointed disc. This gives
\begin{equation}
	\partial_q W_{0;\ell_1}^{\circ}
	=
	\sum_{\ell \geq 1}
		S_{\ell_1,\ell}\left(
			\delta_{\ell,1} t_1
			+
			2W_{0;\ell - 2}^{\circ}
			+
			\delta_{\ell,1} \sum_{m \geq 0} \frac{t_{1,m}}{[m]} W_{0;m}^{\circ}
			+
			\sum_{m \geq 1} t_{m,0} W^{\circ}_{0;\ell + m - 2}
		\right).
\end{equation}
In generating-series form this is
\begin{equation}
\begin{split}
\label{skinning:discs:tubes}
	\partial_q W^\circ_{0,1}(x_1)-\frac{1}{x_1}
	& =
	\bigg[
		S^\circ(x_1,x)\Big(
			T_{0,1}'(x) + 2W^\circ_{0,1}(x)
			+
			\bigl\langle
				\partial_xT_{0,2}(x,\xi)W^\circ_{0,1}(\xi)
			\bigr\rangle_\xi\Big)
	\bigg]_{0,x} \\
	& \quad +
	\big\langle S^{\circ}(x_1,\xi) \partial_{\xi}T_{0,2}(0,\xi)W^\circ_{0,1}(\xi)\big\rangle_{\xi}.		
\end{split}
\end{equation}

\begin{prop}[Skin enumeration for maps with tubes]
\label{skin:enum:maps:tubes}
	We have
	\begin{equation}
		S^{\circ}(x_1,x)
		=
		\frac{
			{\displaystyle\int_{\infty}^{x}}
			\left(
				2W_{0,2}^{\circ}(x_1,y) + \frac{1}{(x_1 - y)^2}
			\right) \dd y
			+
			\big\langle W_{0,2}^{\circ}(x_1,\xi) T_{0,2}(\xi,x) \big\rangle_{\xi}
			-
			\partial_q W_{0,1}^{\circ}(x_1)
		}{
			x - \partial_{x} T_{0,1}(x) - \big\langle \partial_{x} T_{0,2}(x,\xi) W_{0,1}^{\circ}(\xi)\big\rangle_{\xi} - 2W_{0,1}^{\circ}(x) 
		},
	\end{equation}
	where again $\int_{\infty}^x \frac{1}{(x_1 - y)^2} \dd y = \frac{1}{x_1 - x}$, and the result is expanded as in~\eqref{x1:x2:expns}.
\end{prop}

\begin{proof}
	Set
	\begin{equation}
		D^{\circ}(x)
		\coloneqq
		x-T_{0,1}'(x)
		-
		\bigl\langle
			\partial_xT_{0,2}(x,\xi)W^\circ_{0,1}(\xi) 
		\bigr\rangle_\xi -2W^\circ_{0,1}(x).
	\end{equation}
	As in the proof of \Cref{skin:enum:maps}, we determine separately the positive, negative, and constant parts of $S^\circ(x_1,x)D^\circ(x)$ as a Laurent series in $x$. First, the recursive construction of skin maps with tubes gives the positive part, see \eqref{rec:skin:tubes}:
	\begin{equation}
	\label{tubes:positive}
		\bigl[
			S^\circ(x_1,x)D^\circ(x)
		\bigr]_{+,x}
		=
		\left[
			\frac{1}{x_1-x}
			+
			\bigl\langle S^{\circ}(x_1,\xi) \partial_{\xi} T_{0,2}(x,\xi) W_{0,1}^{\circ}(\xi)\big\rangle_{\xi}\right]_{+,x}.
	\end{equation}
	A novelty for maps with tubes is the new term involving $S^{\circ}$ on the right-hand side, which we would like to eliminate. To do so, we transform it with the help of the triple product identity of \Cref{lem:triple:prod}:
	\begin{equation}
	\label{unwanted}
		\big\langle S^{\circ}(x_1,\xi) \partial_{\xi} T_{0,2}(x,\xi) W_{0,1}^{\circ}(\xi) \big\rangle_{\xi}
		=
		\Bigg\langle T_{0,2}(x,\xi) \Bigg\langle \frac{S^{\circ}(x_1,\eta)W_{0,1}^{\circ}(\eta)}{(\xi - \eta)^2}\Bigg\rangle_{\eta} \Bigg\rangle_{\xi}.
	\end{equation}
	We can then simplify the innermost bracket using the skinning relation for annuli with tubes. Indeed, for $(g,n) = (0,2)$ only the second line survives in the skinning relation, \Cref{eq:skinning:tubes}, and we find
	\begin{equation}
		W_{0,2}^{\circ}(x_1,\xi)
		=
		\Bigg\langle S^{\circ}(x_1,\eta) \partial_{\xi}\left(\frac{W_{0,1}^{\circ}(\eta) - W_{0,1}^{\circ}(\xi)}{\eta - \xi}\right)\Bigg\rangle_{\eta}
		=
		\Bigg\langle S^{\circ}(x_1,\eta) \partial_{\xi}\left(\frac{W_{0,1}^{\circ}(\xi) - W_{0,1}^{\circ}(\eta)}{\xi - \eta}\right)\Bigg\rangle_{\eta},
	\end{equation}
	where we used the fact that the ratio is a formal series in $\xi^{-1}$ and $\eta^{-1}$ symmetric in the two variables. The advantage of the second form is that we can split the ratio into two pieces and, recalling the convention that $\frac{1}{\xi - \eta}$ is expanded in non-negative powers of $\eta$, the piece involving $W_{0,1}^{\circ}(\xi)$ does not contribute to the bracket because $S^\circ(x_1,\eta)$ has only non-negative powers of $\eta$. Hence,
	\begin{equation}
		W_{0,2}^{\circ}(x_1,\xi)
		=
		\Bigg \langle S^{\circ}(x_1,\eta)\frac{W_{0,1}^{\circ}(\eta)}{(\xi - \eta)^2}\Bigg\rangle_{\eta}.
	\end{equation}
	Inserting this into \eqref{unwanted} yields
	\begin{equation}
	\label{STW:TW}
		\big\langle S^{\circ}(x_1,\xi) \partial_{\xi} T_{0,2}(x,\xi) W_{0,1}^{\circ}(\xi) \big\rangle_{\xi}
		=
		\big\langle W_{0,2}^{\circ}(x_1,\xi) T_{0,2}(\xi,x) \big\rangle_{\xi}.
	\end{equation}
	This allows us to write the positive part of $S^\circ(x_1,x)D^\circ(x)$, \Cref{tubes:positive}, in the form
	\begin{equation}
	\label{tubes:positive2}
		\bigl[S^\circ(x_1,x)D^\circ(x)\bigr]_{+,x}
		=
		\left[\frac{1}{x_1 - x} + \big\langle W_{0,2}^{\circ}(x_1,\xi) T_{0,2}(\xi,x) \big\rangle_{\xi}\right]_{+,x}.
	\end{equation}
	Second, the constant term is given by
	\begin{equation}
		\bigl[
				S^\circ(x_1,x)D^\circ(x)
		\bigr]_{0,x}
		=
		-\Bigg[
			S^\circ(x_1,x)
			\Big(
				T_{0,1}'(x) 
				+
				\bigl\langle
					\partial_x T_{0,2}(x,\xi)W^\circ_{0,1}(\xi)
				\bigr\rangle_\xi + 2W^\circ_{0,1}(x)
			\Big)
		\Bigg]_{0,x}.
	\end{equation}
	We recognise part of the expression in the skinning relation for pointed discs with tubes \eqref{skinning:discs:tubes}:
	\begin{equation}
	\label{tubes:constant}
	\begin{split}
		\bigl[
			S^\circ(x_1,x)D^\circ(x)
		\bigr]_{0,x}
		&=
		\frac{1}{x_1} - \partial_q W_{0,1}^{\circ}(x_1)
		+
		\big\langle S^{\circ}(x_1,\xi) \partial_{\xi}T_{0,2}(0,\xi) W_{0,1}^{\circ}(\xi) \big\rangle_{\xi} \\
		&=
		\left[\frac{1}{x_1 - x} + \big\langle W_{0,2}^{\circ}(x_1,\xi) T_{0,2}(\xi,x) \big\rangle_{\xi}\right]_{0,x}  - \partial_q W_{0,1}^{\circ}(x_1),
	\end{split}
	\end{equation}
	where we used the transformation \eqref{STW:TW} a second time.
	
	It remains to determine the negative part. This follows from the skinning of annuli with tubes and is unchanged from ordinary maps, since the annulus skinning relation has the same form:
	\begin{equation}
	\label{tubes:negative}
		\bigl[
			S^\circ(x_1,x)D^\circ(x)
		\bigr]_{-,x}
		=
		\int_\infty^x 2W^\circ_{0,2}(x_1,y)\dd y.
	\end{equation}
	Assembling the positive part \eqref{tubes:positive2}, the constant term \eqref{tubes:constant}, and the negative part \eqref{tubes:negative} reconstructs the full series in the form
	\begin{equation}
		S^\circ(x_1,x)D^\circ(x)
		=
		\int_\infty^x 2W^\circ_{0,2}(x_1,y)\dd y
		+
		\frac{1}{x_1-x}
		+
		\bigl\langle
			T_{0,2}(x,\xi)W^\circ_{0,2}(x_1,\xi)
		\bigr\rangle_\xi
		-
		\partial_q W^\circ_{0,1}(x_1).
	\end{equation}
	Pulling $\frac{1}{x_1-x}$ inside the integral and dividing by $D^\circ(x)$ gives the result.
\end{proof}

\section{Generalisation to stuffed maps}
\label{sec:stuffed:maps}

The last generalisation we consider is that of \emph{stuffed maps}, in which internal faces may carry arbitrary topology. The overall picture remains similar, but with additional features. Although the formulae, which now have to keep track of arbitrary face topologies, become more cumbersome, the underlying logic for handling the various cases becomes clearer.

\subsection{Definition and generating series}
\label{ssec:stuffed:defn}
Stuffed maps are defined in the same way as maps, except that we allow internal faces of arbitrary topology \cite{Bor14}. The degree of an internal face of topology $(h,m)$ is the multiset $\set{k_1,\ldots,k_m}$ of the degrees of its boundary components, and we only allow $\set{k_1,\ldots,k_m} \neq \set{0,\ldots,0}$. We use Boltzmann weights $t_{h;k_1,\ldots,k_m}$ symmetric in the $k_i$, and set $t_{h;k_1,\ldots,k_m}=0$ whenever $\set{k_1,\ldots,k_m} = \set{0,\ldots,0}$ or one of the $k_i$ is negative.

We denote by $\MM_{g;\ell_1,\ldots,\ell_n}^{\ast}$ the set of stuffed maps of topology $(g,n)$ with respective boundary degrees $\ell_1,\ldots,\ell_n$. The trivial map is the only one with a boundary of degree $0$, so $\MM_{0;0}^{\ast}$ is a singleton. For a stuffed map $\mathfrak{m}$, we set
\begin{equation}
	N_{h;k_1,\ldots,k_m}(\mathfrak{m})
	\coloneqq
	\card*{\Set{\textnormal{internal faces of topology $(h,m)$ and degree $\{k_1,\ldots,k_m\}$}}}.
\end{equation}
We then define the generating series of stuffed maps of topology $(g,n)$ with fixed boundary degrees $\ell_1,\ldots,\ell_n$ by
\begin{equation}
	W_{g;\ell_1,\ldots,\ell_n}^{\ast}
	\coloneqq
	\sum_{\mathfrak{m} \in \MM_{g;\ell_1,\ldots,\ell_n}^{\ast}}
	\frac{q^{v(\mathfrak{m})}}{\card{\Aut(\mathfrak{m})}}
	\prod_{\substack{h\geq 0 \\ m\geq 1}}
	\prod_{\substack{k_1 \geq \cdots \geq k_m \geq 0 \\ k_1+\cdots+k_m>0}}
		t_{h;k_1,\ldots,k_m}^{N_{h;k_1,\ldots,k_m}(\mathfrak{m})},
\end{equation}
and, with free boundary degrees, by
\begin{equation}
	W_{g,n}^{\ast}(x_1,\ldots,x_n)
	\coloneqq
	\delta_{g,0}\delta_{n,1}\frac{q}{x_1}
	+
	\sum_{\ell_1,\ldots,\ell_n \geq 1}
	\frac{W_{g;\ell_1,\ldots,\ell_n}^{\ast}}{x_1^{\ell_1 + 1} \cdots x_n^{\ell_n + 1}}.
\end{equation}
Maps with tubes are recovered by setting to zero Boltzmann weights with $2h-2+m>0$. Maps correspond to the further specialisation in which only the weights $t_{0;k}=t_k$ may be non-zero.

For each topology $(h,m)$ we define the corresponding potential by
\begin{equation}
	T_{h,m}(x_1,\ldots,x_m)
	\coloneqq
	\sum_{\substack{k_1,\ldots,k_m\geq 0 \\ k_1 + \cdots + k_m > 0}}
	\frac{t_{h;k_1,\ldots,k_m}}{[k_1]\cdots [k_m]}
	x_1^{k_1}\cdots x_m^{k_m}.
\end{equation}

\subsection{Tutte's procedure}
\label{ssec:Tutte:stuffed}
Tutte's procedure for stuffed maps follows the same case distinction as for maps with tubes. The cases where $r_1$ bounds only boundary faces are unchanged, giving (R), (D${}^{=}$), and (D${}^{>}$). If $r_1$ bounds an internal face, we again distinguish the topology-preserving and topology-changing cases (I${}^{=}$) and (I${}^{>}$). The only new feature is that the internal face may now have arbitrary topology.

Let us describe the internal-face cases in more detail. Suppose that the edge supporting $\rho_1$ belongs both to $\partial_1$ and to a boundary $\delta_1$ of degree $k_1\geq 1$ of an internal face $\mathfrak{f}$ of topology $(h,m)$. The face $\mathfrak{f}$ may have further boundary components $\delta_2,\ldots,\delta_m$ of degrees $k_2,\ldots,k_m$, which are neither rooted nor ordered. We use this notation throughout this subsection for the boundary components of the internal face. When erasing $r_1$, we also erase the interior of $\mathfrak{f}$. The result is a collection of maps $\mathfrak{m}_1,\ldots,\mathfrak{m}_c$ for some $c\in\set{1,\ldots,m}$. If $k_1 \geq 2$ the \emph{outer map} $\mathfrak{m}_1$ is the one containing $\overline{\rho_1}{}^{-}$; this half-edge is declared to be the new root of the new boundary, of degree $\ell_1+k_1-2$. If $k_1 = 1$, we are forced to have $\ell_1 = 1$ and the outer map $\mathfrak{m}_1$ is the trivial disc. Despite the notation, the remaining maps $\mathfrak{m}_2,\ldots,\mathfrak{m}_c$ do not come with a preferred ordering.

To describe the topology of the components $\mathfrak{m}_1,\ldots,\mathfrak{m}_c$, we introduce a set $\mathcal{T}_{g,n;h,m;c}$ of triples $(\bm{p},\bm{K},\bm{J})$ where: the tuple $\bm{p}=(p_1,\ldots,p_c)$ records the genera of the components, the partition $\bm{K}=(K_1,\ldots,K_c)$ records how the boundaries $\delta_2,\ldots,\delta_m$ of $\mathfrak{f}$ are distributed, and the partition $\bm{J}=(J_1,\ldots,J_c)$ records how the original boundaries $\partial_2,\ldots,\partial_n$ are distributed among the components. More precisely:
\begin{itemize}
	\item $\bm{p} = (p_1,\ldots,p_c)$ is a $c$-tuple of non-negative integers satisfying the genus constraint
	\begin{equation}
	\label{tot:genus}
		g = h + \sum_{b=1}^{c} \bigl(p_b + \card{K_b} - 1\bigr).
	\end{equation}
	\item $\bm{K} = (K_1,\ldots,K_c)$ is a partition of $\set{2,\ldots,m}$ into $c$ subsets, all non-empty except possibly $K_1$.
	\item $\bm{J} = (J_1,\ldots,J_c)$ is a partition of $\set{2,\ldots,n}$ into $c$ subsets, which may be empty.
	\end{itemize}
Then, for each $b \in \set{2,\ldots,c}$, the map $\mathfrak{m}_b$ has genus $p_b$, contains the new unrooted boundaries $\bm{\delta}_{K_b}$, and also the original boundaries $\bm{\partial}_{J_b}$. The outer map $\mathfrak{m}_1$ has genus $p_1$, new unrooted and unordered boundaries $\bm{\delta}_{K_1}$, as well as the new rooted boundary and the original boundaries $\bm{\partial}_{J_1}$. We give an example in \Cref{fig:stuffed}. The genus constraint \eqref{tot:genus} expresses that, after gluing back the face $\mathfrak{f}$ of topology $(h,m)$, one recovers a stuffed map of genus $g$. We denote by $\mathcal{T}^{>}_{g,n;h,m;c}$ the subset of $\mathcal{T}_{g,n;h,m;c}$ that can occur when none of the maps $\mathfrak{m}_1,\ldots,\mathfrak{m}_c$ has topology $(g,n)$.

\begin{figure}
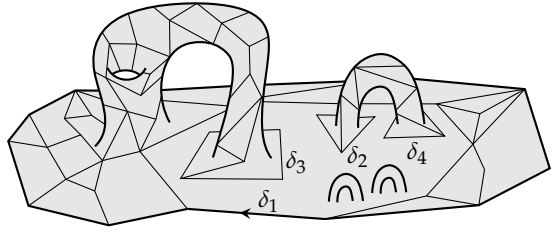

\centering

	\caption{A stuffed map $\mathfrak{m}$ of topology $(g,n)=(5,1)$, and the bookkeeping after deleting the root-edge bounding an internal face $\mathfrak{f}$. Here $\mathfrak{f}$ has type $(h,m)=(2,4)$, and its deletion produces $c=2$ components of genera $p_1=1$ and $p_2=0$. The boundary components of $\mathfrak{f}$ are distributed as $K_1=\set{3}$ and $K_2=\set{2,4}$, and since there are no boundary components besides $\partial_1$, we have $J_1=J_2=\varnothing$.}
	\label{fig:stuffed}
\end{figure}

With this notation, we make the following case distinction.

\begin{itemize}
	\item[(I${}^{=}$)] $r_1$ bounds an internal face, and one of the maps $\mathfrak{m}_{b_0}$, for some $b_0 \in \set{1,\ldots,c}$, has topology $(g,n)$, while all the others have disc topology, and moreover $h=0$. In that case, either the outer map $\mathfrak{m}_1$ has topology $(g,n)$, or one of the other components does. In the latter case,  up to relabelling $\delta_2,\ldots,\delta_m$ we may assume that $\mathfrak{m}_{b_0 = 2}$ is the component of topology $(g,n)$.

	\item[(I${}^{>}$)] $r_1$ bounds an internal face, and none of the maps $\mathfrak{m}_1,\ldots,\mathfrak{m}_c$ has topology $(g,n)$.

	\item[(R)] $r_1$ bounds $\partial_i$ for some $i \in I$.

	\item[(D${}^{=}$)] $r_1$ bounds $\partial_1$ on both sides, and erasing it produces two maps, one of which has disc topology.

	\item[(D${}^{>}$)] $r_1$ bounds $\partial_1$ on both sides, and erasing it either produces one map or produces two maps, neither of which has disc topology.
\end{itemize}

The cases (I${}^{=}$) and (D${}^{=}$) are topology-preserving, whereas the cases (I${}^{>}$), (R), and (D${}^{>}$) are topology-changing. Compared with the case of maps with tubes (\Cref{ssec:Tutte:tubes}), the topology-preserving internal-face cases are now subsumed into (I${}^{=}$), while the topology-changing internal-face cases are subsumed into (I${}^{>}$). Altogether, we obtain a bijection
\begin{equation}
	\Tutte \colon
	\MM_{g;\ell_1,\bm{\ell}_I}^{\ast}
	\overset{\simeq}{\longrightarrow}
	\mathbb{I}^{=}_{g;\ell_1,\bm{\ell}_I}
	\sqcup
	\mathbb{I}^{>}_{g;\ell_1,\bm{\ell}_I}
	\sqcup
	\mathbb{R}_{g;\ell_1,\bm{\ell}_I}
	\sqcup
	\mathbb{D}_{g;\ell_1,\bm{\ell}_I}^{=}
	\sqcup
	\mathbb{D}_{g;\ell_1,\bm{\ell}_I}^{>}.
\end{equation}
Here
\begin{equation}
\begin{split}
	\mathbb{I}_{g;\ell_1,\bm{\ell}_I}^{=}
	&\coloneqq
	\bigsqcup_{m \geq 1}
	\Biggl(\Biggl(
	\bigsqcup_{\substack{k_1 \geq 1 \\ k_2,\ldots,k_m \geq 0}}
	\MM^{\ast}_{g;\ell_1 + k_1 - 2,\bm{\ell}_I}
	\times
	\prod_{b = 2}^{m} \MM^{\ast}_{0;\underline{k}_b}
	\Biggr)^{\sim} \\
	&\qquad\qquad
	\sqcup
	\Biggl(
	\bigsqcup_{\substack{k_1,k_2 \geq 1 \\ k_3,\ldots,k_m \geq 0}}
	\MM_{0;\ell_1 + k_1 - 2}^{\ast}
	\times
	\MM^{\ast}_{g;\underline{k}_2,\bm{\ell}_{I}}
	\times
	\prod_{b = 3}^{m} \MM^{\ast}_{0;\underline{k}_b}
	\Biggr)^{\sim}
	\Biggr), \\
	\mathbb{I}_{g;\ell_1,\bm{\ell}_I}^{>}
	&\coloneqq
	\bigsqcup_{\substack{h \geq 0 \\ m \geq 1}}
	\bigsqcup_{\substack{k_1 \geq 1 \\ k_2,\ldots,k_m \geq 0}}
	\Biggl(
	\bigsqcup_{\substack{1 \leq c \leq m \\ (\bm{p},\bm{K},\bm{J}) \in \mathcal{T}^{>}_{g,n;h,m;c}}}
	\MM^{\ast}_{p_1;\ell_1 + k_1 - 2,\bm{\ell}_{J_1},\underline{\bm{k}}_{K_1}}
	\times
	\prod_{b = 2}^{c} \MM^{\ast}_{p_b;\bm{\ell}_{J_b},\underline{\bm{k}}_{K_b}}
	\Biggr)^{\sim}, \\
	\mathbb{R}_{g;\ell_1,\bm{\ell}_I}
	&\coloneqq
	\bigsqcup_{i = 2}^{n} \MM^{\ast}_{g;\ell_1 + \ell_i - 2,\bm{\ell}_{I_i}} \times \set{0,\ldots,\ell_i - 1}, \\
	\mathbb{D}^=_{g;\ell_1,\bm{\ell}_I}
	&\coloneqq
	\bigsqcup_{\substack{k',k'' \geq 1 \\ k' + k'' = \ell_1 - 2}} \Biggl(
		\MM^{\ast}_{g - 1;k',k'',\bm{\ell}_I}
		\sqcup  \;
		\bigsqcup_{\substack{h' + h'' = g \\ J' \sqcup J'' = I}}^{\nodisc}
		\MM^{\ast}_{h';k',\bm{\ell}_{J'}} \times \MM^{\ast}_{h'';k'',\bm{\ell}_{J''}}
	\Biggr), \\
	\mathbb{D}_{g;\ell_1,\bm{\ell}_{I}}^{>}
	&\coloneqq
	\bigsqcup_{k = 0}^{\ell_1 - 3}
		\bigl(\MM^{\ast}_{0;k} \times \MM_{g;\ell_1 - k - 2,\bm{\ell}_I}^{\ast}\bigr)
		\sqcup
		\bigl(\MM^{\ast}_{g;\ell_1  - k - 2,\bm{\ell}_I} \times \MM_{0;k}^{\ast}\bigr).
\end{split}
\end{equation}
An index $\underline{k}_i$ means that we consider a quotient set in which the corresponding boundary of degree $k_i$ is not rooted. The symbol $\sim$ indicates that we quotient by the permutation of the indices $2,\ldots,m$ in the first contribution to (I${}^{=}$), by the permutation of the indices $3,\ldots,m$ in the second contribution to (I${}^{=}$), and by the permutation of the indices $2,\ldots,c$ in the contribution to (I${}^{>}$). Note that the two separate contributions to (I${}^{=}$) reflect whether or not $\mathfrak{m}_1$ is a disc. This translates into the following relations at the level of generating series, cf.~\cite{Bor14}.

\begin{lem}
\label{Tutte:lemma:stuffed}
	Let $g \geq 0$, $n \geq 1$, and $\ell_1,\ldots,\ell_n \geq 1$ be such that $(g,n) \neq (0,1)$. We have
	\begin{multline}
		W^{\ast}_{g;\ell_1,\bm{\ell}_I}
		=
		\sum_{\substack{h \geq 0 \\ m \geq 1}}
		\sum_{\substack{k_1 \geq 1 \\ k_2,\ldots,k_m \geq 0}}
		\sum_{\substack{1 \leq c \leq m \\ (\bm{p},\bm{K},\bm{J}) \in \mathcal{T}_{g,n;h,m;c}}}
		\frac{t_{h;k_1,\ldots,k_m}}{[k_2] \cdots [k_m]}
		\frac{W^{\ast}_{p_1;\ell_1 + k_1 - 2,\bm{k}_{K_1},\bm{\ell}_{J_1}}}{(c - 1)!}
		\prod_{b = 2}^{c} W^{\ast}_{p_b;\bm{k}_{K_b},\bm{\ell}_{J_b}} \\
		+
		\sum_{i=2}^n \ell_i \, W^{\ast}_{g;\ell_1+\ell_i-2,\bm{\ell}_{I_i}}
		+
		\sum_{\substack{k',k'' \geq 0 \\ k' + k'' = \ell_1 - 2}}
		\Biggl(
		W^{\ast}_{g-1;k',k'',\bm{\ell}_{I}}
		+
		\sum_{\substack{h' + h'' = g \\ J' \sqcup J'' = I}}
		W^{\ast}_{h';k',\bm{\ell}_{J'}}W^{\ast}_{h'';k'',\bm{\ell}_{J''}}
		\Biggr).
	\end{multline}
	For generating series with free boundary degrees this translates into
	\begin{equation}
	\label{Tutte:rec:stuffed}
	\begin{split}
		&
		W^{\ast}_{g,n}(x_1,\bm{x}_I) \\
		&\quad
		=
		\sum_{\substack{h \geq 0 \\ m \geq 1}}
		\sum_{\substack{1 \leq c \leq m \\ (\bm{p},\bm{K},\bm{J}) \in \mathcal{T}_{g,n;h,m;c}}}
		\Bigg\langle
		\frac{\partial_{\xi_1} T_{h,m}(\xi_1,\ldots,\xi_m)}{(c - 1)!(x_1 - \xi_1)}
		W^{\ast}_{p_1,1 + \card{K_1} + \card{J_1}}(\xi_1,\bm{\xi}_{K_1},\bm{x}_{J_1})
		\prod_{b = 2}^c W^{\ast}_{p_b,\card{K_b} + \card{J_b}}(\bm{\xi}_{K_b},\bm{x}_{J_b})
		\Bigg\rangle_{\bm{\xi}} \\
		&\qquad
		+
		\sum_{i=2}^n
		\partial_{x_i}\left(
		\frac{W^{\ast}_{g,n-1}(x_1,\bm{x}_{I_i}) - W^{\ast}_{g,n-1}(x_i,\bm{x}_{I_i})}{x_1 - x_i}
		\right) \\
		&\qquad
		+
		W^{\ast}_{g-1,n+1}(x_1,x_1,\bm{x}_I)
		+
		\sum_{\substack{h' + h'' = g \\ J' \sqcup J'' = I}}
		W^{\ast}_{h',1+\card{J'}}(x_1,\bm{x}_{J'})W^{\ast}_{h'',1+\card{J''}}(x_1,\bm{x}_{J''}).
	\end{split}
	\end{equation}
\end{lem}

In the right-hand side of \eqref{Tutte:rec:stuffed}, the first line collects the contributions of type (I), the second those of type (R), and the third those of type (D). Unlike for maps and maps with tubes, the quadratic term in the third line is summed without the superscript ``no $(0,1)$'': the contributions of type (D${}^{>}$) are not absorbed into the first line, and hence remain in the quadratic term. For the same reason $k' = 0$ or $k'' = 0$ is allowed, corresponding to contributions of trivial discs.

\subsection{Core and skin maps}
As in the previous sections, we now separate topology-changing and topology-preserving cases of Tutte's procedure, and use this distinction to define core maps and skin maps. In the stuffed setting, however, the detached piece may already have arbitrary topology, so the corresponding bookkeeping becomes longer.

\begin{defn}
\label{core:stuffed}
	Let $g \geq 0$, $n \geq 1$, and $\ell_1,\ldots,\ell_n \geq 1$ be such that $(g,n) \neq (0,1)$. We denote by $\hat{\MM}^{\ast}_{g;\ell_1,\bm{\ell}_I}$ the set of stuffed maps $\hat{\mathfrak{m}}$ of topology $(g,n)$ with boundary degrees $\ell_1,\bm{\ell}_I$ such that applying Tutte's procedure changes the topology, that is,
	\begin{equation}
		\Tutte(\hat{\mathfrak{m}})
		\in
		\mathbb{I}^{>}_{g;\ell_1,\bm{\ell}_I}
		\sqcup
		\mathbb{R}_{g;\ell_1,\bm{\ell}_I}
		\sqcup
		\mathbb{D}^{>}_{g;\ell_1,\bm{\ell}_I}.
	\end{equation}
	These are the \emph{stuffed core maps}. We introduce their generating series
	\begin{equation}
	\begin{split}
		\hat{W}_{g;\ell_1,\bm{\ell}_I}^{\ast}
		&\coloneqq
		\sum_{\hat{\mathfrak{m}} \in \hat{\MM}^{\ast}_{g;\ell_1,\ldots,\ell_n}}
			\frac{q^{v(\hat{\mathfrak{m}})}}{\card{\Aut(\hat{\mathfrak{m}})}}
			\prod_{\substack{h\geq 0 \\ m\geq 1}}
			\prod_{\substack{k_1 \geq \cdots \geq k_m \geq 0 \\ k_1 + \cdots + k_m > 0}}
				t_{h;k_1,\ldots,k_m}^{N_{h;k_1,\ldots,k_m}(\hat{\mathfrak{m}})}, \\
		\hat{W}_{g,n}^{\ast}(x_1,\bm{x}_I)
		&\coloneqq
		\sum_{\ell_1,\ldots,\ell_n \geq 1}
			\frac{\hat{W}_{g;\ell_1,\bm{\ell}_I}^{\ast}}{x_1^{\ell_1}x_2^{\ell_2 + 1} \cdots x_n^{\ell_n + 1}}.
	\end{split}
	\end{equation}
\end{defn}

Keeping only the topology-changing terms on the right-hand side of \eqref{Tutte:rec:stuffed}, we find
\begin{equation}
\begin{split}
	&
	\hat{W}_{g,n}^{\ast}(x_1,\bm{x}_I) \\
	&\quad =
	\sum_{\substack{h \geq 0 \\ m \geq 1}}
	\sum_{\substack{1 \leq c \leq m \\ (\bm{p},\bm{K},\bm{J}) \in \mathcal{T}_{g,n;h,m;c}^{>}}}
	\Bigg\langle
	\frac{\partial_{\xi_1} T_{h,m}(\xi_1,\ldots,\xi_m)}{(c - 1)!(x_1 - \xi_1)}
	W^{\ast}_{p_1,1 + \card{K_1} + \card{J_1}}(\xi_1,\bm{\xi}_{K_1},\bm{x}_{J_1})
	\prod_{b = 2}^c W^{\ast}_{p_b,\card{K_b} + \card{J_b}}(\bm{\xi}_{K_b},\bm{x}_{J_b})
	\Bigg\rangle_{\bm{\xi}} \\
	&\qquad
	+
	\sum_{i=2}^n
	\partial_{x_i}\left(
	\frac{W^{\ast}_{g,n-1}(x_1,\bm{x}_{I_i}) - W^{\ast}_{g,n-1}(x_i,\bm{x}_{I_i})}{x_1 - x_i}
	\right) \\
	&\qquad
	+
	W^{\ast}_{g-1,n+1}(x_1,x_1,\bm{x}_I)
	+
	\sum_{\substack{h' + h'' = g \\ J' \sqcup J'' = I}}^{\nodisc}
	W^{\ast}_{h',1+\card{J'}}(x_1,\bm{x}_{J'})W^{\ast}_{h'',1+\card{J''}}(x_1,\bm{x}_{J''}).
\end{split}
\end{equation}

We define \emph{stuffed elementary skin maps} by considering each case of Tutte's procedure. In cases (R) and (D), the construction is the same as in \Cref{ssec:skin}. The novelty lies in cases (I${}^{=}$) and (I${}^{>}$). In these cases, the elementary skin map has topology $(h,m+1-d)$, where $d \in \set{1,\ldots,m}$ is the number of maps $\mathfrak{m}_{j_1},\ldots,\mathfrak{m}_{j_d}$ among $\mathfrak{m}_1,\ldots,\mathfrak{m}_c$ that have disc topology. We first construct an auxiliary map $\mathfrak{a}$ of topology $(h,m+1)$ with one internal face of topology $(h,m)$. The boundary $\partial_1\mathfrak{a}$ has degree $\ell_1$, root $\rho_1$, and is simple. The simple boundary $\partial_2\mathfrak{a}$ has degree $\ell_1+k_1-2$ and root $\bbarsuperscript{\rho_1}{-}$. The $\ell_1-1$ edges preceding this root are glued to the $\ell_1-1$ edges following $\rho_1$. The remaining boundaries $\partial_3\mathfrak{a},\ldots,\partial_{m+1}\mathfrak{a}$ are unrooted and have respective degrees $k_2,\ldots,k_m$. Furthermore, we glue the internal face $\mathfrak{f}$ as follows: one half-edge of $\delta_1$ is glued to $\rho_1$, while its next $k_1-1$ half-edges are glued to the root of $\partial_2\mathfrak{a}$ and the $k_1-2$ half-edges following it. For each $i \in \set{2,\ldots,m}$, we glue $\delta_i$ to $\partial_{i+1}\mathfrak{a}$.

We then obtain the elementary skin map $\mathfrak{e}$ from $\mathfrak{a}$ by declaring\footnote{
	Note that, with this convention, the boundary labels may not be consecutive.
} $\partial_i\mathfrak{e}=\partial_i\mathfrak{a}$ for all $i \in \set{2,\ldots,m} \setminus \set{j_1+1,\ldots,j_d+1}$, and by gluing the boundary faces of $d$ discs of respective degrees $k_{j_1},\ldots,k_{j_d}$ to the boundaries $\partial_{j_1+1}\mathfrak{a},\ldots,\partial_{j_d+1}\mathfrak{a}$; see~\Cref{fig:stuffed:elementary}. Since these boundaries are unrooted, roots must be chosen for the gluing, but the resulting map is independent of these choices. The absence of a preferred ordering of $\delta_2,\ldots,\delta_m$ translates into the absence of a preferred ordering of the boundaries of $\mathfrak{e}$ other than $\partial_1\mathfrak{e}$ and $\partial_2\mathfrak{e}$. A disc may or may not be glued into $\partial_2\mathfrak{a}$, and this distinction will be visible in several equations, such as~\eqref{eskin:01:stuffed}.

In the case (I${}^{=}$), the map $\mathfrak{e}$ is an annulus, since then $d=m-1$ and $h=0$. In the case (I${}^{>}$), $\mathfrak{e}$ has negative Euler characteristic.

\begin{figure}
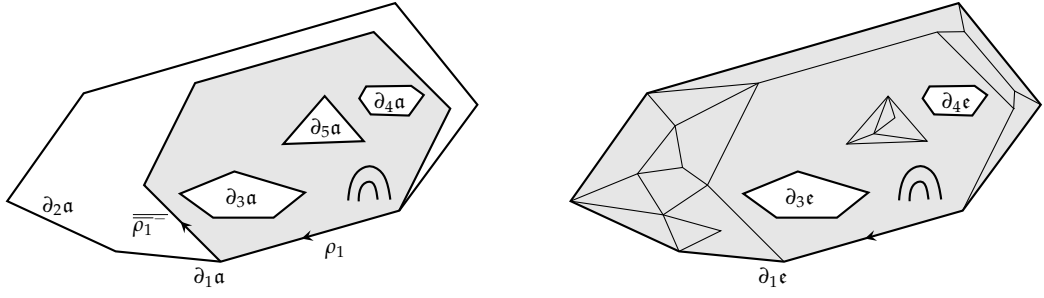

	\centering

	\caption{Construction of the stuffed elementary skin map $\mathfrak{e}$ in the internal-face cases. In the displayed situation, the map $\mathfrak{e}$ (on the right) has negative Euler characteristic. It is obtained from $\mathfrak{a}$ (on the left) with $d=2$ and $j_1=1,j_2=4$ by gluing the discs $\mathfrak{m}_1$ and $\mathfrak{m}_4$ into $\partial_2\mathfrak{a}$ and $\partial_5\mathfrak{a}$.}
	\label{fig:stuffed:elementary}
\end{figure}

\begin{defn}
	Let $\ell_1,\ell \geq 1$. We denote by $\mathbb{E}_{\ell_1,\ell}^{\ast}$ the set of elementary skin maps $\mathfrak{e}$ with boundary degrees $\ell_1$ and $\ell$ arising from the topology-preserving cases (I${}^{=}$) or (D${}^{=}$). We define their generating series by
	\begin{equation}
		E_{\ell_1,\ell}^{\ast}
		\coloneqq
		\sum_{\mathfrak{e} \in \mathbb{E}_{\ell_1,\ell}^{\ast}}
		\frac{q^{v(\mathfrak{e}) - \ell}}{\card{\Aut(\mathfrak{e})}}
		\prod_{\substack{h\geq 0 \\ m \geq 1}}
		\prod_{\substack{k_1 \geq \cdots \geq k_m \geq 0 \\ k_1 + \cdots + k_m > 0}}
		t_{h;k_1,\ldots,k_m}^{N_{h;k_1,\ldots,k_m}(\mathfrak{e})}.
	\end{equation}
\end{defn}

Our construction implies
\begin{multline}
\label{eskin:01:stuffed}
	E^\ast_{\ell_1,\ell}
	=
	2W_{0;\ell_1 - \ell - 2}^{\ast}
	+
	\sum_{m \geq 1}
	\sum_{k_2,\ldots,k_m \geq 0}
	\frac{1}{(m - 1)!} \frac{t_{0;\ell - \ell_1 + 2,k_2,\ldots,k_m}}{[k_2] \cdots [k_m]}
	\prod_{b = 2}^{m} W_{0;k_b}^{\ast} \\
	+
	\sum_{m \geq 2}
	\sum_{\substack{k_1 \geq 1 \\ k_3,\ldots,k_m \geq 0}}
	\frac{1}{(m - 2)!} \frac{t_{0;k_1,\ell,k_3,\ldots,k_m}}{\ell\,[k_3]\cdots [k_m]}
	W_{0;\ell_1 + k_1 - 2}^{\ast}
	\prod_{b = 3}^{m} W_{0;k_b}^{\ast}.
\end{multline}

\begin{defn}
	A \emph{stuffed skin map} is either the trivial skin, or is obtained from a sequence of stuffed elementary skin maps $\mathfrak{e}_1,\ldots,\mathfrak{e}_a$, with $a \geq 1$, such that $\mathfrak{e}_j \in \mathbb{E}^{\ast}_{\ell_j,\ell_{j + 1}}$ for some $\ell_1,\ldots,\ell_{a + 1} \geq 1$, by gluing $\partial_2\mathfrak{e}_j$ to $\partial_1\mathfrak{e}_{j + 1}$ for every $j \in \set{1,\ldots,a-1}$. We denote by $\mathbb{S}_{\ell_1,\ell}^{\ast}$ the set of stuffed skin maps with boundaries of respective degrees $\ell_1,\ell$, and introduce their generating series
	\begin{equation}
		S_{\ell_1,\ell}^{\ast}
		\coloneqq
		\sum_{\mathfrak{s} \in \mathbb{S}_{\ell_1,\ell}^{\ast}}
		\frac{q^{v(\mathfrak{s}) - \ell}}{\card{\Aut(\mathfrak{s})}}
		\prod_{\substack{h\geq 0 \\ m \geq 1}}
		\prod_{\substack{k_1 \geq \cdots \geq k_m \geq 0 \\ k_1 + \cdots + k_m > 0}}
		t_{h;k_1,\ldots,k_m}^{N_{h;k_1,\ldots,k_m}(\mathfrak{s})},
		\qquad
		S^{\ast}(x_1,x)
		\coloneqq
		\sum_{\ell_1,\ell \geq 1} S_{\ell_1,\ell}^{\ast} \frac{x^{\ell - 1}}{x_1^{\ell_1 + 1}}.
	\end{equation}
\end{defn}

Mutatis mutandis, the argument of \Cref{ssec:rec:skin} yields the recursive relations
\begin{equation}
\label{rec:skin:stuffed}
	S_{\ell_1,\ell}^{\ast}
	=
	\delta_{\ell_1,\ell}
	+
	\sum_{\ell' \geq 1} E_{\ell_1,\ell'}^{\ast} S_{\ell',\ell}^{\ast},
	\qquad
	S_{\ell_1,\ell}^{\ast}
	=
	\delta_{\ell_1,\ell}
	+
	\sum_{\ell' \geq 1} S_{\ell_1,\ell'}^{\ast} E_{\ell',\ell}^{\ast}.
\end{equation}
The recursive formula can be written compactly by introducing the stuffed \emph{annular potentials}, already considered in \cite[Section~3.3]{Bor14}:
\begin{equation}
\label{ann:pot}
\begin{split}
	O(x,y)
	&\coloneqq
	\sum_{m\geq 2}\frac{1}{(m-2)!}
	\bigg\langle
		T_{0,m}(x,\xi_2,\ldots,\xi_{m-1},y)
		\prod_{b=2}^{m-1} W_{0,1}^{\ast}(\xi_b)
	\bigg\rangle_{\bm{\xi}}, \\
	\widetilde{O}(x,y)
	&\coloneqq
	\sum_{m\geq 2}\frac{1}{(m-1)!}
	\bigg\langle
		T_{0,m}(x,\xi_2,\ldots,\xi_{m-1},y)
		\prod_{b=2}^{m-1} W_{0,1}^{\ast}(\xi_b)
	\bigg\rangle_{\bm{\xi}}.
\end{split}
\end{equation}
These potentials package the planar part of the auxiliary-map construction above. Namely, one starts from a genus-zero internal face with several boundary components, keeps two of them as the boundary components of an annulus, and fills the others with arbitrary stuffed discs. The resulting object has the topology of an annulus, which explains the terminology. We use two versions because the number of boundaries playing a symmetric role differs. In $O(x,y)$, the two boundary components carrying $x$ and $y$ are distinguished, so only the remaining $m-2$ boundaries play a symmetric role. In $\widetilde{O}(x,y)$, only the $x$-boundary is distinguished first, while the boundary carrying $y$ is selected among the other $m-1$ boundaries.

After inserting \eqref{eskin:01:stuffed} into the second equality in~\eqref{rec:skin:stuffed} and taking generating series, this gives
\begin{multline}
\label{rec:skin:stuffed:gen}
	\bigg[
		S^\ast(x_1,x)
		\Big(
			x
			-
			T_{0,1}'(x)
			-
			\big\langle \partial_{x} \widetilde{O}(x,\xi) W_{0,1}^{\ast}(\xi)\big\rangle_{\xi} -
			2W_{0,1}^{\ast}(x)
		\Big)
	\bigg]_{+,x} \\
	=
	\bigg[
		\frac{1}{x_1-x}
		+
		\big\langle
			S^\ast(x_1,\xi)
			\partial_{\xi} O(x,\xi)
			W^\ast_{0,1}(\xi)
		\big\rangle_{\xi}
	\bigg]_{+,x}.
\end{multline}

\subsection{Skinning formula and skin enumeration}
Given a stuffed map $\mathfrak{m}$ of topology $(g,n)\neq(0,1)$, iterating Tutte's procedure\footnote{
	As discussed in \Cref{ssec:skinning:tubes}, in order for the iteration to be well defined, arbitrary choices of roots must be made whenever Tutte's procedure encounters an internal face with more than one boundary.
} until the topology changes produces a pair $(\mathfrak{s},\hat{\mathfrak{m}})$ consisting of a stuffed skin map and a stuffed core map. The boundary $\partial_2\mathfrak{s}$ is simple but may not be rooted, whereas $\partial_1\hat{\mathfrak{m}}$ is rooted but may not be simple. As before, one can glue $\partial_2\mathfrak{s}$ to $\partial_1\hat{\mathfrak{m}}$ even when $\partial_2\mathfrak{s}$ is unrooted: any choice of root on $\partial_2\mathfrak{s}$ produces, after gluing, the same stuffed maps. This gives the inverse bijection, and hence the skinning relation
\begin{equation}
\label{skinning:rel:stuffed}
	W_{g;\ell_1,\bm{\ell}_I}^{\ast}
	=
	\sum_{\ell \geq 1} S_{\ell_1,\ell}^{\ast} \hat{W}^{\ast}_{g;\ell,\bm{\ell}_I}
	\qquad\qquad
	W_{g,n}^{\ast}(x_1,\bm{x}_I)
	=
	\Big\langle
		S^{\ast}(x_1,x) \hat{W}^{\ast}_{g,n}(x,\bm{x}_I)
	\Big\rangle_{x}.
\end{equation}
Its generating-series form is parallel to the previous cases.

\begin{prop}[Skinning for stuffed maps]
	\label{prop:skinning:stuf}
	Let $g \geq 0$, $n \geq 1$ be such that $(g,n) \neq (0,1)$. Then
	\begin{equation}\label{eq:skinning:stuf}
	\begin{split}
		& \quad  W_{g,n}^{\ast}(x_1,\bm{x}_I) \\
		&\quad=
		\Bigg\langle
			S^{\ast}(x_1,x)
			\sum_{\substack{h \geq 0 \\ m \geq 1}}
			\sum_{\substack{1 \leq c \leq m \\ (\bm{p},\bm{K},\bm{J}) \in \mathcal{T}^{>}_{g,n;h,m;c}}}
				\!\!\!\!\frac{\partial_{\xi_1} T_{h,m}(\xi_1,\ldots,\xi_m)}{(c - 1)!(x - \xi_1)}
				W^{\ast}_{p_1,1 + \card{K_1} + \card{J_1}}(\xi_1,\bm{\xi}_{K_1},\bm{x}_{J_1})
				\prod_{b = 2}^c W^{\ast}_{p_b,\card{K_b} + \card{J_b}}(\bm{\xi}_{K_b},\bm{x}_{J_b})
		\Bigg\rangle_{x,\bm{\xi}} \\
		&\qquad
		+
		\Bigg\langle
			S^{\ast}(x_1,x)
			\Biggl(
				\sum_{i=2}^n
				\partial_{x_i}\left(
					\frac{W_{g,n-1}^{\ast}(x,\bm{x}_{I_i}) - W_{g,n-1}^{\ast}(x_i,\bm{x}_{I_i})}{x - x_i}
				\right)
				\\
				&\qquad\qquad
				+
				W_{g-1,n+1}^{\ast}(x,x,\bm{x}_I)
				+
				\sum_{\substack{h' + h'' = g \\ J' \sqcup J'' = I}}^{\nodisc}
				W_{h',1+\card{J'}}^{\ast}(x,\bm{x}_{J'})W_{h'',1+\card{J''}}^{\ast}(x,\bm{x}_{J''})
			\Biggr)
		\Bigg\rangle_{x}.
	\end{split}
	\end{equation}
\end{prop}

The skinning of pointed discs works as in \Cref{ssec:skinning:tubes} and yields
\begin{multline}
\label{skinning:discs:stuf}
	\partial_q W^\ast_{0,1}(x_1)-\frac{1}{x_1}
	=
	\bigg[
		S^\ast(x_1,x)\Big(
			T_{0,1}'(x) + 2W^\ast_{0,1}(x)
			+
			\bigl\langle
				\partial_x \widetilde{O}(x,\xi)W^\ast_{0,1}(\xi)
			\bigr\rangle_\xi\Big)
	\bigg]_{0,x} \\
	+
	\big\langle S^{\ast}(x_1,\xi) \partial_{\xi} O(0,\xi) W^\ast_{0,1}(\xi)\big\rangle_{\xi},
\end{multline}
where $O(x,y)$ and $\widetilde{O}(x,y)$ are the annular potentials defined in \Cref{ann:pot}.

As for maps and maps with tubes, the skin enumeration follows by combining the recursive computation of stuffed skin maps \eqref{rec:skin:stuffed:gen}, the skinning formula for annuli, and the skinning formula for pointed discs \eqref{skinning:discs:stuf}. The argument is parallel to the proof of \Cref{skin:enum:maps:tubes}, so we omit the details.

\begin{prop}[Skin enumeration for stuffed maps]
\label{skin:enum:stuf}
	We have
	\begin{equation}
		S^{\ast}(x_1,x)
		=
		\frac{
			{\displaystyle\int^{x}_{\infty}}
			\left(
				2W_{0,2}^{\ast}(x_1,y) + \frac{1}{(x_1 - y)^2}
			\right)\dd y
			+
			\big\langle W_{0,2}^{\ast}(x_1,\xi) O(\xi,x) \big\rangle_{\xi}
			-
			\partial_q W_{0,1}^{\ast}(x_1)
		}{
			x -T_{0,1}'(x) 
			-
			\big\langle \partial_{x} \widetilde{O}(x,\xi) W_{0,1}^{\ast}(\xi)\big\rangle_{\xi} - 2W_{0,1}^{\ast}(x)
		}.
	\end{equation}
\end{prop}

\section{Pair-of-pants excision}
\label{sec:pop}

The goal of this section is to bring the recursive decomposition coming from skinning into a form as close as possible to the excision of an embedded  pair of pants, in the spirit of topological recursion.  Most of the time, it is easy to identify a pair of pants, except in the case (I${}^{>}$).  In that case, removing an internal face can change the topology in a more drastic way than simply removing a pair of pants. We therefore need an alternative procedure, which separates off the relevant part of the internal face. In any case, we will construct a pair of pants $\mathfrak{p}$ and a stuffed \emph{thinned core map} $\check{\mathfrak{m}}$ whose Euler characteristic is one larger, namely $(2-2g-n)+1$:
\begin{equation}
	\mathfrak{m} \longmapsto (\mathfrak{p},\check{\mathfrak{m}}).
\end{equation}
The boundaries of $\mathfrak{p}$ and $\check{\mathfrak{m}}$ are ordered and denoted
\begin{equation}
	(\partial_1\mathfrak{p}, \partial'\mathfrak{p},\partial''\mathfrak{p})
	\qquad\text{and}\qquad
	(\partial'\check{\mathfrak{m}},\partial''\check{\mathfrak{m}},
		\partial_2\check{\mathfrak{m}}, \ldots, \partial_n\check{\mathfrak{m}}
	).
\end{equation}
The original stuffed map $\mathfrak{m}$ is recovered by gluing $\partial'\mathfrak{p}$ to $\partial'\check{\mathfrak{m}}$ and $\partial''\mathfrak{p}$ to $\partial''\check{\mathfrak{m}}$. In particular, $\check{\mathfrak{m}}$ can only have one or two components. When $\check{\mathfrak{m}}$ has no component of topology $(g,n)$, we say that $\mathfrak{p}$ is \emph{relevant}: excising it gives us a recursion on the complexity $2g - 2 + n$. When $\check{\mathfrak{m}}$ has a component of topology $(g,n)$, the excision of $\mathfrak{p}$ does not reduce the complexity and requires a different treatment.

We work throughout in the stable range $2g-2+n>0$, where pair-of-pants excision is meaningful, and directly with stuffed maps, since the organising principles appear most clearly in that generality; the cases of maps and maps with self-avoiding loops are recovered by specialisation.

\subsection{From skinning to pair-of-pants excision}
\label{ssec:skinning:to:pop}
Recall the usual skinning procedure
\begin{equation}
	\mathfrak{m} \longmapsto (\mathfrak{s},\hat{\mathfrak{m}}),
\end{equation}
which scraps off a stuffed skin map $\mathfrak{s}$ and leaves a stuffed core map $\hat{\mathfrak{m}}$. We now apply Tutte's procedure once more to $\hat{\mathfrak{m}}$, and call $\mathfrak{e}$ the resulting stuffed elementary skin map. The topology-changing cases at this last step fall into three types which reorganise the previous classification.
\begin{itemize}
	\item The first type is the case (D${}^{>}$). Apart from degree matching, the pair of pants and the complementary map will be arbitrary. We say that the boundaries of the complementary map are \emph{standard}.

	\item The second type consists of the case (R) together with the internal-face cases (I${}^{>}$) that behave in the same way, i.e. give a relevant pair of pants. In such cases, one of the two new boundaries of the complementary stuffed map is adjacent to a single face: we say that the complementary stuffed map is \emph{special} along that boundary.
	
	\item The third type consists of the remaining internal-face case. Here one complementary component is a stuffed disc, while the other still has topology $(g,n)$. The complexity cannot be reduced by excise the pair of pants, but only by getting rid of the internal face which has negative Euler characteristic. We say these configurations as \emph{indecomposable}.
\end{itemize}
We now describe the combinatorics of these three cases separately. In the first two cases, we also describe how to excise a relevant pair of pants $\mathfrak{p}$, obtained by gluing one additional layer to the skin map $\mathfrak{s}$.

\subsubsection{The standard case}
This is the case in which the stuffed core map $\hat{\mathfrak{m}}$ is of type (D${}^{>}$). The stuffed map $\check{\mathfrak{m}}$ is obtained from $\hat{\mathfrak{m}}$ by applying Tutte's procedure once more. Since $r_1$ bounds the first boundary on both sides, erasing it creates two new boundaries $\partial'\check{\mathfrak{m}}$ and $\partial''\check{\mathfrak{m}}$, rooted at $\rho_1^+$ and $\rho_1^-$, respectively. $\mathfrak{e}$ is a pair of glasses, see~\Cref{ssec:skin}, with boundaries $\partial_1 \mathfrak{e} = \partial_1 \hat{\mathfrak{m}}$, and $\partial' \mathfrak{e},\partial'' \mathfrak{e}$ that can be glued to $\partial' \check{\mathfrak{m}},\partial'' \check{\mathfrak{m}}$ to get back $\hat{\mathfrak{m}}$. The boundary degrees satisfy
\begin{equation}
	\ell(\partial'\check{\mathfrak{m}})
	+
	\ell(\partial''\check{\mathfrak{m}})
	=
	\ell(\partial_1\hat{\mathfrak{m}}) - 2.
\end{equation}

The stuffed map $\mathfrak{p}$ is obtained by gluing $\partial_1\mathfrak{e}$ to $\partial_2\mathfrak{s}$. The resulting stuffed map is a pair of pants, with $\partial_1\mathfrak{p} = \partial_1\mathfrak{s}$, while $\partial'\mathfrak{p}$ and $\partial''\mathfrak{p}$ are opposite to $\partial'\check{\mathfrak{m}}$ and $\partial''\check{\mathfrak{m}}$, respectively.

We denote by $\hat{\MM}^{\std}_{g;\ell_1,\bm{\ell}_I}$ the set of stuffed core maps $\hat{\mathfrak{m}}$ of topology $(g,n)$ and boundary degrees $\ell_1,\bm{\ell}_{I}$ in the standard case. Then the above construction gives the bijection
\begin{equation}
\label{Mstd}
	\hat{\MM}^{\std}_{g;\ell_1,\bm{\ell}_I}
	\overset{\simeq}{\longrightarrow}
	\bigsqcup_{\substack{k',k'' \geq 1 \\ k'+k''=\ell_1-2}}
	\Biggl(
		\MM^\ast_{g-1;k',k'',\bm{\ell}_I}
		\sqcup
		\bigsqcup_{\substack{h'+h''=g \\ J'\sqcup J'' = I}}^{\nodisc}
			\MM^\ast_{h';k',\bm{\ell}_{J'}}
			\times
			\MM^\ast_{h'';k'',\bm{\ell}_{J''}}
	\Biggr),
	\qquad
	\hat{\mathfrak{m}} \longmapsto \check{\mathfrak{m}}.
\end{equation}
The inverse bijection is given as follows. We take a copy $\partial'\mathfrak{e}$ of $\partial'\check{\mathfrak{m}}$ and a copy $\partial''\mathfrak{e}$ of $\partial''\check{\mathfrak{m}}$ that we orient in the opposite way. We add an edge joining the endpoint of the root of $\partial'\mathfrak{e}$ to the starting point of the root of $\partial''\mathfrak{e}$, so as to form a pair of glasses $\mathfrak{e}$, where $\partial_1\mathfrak{e}$ is the longest boundary. Eventually, we get $\hat{\mathfrak{m}}$ by gluing $\partial'\mathfrak{e}$ to $\partial'\check{\mathfrak{m}}$ and $\partial''\mathfrak{e}$ to $\partial''\check{\mathfrak{m}}$.

\subsubsection{The special case.}
This case consists of the case (R), together with the sub-case of (I${}^{>}$) in which the pair-of-pants is still relevant. We first analyse the case (R).

(R) Suppose the stuffed core map $\hat{\mathfrak{m}}$ is of type (R): its root-edge $r_1$ bounds $\partial_i\hat{\mathfrak{m}}$ for some $i\in\set{2,\ldots,n}$. We define $\check{\mathfrak{m}} = \check{\mathfrak{m}}'\sqcup \check{\mathfrak{m}}''$, where $\check{\mathfrak{m}}'$ is obtained from $\hat{\mathfrak{m}}$ by applying Tutte's procedure once more, and $\check{\mathfrak{m}}''$ is the trivial skin of degree $\ell_i$, i.e. an annulus with no internal faces. The component $\check{\mathfrak{m}}'$ carries the new boundary $\partial'\check{\mathfrak{m}}$ together with the boundaries $\partial_j\check{\mathfrak{m}}$ for $j\in\set{2,\ldots,n}\setminus\set{i}$, inherited from $\hat{\mathfrak{m}}$. We keep this trivial skin $\check{\mathfrak{m}}''$ in the game so that the stuffed thinned core map $\check{\mathfrak{m}}$ still has two distinguished boundaries to be glued back to the last two boundaries of the pair of pants. 

The stuffed map $\mathfrak{p}$ is obtained as in the standard case. The map $\mathfrak{e}$ is a theta map in this case (see~\Cref{ssec:skin}) and we glue $\partial_2\mathfrak{s}$ to $\partial_1\mathfrak{e}$. The result is a pair of pants $\mathfrak{p}$ with boundaries $(\partial_1\mathfrak{p},\partial'\mathfrak{p},\partial''\mathfrak{p}) = (\partial_1\mathfrak{s},\partial'\mathfrak{e},\partial''\mathfrak{e})$. The original stuffed map $\mathfrak{m}$ is recovered by gluing $\partial'\mathfrak{p}$ to $\partial'\check{\mathfrak{m}}$ and $\partial''\mathfrak{p}$ to $\partial''\check{\mathfrak{m}}$.

(I${}^{>}$) We now turn to the internal-face contribution. A new phenomenon occurs, as applying Tutte's procedure may change the topology more drastically than a usual increase of the Euler characteristic by $1$. Suppose that the stuffed core map $\hat{\mathfrak{m}}$ is of type (I${}^{>}$), so that $r_1$ belongs both to $\partial_1\hat{\mathfrak{m}}$ and to a boundary $\delta_1$ of degree $k_1$ of an internal face $\mathfrak{f}$ of topology $(h,m)\neq(0,1)$. We first describe the construction common to all such cases, and then separate them further depending on whether the stuffed thinned core map has a component of topology $(g,n)$ or not.

The map $\check{\mathfrak{m}}$ is obtained from $\hat{\mathfrak{m}}$ by cutting along $\delta_1$ (this is the inverse operation to gluing). The net effect is to replace $\partial_1\hat{\mathfrak{m}}$ by two labelled boundaries $\partial'\check{\mathfrak{m}}$ and $\partial''\check{\mathfrak{m}}$. The boundary $\partial'\check{\mathfrak{m}}$ has degree $\ell_1 + k_1 - 2$ and we root it at $\overline{\rho_1}{}^{-}$. The boundary $\partial''\check{\mathfrak{m}}$ has degree $k_1$, is glued to $\delta_1$ (this defines unambiguously its root to be the half-edge opposite to $\overline{\rho_1}$). The rest of $\check{\mathfrak{m}}$ is unchanged compared to $\hat{\mathfrak{m}}$, and the boundary degrees satisfy
\begin{equation}
	\ell(\partial_1\hat{\mathfrak{m}})
	=
	\ell(\partial'\check{\mathfrak{m}})
	-
	\ell(\partial''\check{\mathfrak{m}})
	+
	2.
\end{equation}

To construct the pair of pants, we introduce an auxiliary theta map $\mathfrak{e}$ where $\partial_1\mathfrak{e} = \partial_1\hat{\mathfrak{m}}$, where $\partial'\mathfrak{e}$ is a copy of $\partial'\check{\mathfrak{m}}$ with opposite orientation, and where $\partial''\mathfrak{e}$ is a copy of  $\partial''\check{\mathfrak{m}} = \delta_1$ with opposite orientation. We glue $\partial_2\mathfrak{s}$ to $\partial_1\mathfrak{e}$. The result is a pair of pants $\mathfrak{p}$, with $(\partial_1\mathfrak{p},\partial'\mathfrak{p},\partial''\mathfrak{p}) = (\partial_1\mathfrak{s},\partial'\mathfrak{e},\partial''\mathfrak{e})$, see~\Cref{fig:pop:I}.

\begin{figure}
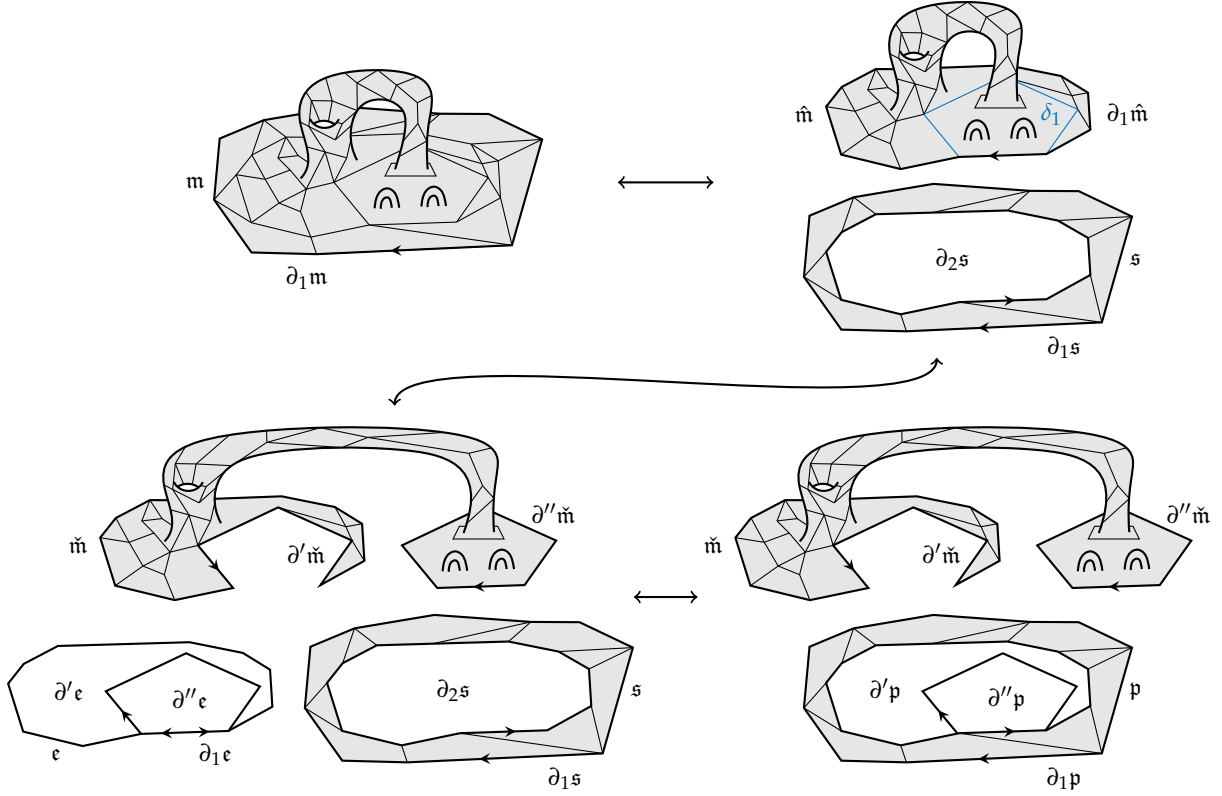

	\centering

	\caption{Starting from a map $\mathfrak{m}$ (top left), the skinning process produces the skin $\mathfrak{s}$ and the core map $\hat{\mathfrak{m}}$ (top right). The internal face is then separated off along $\delta_1$ (in blue) to produce the thinned core map $\check{\mathfrak{m}}$, and the auxiliary elementary piece $\mathfrak{e}$ is prepared (bottom left). Finally, gluing $\mathfrak{e}$ to $\mathfrak{s}$ produces the pair of pants $\mathfrak{p}$ (bottom right).}
	\label{fig:pop:I}
\end{figure}

With these definitions, $\check{\mathfrak{m}}$ is either connected or has two labelled connected components. If disconnected, we denote its components by $\check{\mathfrak{m}}'$ and $\check{\mathfrak{m}}''$, where $\check{\mathfrak{m}}'$ contains $\partial'\check{\mathfrak{m}}$ and $\check{\mathfrak{m}}''$ contains the special boundary $\partial''\check{\mathfrak{m}}$. The sub-case where $\check{\mathfrak{m}}'$ is not a disc belongs to the special case: it gives a pair-of-pants excision with special complement. The remaining sub-case, where $\check{\mathfrak{m}}'$ is a disc and $\check{\mathfrak{m}}''$ still has topology $(g,n)$, defines the indecomposable case. Notice that $\check{\mathfrak{m}}''$ cannot be a disc, otherwise $\Tutte$ applied to $\hat{\mathfrak{m}}$ would not be topology-changing.

The case (R) can now be viewed as the degenerate boundary-face version of the disconnected internal-face case: viewing the boundary as a degenerate annular face, we apply the same operation as in (I${}^{>}$) by cutting along its edges. This yields the trivial skin for $\check{\mathfrak{m}}''$, and again a theta map $\mathfrak{e}$ agreeing with the elementary skin map $\mathfrak{e}$ from before.

Since in all these cases the boundary $\partial'' \check{\mathfrak{m}}$ is adjacent to a single face, either boundary or internal, it is natural to isolate this property.

\begin{defn}
\label{del:special}
	A stuffed map in $\MM^\ast_{g;\ell_1,\ldots,\ell_n}$ is called $\partial_i$-\emph{special} if $\partial_i$ is adjacent to a single face, either boundary or internal. We denote by $\mathcal{R}_i\MM^\ast_{g;\ell_1,\ldots,\ell_n}$ the set of $\partial_i$-special stuffed maps of topology $(g,n)$ and boundary degrees $\ell_1,\ldots,\ell_n$. The generating series of $\partial_i$-special stuffed maps with fixed boundary degrees is
	\begin{equation}
		\mathcal{R}_i W_{g;\ell_1,\ldots,\ell_n}^{\ast}
		\coloneqq
		\sum_{\mathfrak{m} \in \mathcal{R}_i \MM_{g;\ell_1,\ldots,\ell_n}^\ast}
		\frac{q^{v(\mathfrak{m})}}{\card{\Aut(\mathfrak{m})}}
		\prod_{\substack{h\geq 0 \\ m \geq 1}}
		\prod_{\substack{k_1 \geq \cdots \geq k_m \geq 0 \\ k_1 + \cdots + k_m > 0}}
		t_{h;k_1,\ldots,k_m}^{N_{h;k_1,\ldots,k_m}(\mathfrak{m})}.
	\end{equation}
	With free boundary degrees, we set
	\begin{equation}
		\mathcal{R}_{x_i}W_{g,n}^{\ast}(x_1,\ldots,x_n)
		\coloneqq
		- x_1\delta_{g,0}\delta_{n,1}
		+
		\sum_{\ell_1,\ldots,\ell_n \geq 1}
			\mathcal{R}_i W_{g;\ell_1,\ldots,\ell_n}^{\ast}  x_i^{\ell_i - 1} \prod_{j \neq i} x_j^{-(\ell_j + 1)}. 
	\end{equation}
	The initial term for $(g,n) = (0,1)$ is a convention that will be convenient for the formula in \Cref{revisit:skin}. In the case of annuli, we also need to single out those $\partial_1$-special maps for which $\partial_1$ is adjacent to an internal face, rather than to a boundary face. We denote their generating series by $\mathcal{R}_{x_1}\widetilde{W}^{\ast}_{0,2}(x_1,x_2)$.
\end{defn}
\begin{figure}[H]
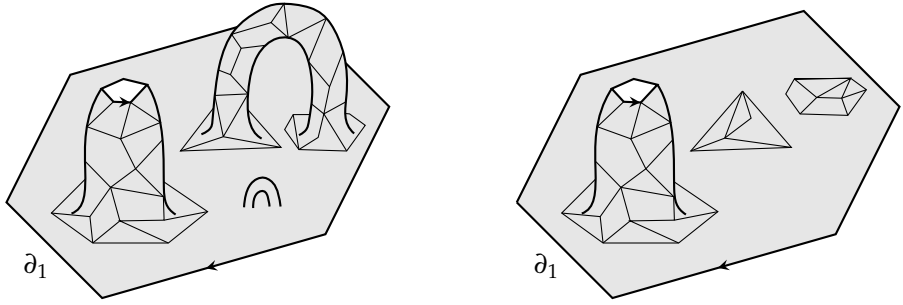

	\centering

	\caption{Examples of $\partial_1$-special maps. The left one is $\partial_1$-indecomposable (cf.~\Cref{def:indec}), while the right one is not.}
	\label{fig:special}
\end{figure}

The symbol $\mathcal{R}_{x_i}$ is not an operator in the strict sense: it is only defined on the specific series $W_{g,n}^{\ast}$, and by independently enumerating a restricted class of maps. Nevertheless, it is convenient to use operator-like notation. This notation anticipates the analytic interpretation, where $\mathcal{R}_{x_i}$ corresponds to the \emph{regular part} of the associated series across the cut.

We denote by $\hat{\MM}^{\spl}_{g;\ell_1,\bm{\ell}_I}$ the set of core maps $\hat{\mathfrak{m}}$ of topology $(g,n)$ and boundary degrees $\ell_1,\bm{\ell}_{I}$ in the special case. Then the above construction for $(g,n)\neq(1,1)$ gives the bijection
\begin{equation}
\label{spl}
	\hat{\MM}^{\spl}_{g;\ell_1,\bm{\ell}_I}
	\overset{\simeq}{\longrightarrow}
	\bigsqcup_{k_1 \geq 1}
	\Biggl(
		\mathcal{R}_2 \MM^\ast_{g-1;\ell_1+k_1-2,k_1,\bm{\ell}_I}
		\sqcup
		\bigsqcup_{\substack{h'+h''=g \\ J'\sqcup J''=I}}^{\nodisc}
			\MM^\ast_{h';\ell_1+k_1-2,\bm{\ell}_{J'}}
			\times
			\mathcal{R}_1 \MM^\ast_{h'';k_1,\bm{\ell}_{J''}}
	\Biggr),
	\qquad
	\hat{\mathfrak{m}} \longmapsto \check{\mathfrak{m}}.
\end{equation}
The case (R) is included in the second summand, with the special boundary adjacent to a boundary face rather than to an internal face (i.e. one connected component is a trivial cylinder). The bijection inverse to \eqref{spl} is obtained similarly to \eqref{Mstd}, by constructing a theta map instead of a pair of glasses, and gluing its first and third boundaries with $\partial'\check{\mathfrak{m}}$ and $\partial''\check{\mathfrak{m}}$.

Note that maps in which the special boundary is adjacent to another boundary never appear in the first summand, as desired, since they correspond to case (R), which does not change the genus. For $(g,n)=(1,1)$, however, such maps are contained in $\mathcal{R}_2\MM^\ast_{g-1;\ell_1+k_1-2,k_1,\bm{\ell}_I}$ and must therefore be excluded to obtain a bijection. Thus, for $(g,n)=(1,1)$,
\begin{equation}
\label{spl:1:1}
	\hat{\MM}^{\spl}_{1;\ell_1}
	\overset{\simeq}{\longrightarrow}
	\bigsqcup_{k_1 \geq 1}
		\mathcal{R}_2\widetilde{\MM}^\ast_{0;\ell_1+k_1-2,k_1},
	\qquad
	\hat{\mathfrak{m}} \longmapsto \check{\mathfrak{m}},
\end{equation}
where $\mathcal{R}_2\widetilde{\MM}^\ast$ denotes the set of $\partial_2$-special cylinders other than the trivial cylinder.

\subsubsection{The indecomposable case.}
This is the remaining sub-case of (I${}^{>}$). With the notation above, the thinned core map $\check{\mathfrak{m}}$ is disconnected, the component $\check{\mathfrak{m}}'$ containing $\partial'\check{\mathfrak{m}}$ is a disc, and the component $\check{\mathfrak{m}}''$ containing the special boundary $\partial''\check{\mathfrak{m}}$ still has topology $(g,n)$. In this case, we modify our construction:  we glue back the boundary face of the disc $\check{\mathfrak{m}}'$ to $\partial'\mathfrak{p}$. This produces a stuffed annulus $\mathfrak{o}$ with boundaries $\partial_1\mathfrak{o}$ and $\partial''\mathfrak{o}$, see~\Cref{fig:blob}, and gives a modified assignment
\begin{equation}
\label{ind:bij}
	\mathfrak{m} \longmapsto (\mathfrak{o},\check{\mathfrak{m}}'').
\end{equation}
This assignment is a bijection onto its image. The reverse bijection is the gluing of the $\partial''$ boundaries, but we still need to describe the image of \eqref{ind:bij}.

First, $\mathfrak{o}$ is an arbitrary stuffed annulus: the map $(\mathfrak{p},\check{\mathfrak{m}}') \mapsto \mathfrak{o}$ is inverse to annulus skinning, the only difference being that the theta map is glued to $\mathfrak{s}$ to form $\mathfrak{p}$. Second, $\check{\mathfrak{m}}''$ must be a stuffed map of topology $(g,n)$ whose first boundary is special, is adjacent to an internal face, and is such that applying Tutte's procedure to it changes topology. The only compatibility constraint is that the degrees of $\partial''\mathfrak{o}$ should match the degree of $\partial_1\check{\mathfrak{m}}''$.  This leads to the following definition.

\begin{figure}
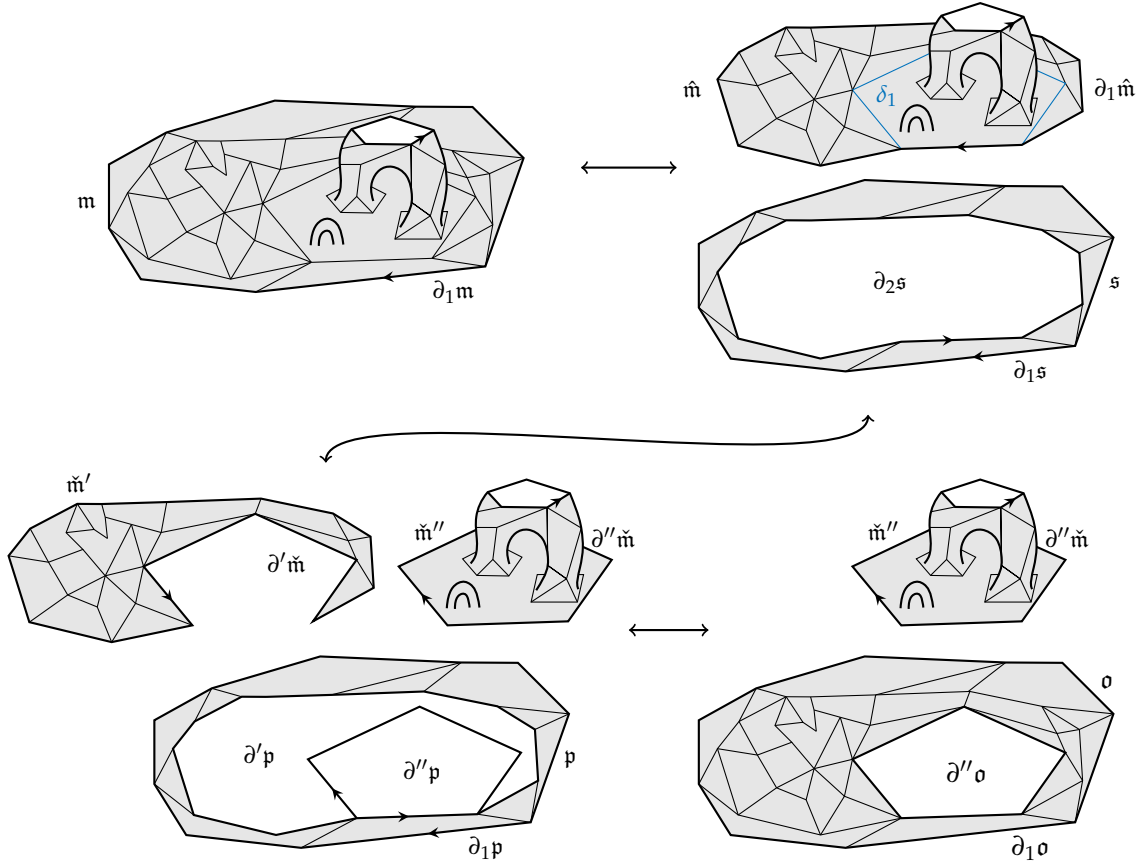

	\centering

	\caption{Starting from a map $\mathfrak{m}$ (top left), the skinning process produces the skin $\mathfrak{s}$ and the core map $\hat{\mathfrak{m}}$ (top right). The internal face is then separated off along $\delta_1$ (in blue) to produce the pair of pants $\mathfrak{p}$ and the thinned core map $\check{\mathfrak{m}} = \check{\mathfrak{m}}' \sqcup \check{\mathfrak{m}}''$ (bottom left). Finally, gluing $\check{\mathfrak{m}}'$ to $\mathfrak{p}$ produces the annulus $\mathfrak{o}$ (bottom right).}
	\label{fig:blob}
\end{figure}

\begin{defn}\label{def:indec}
	A $\partial_1$-special stuffed map is called $\partial_1$-\emph{indecomposable} if $\partial_1$ is adjacent to an internal face and applying Tutte's procedure at $\partial_1$ changes topology. We denote by $\mathbb{V}_{g;\underline{\ell_1};\ell_2,\ldots,\ell_n}$ the set of $\partial_1$-indecomposable stuffed maps of topology $(g,n)$ and boundary degrees $\ell_1,\ldots,\ell_n$, in which $\partial_1$ is unrooted. The generating series of $\partial_1$-indecomposable stuffed maps with fixed boundary degrees is
	\begin{equation}
	\label{eq:indecomp}
		V_{g;\underline{\ell_1};\bm{\ell}_I}
		\coloneqq
		\sum_{\mathfrak{m} \in \mathbb{V}_{g;\underline{\ell_1};\bm{\ell}_I}}
		\frac{q^{v(\mathfrak{m})}}{\card{\Aut(\mathfrak{m})}}
		\prod_{\substack{h\geq 0 \\ m \geq 1}} \;
		\prod_{\substack{k_1 \geq \cdots \geq k_m \geq 0 \\ k_1 + \cdots + k_m > 0}}
		t_{h;k_1,\ldots,k_m}^{N_{h;k_1,\ldots,k_m}(\mathfrak{m})}.
	\end{equation}
	With free boundary degrees, we set
	\begin{equation}
		V_{g,n}(x_1;\bm{x}_I)
		\coloneqq
		\sum_{\ell_1,\ldots,\ell_n \geq 1}
			V_{g;\underline{\ell_1};\bm{\ell}_I}
			 x_1^{\ell_1 - 1} \prod_{j = 2}^{n} x_j^{-(\ell_j + 1)}.
	\end{equation}
\end{defn}

Note that a $\partial_1$-special map is not $\partial_1$-indecomposable if and only if it is of topology $(0,k)$ with $k-2$ discs glued to its boundaries, see for example \Cref{fig:special}. These maps should be viewed as stuffed generalisations of annular faces, since they cannot change the topology on their own.

We denote by $\hat{\MM}^{\ind}_{g;\ell,\bm{\ell}_I}$ the set of core maps $\hat{\mathfrak{m}}$ of topology $(g,n)$ and boundary degrees $\ell,\bm{\ell}_I$ in the indecomposable case. The discussion above gives the bijection
\begin{equation}
\label{eq:blob:bijection}
	\bigsqcup_{\ell\geq 1}
		\mathbb{S}^\ast_{\ell_1,\ell}
		\times
		\hat{\MM}^{\ind}_{g;\ell,\bm{\ell}_I}
	\overset{\simeq}{\longrightarrow}
	\bigsqcup_{\ell\geq 1}
		\MM^\ast_{0;\ell_1,\ell}
		\times
		\mathbb{V}_{g;\underline{\ell};\bm{\ell}_I},
	\qquad
	(\mathfrak{s},\hat{\mathfrak{m}})
	\longmapsto
	(\mathfrak{o},\check{\mathfrak{m}}'').
\end{equation}
Indeed, the left-hand side records the original skin together with a core map in the indecomposable case. Applying the modified construction described above, namely thinning $\hat{\mathfrak{m}}$ and then gluing back the disc component $\check{\mathfrak{m}}'$ to the corresponding boundary of the pair of pants, produces an arbitrary stuffed annulus $\mathfrak{o}$, while the remaining component $\check{\mathfrak{m}}''$ is precisely a $\partial''$-indecomposable stuffed map with unrooted $\partial''$ boundary. Conversely, choosing any root on the $\partial''$ boundary of the indecomposable map and gluing the two factors on the right-hand side along their boundaries of degree $\ell$ and applying the skinning procedure retrieves a unique pair $(\mathfrak{s},\hat{\mathfrak{m}})$ with $\hat{\mathfrak{m}}$ in the indecomposable case. 

This reformulation reflects an intrinsic property of the stuffed maps that fall into the indecomposable case. Informally, these are the stuffed maps for which the change of topology seen from the first boundary comes from the topology of an internal face, rather than from the gluing of edges. More precisely, these are the stuffed maps containing an internal face $\mathfrak{f}$ with one boundary $\delta_1$ freely homotopic to the first boundary, but neither null-homotopic nor homotopic to another boundary of $\mathfrak{f}$. Such maps cannot be decomposed in a way that increases the Euler characteristic of each component unless one includes the face $\mathfrak{f}$ in the excised part; we refrain from doing this because then the excised piece could have arbitrary topology. These indecomposable configurations are absent for maps and maps with tubes.

\subsubsection{The excision formula.}
Finally, we relate skin maps to the pairs of pants arising in the recursive cases. In the standard and special cases, the pair of pants $\mathfrak{p}$ is obtained by gluing the stuffed skin map $\mathfrak{s}$ to the elementary skin map $\mathfrak{e}$ extracted from the first boundary of the core map. Conversely, once the thinned core map $\check{\mathfrak{m}}$ is fixed, the second and third boundaries of $\mathfrak{p}$ determine $\mathfrak{e}$, so recovering $\mathfrak{p}$ is equivalent to recovering the original skin map. Thus, for fixed boundary degrees, stuffed skin maps are in bijection with the pairs of pants arising in the standard and special cases.

Combining the three cases above, for $(g,n)\neq(1,1)$ we obtain the bijection
\begin{equation}
\label{eq:pop:global:bijection}
\begin{split}
	\MM^\ast_{g;\ell_1,\bm{\ell}_I}
	\overset{\simeq}{\longrightarrow}
	&\bigsqcup_{\ell\geq 1}
		\mathbb{S}^\ast_{\ell_1,\ell} \times
		\bigsqcup_{\substack{k',k'' \geq 1 \\ k'+k''=\ell-2}}
		\Biggl(
			\MM^\ast_{g-1;k',k'',\bm{\ell}_I}
			\sqcup
			\bigsqcup_{\substack{h'+h''=g \\ J'\sqcup J'' = I}}^{\nodisc}
				\MM^\ast_{h';k',\bm{\ell}_{J'}}
				\times
				\MM^\ast_{h'';k'',\bm{\ell}_{J''}}
		\Biggr) \\
	&\sqcup \bigsqcup_{\ell\geq 1}
		\mathbb{S}^\ast_{\ell_1,\ell} \times
		\bigsqcup_{k\geq 1}
		\Biggl(
			\mathcal{R}_2\MM^\ast_{g-1;\ell+k-2,k,\bm{\ell}_I}
			\sqcup
			\bigsqcup_{\substack{h'+h''=g \\ J'\sqcup J''=I}}^{\nodisc}
				\MM^\ast_{h';\ell+k-2,\bm{\ell}_{J'}}
				\times
				\mathcal{R}_1\MM^\ast_{h'';k,\bm{\ell}_{J''}}
		\Biggr) \\
	&\sqcup
	\bigsqcup_{\ell\geq 1}
		\big(\MM^\ast_{0;\ell_1,\ell}
		\times
		\mathbb{V}_{g;\underline{\ell};\bm{\ell}_I}\big).
\end{split}
\end{equation}
The first two lines correspond to the genuine pair-of-pants excision: the factor $\mathbb{S}^\ast_{\ell_1,\ell}$ records the pair of pants, or equivalently the skin map, while the remaining factors record the recursive thinned core in the standard and special cases. The last line is the part where excising a pair of pants would not decrease the topology: it has been replaced by an arbitrary stuffed annulus in $\MM^\ast_{0;\ell_1,\ell}$, and the remaining map is $\partial_1$-indecomposable in $\mathbb{V}_{g;\underline{\ell};\bm{\ell}_I}$.

In terms of generating series with free boundary degrees, the above bijection yields the following excision formula:
\begin{equation}
\begin{split}
	&W_{g,n}^{\ast}(x_1,x_2,\ldots,x_n) = \\
	&
	\Bigg\langle
		S^{\ast}(x_1,x)
		\bigg(
			W_{g-1,n+1}^{\ast}(x,x,x_2,\ldots,x_n)
			+
			\sum_{\substack{h' + h'' = g \\ J' \sqcup J'' = \set{2,\ldots,n}}}^{\nodisc}
				W_{h',1+\card{J'}}^{\ast}(x,\bm{x}_{J'})
				\mathcal{R}_{x} W_{h'',1+\card{J''}}^{\ast}(x,\bm{x}_{J''})
		\bigg) \\
	& +
		S^{\ast}(x_1,x)
		\bigg(
			\mathcal{R}_{y} W_{g-1,n+1}^{\ast}(x,y,x_2,\ldots,x_n) \big|_{y=x}
			+
			\sum_{\substack{h' + h'' = g \\ J' \sqcup J'' = \set{2,\ldots,n}}}^{\nodisc}
				W_{h',1+\card{J'}}^{\ast}(x,\bm{x}_{J'})
				\mathcal{R}_{x} W_{h'',1+\card{J''}}^{\ast}(x,\bm{x}_{J''})
		\bigg) \\
	& +
		W_{0,2}^{\ast}(x_1,x)V_{g,n}(x,x_2,\ldots,x_n)
	\Bigg\rangle_x
\end{split}
\end{equation}
Notice that the combination that appears naturally in the excision formula is the following.

\begin{defn}
\label{def:mndrm}
	Given $i \in \set{1,\ldots,n}$, we define the \emph{formal monodromy} by
	\begin{equation}
		\mathcal{M}_{x_i}W_{g,n}^{\ast}(x_1,\ldots,x_n)
		\coloneqq
		W_{g,n}^{\ast}(x_1,\ldots,x_n) + \mathcal{R}_{x_i}W_{g,n}^{\ast}(x_1,\ldots,x_n).
	\end{equation}
	We also need the combination
	\begin{equation}
		\mathcal{M}_{x_2}\widetilde{W}^{\ast}_{0,2}(x_1,x_2)
		\coloneqq
		W_{0,2}^{\ast}(x_1,x_2) + \mathcal{R}_{x_2}\widetilde{W}^{\ast}_{0,2}(x_1,x_2).
	\end{equation}
\end{defn}

As for $\mathcal{R}_{x_i}$, the symbol $\mathcal{M}_{x_i}$ is not an operator in the strict sense, but it is convenient to use operator-like notation and write symbolically $\mathcal{M}_{x_i}=\textnormal{Id}+\mathcal{R}_{x_i}$. The name `formal monodromy' anticipates the analytic interpretation in the next section: it corresponds to evaluating the generating series on the opposite sheet across the cut. With this notation in place, the excision formula takes the form stated in the introduction, \Cref{intro:excision}.

\begin{thm}[Excision formulae]
\label{thm:excision}
	We have the excision relations
	\begin{multline}
		W_{g,n}^{\ast}(x_1,x_2,\ldots,x_n)
		=
		\Bigg\langle
			S^{\ast}(x_1,x)
			\bigg(
				\mathcal{M}_{y}W_{g-1,n+1}^{\ast}(x,y,x_2,\ldots,x_n)\big|_{y=x} \\
				\qquad\qquad\qquad
				+
				\sum_{\substack{h' + h'' = g \\ J' \sqcup J'' = \set{2,\ldots,n}}}^{\nodisc}
					W_{h',1+\card{J'}}^{\ast}(x,\bm{x}_{J'})
					\mathcal{M}_{x} W_{h'',1+\card{J''}}^{\ast}(x,\bm{x}_{J''})
			\bigg)
			+ W_{0,2}^{\ast}(x_1,x)V_{g,n}(x,x_2,\ldots,x_n)
		\Bigg\rangle_x
	\end{multline}
	for $(g,n) \neq (0,1),(0,2),(1,1)$, and the low-topology cases
	\begin{equation}
	\begin{split}
		W_{0,2}^{\ast}(x_1,x_2)
		&=
		\bigg \langle \frac{S^{\ast}(x_1,x) W_{0,1}^{\ast}(x)}{(x_2 - x)^2} \bigg\rangle_{x}, \\
		W_{1,1}^{\ast}(x_1)
		&=
		\Big\langle
			S^{\ast}(x_1,x) \mathcal{M}_{y}\widetilde{W}^{\ast}_{0,2}(x,y)\big|_{y = x}
			+
			W_{0,2}^{\ast}(x_1,x) V_{1,1}(x)
		\Big\rangle_{x}.
	\end{split}
	\end{equation}
\end{thm}

The case $(0,2)$ is precisely the skinning formula of \Cref{prop:skinning:stuf}, since cylinders do not admit a pair-of-pants decomposition. The case $(1,1)$ follows from \Cref{spl:1:1}. We remark that \Cref{thm:excision} can also be derived from \Cref{prop:skinning:stuf} by straightforward algebraic manipulations. This derivation mirrors the bijective construction above, and we omit it for brevity.

\subsection{Boundary special maps, indecomposable maps, and the skin enumeration}
\label{ssec:special:indecomp:skin}
The excision formula of \Cref{thm:excision} is not yet fully explicit. Besides the usual generating series $W_{g,n}^{\ast}$ and the skin series $S^{\ast}$, it involves the regular parts $\mathcal{R}_{x_i}W_{g,n}^{\ast}$, which enumerate $\partial_i$-special maps, and the series $V_{g,n}$, which enumerates $\partial_1$-indecomposable maps. The goal of this subsection is to express these auxiliary series in terms of the basic generating series introduced earlier. We also revisit the formula for the enumeration of stuffed skin maps in \Cref{intro:skin:enum} in a form better adapted to the analytic interpretation in the next section.

\subsubsection{Boundary special and indecomposable maps}
We first treat the cases $2g -2  +n > 0$. The $\partial_1$-special stuffed maps have two possible origins: either they are indecomposable, or they are obtained by gluing a planar face to the first boundary and filling all but one of its boundaries with stuffed discs. The building blocks for the latter operation are enumerated by the annular potential $O(x,y)$, and we use it as a linear operator on generating series.

\begin{defn}
\label{def:O:operator}
	For a series $F(x)=\sum_{\ell \geq 1} F_\ell x^{-\ell-1}$, we define
	\begin{equation}
	\label{eq:O:operator}
		\mathcal{O}F(x)
		\coloneqq
		\big\langle \partial_x O(x,\xi) F(\xi)\big\rangle_{\xi}
		=
		\sum_{m\geq 2}\frac{1}{(m-2)!}
		\bigg\langle
			\partial_x T_{0,m}(x,\xi,\xi_3,\ldots,\xi_m)\,
			F(\xi)
			\prod_{b=3}^{m}W_{0,1}^\ast(\xi_b)
		\bigg\rangle_{\bm{\xi}}.
	\end{equation}
	When the series depends on several variables, we write $\mathcal{O}_{x_i}$ to indicate that the operator acts on the variable $x_i$.
\end{defn}

With this notation, the informal description we just gave of $\partial_1$-special stuffed maps translates into the following identity.

\begin{lem}
\label{lem:R1W:stable}
	For $2g-2+n>0$,
	\begin{equation}
	\label{R1W:stable}
		\mathcal{R}_{x_1}W_{g,n}^{\ast}(x_1,\bm{x}_I)
		=
		\partial_{x_1} V_{g,n}(x_1;\bm{x}_I)
		+
		\mathcal{O}_{x_1}W_{g,n}^{\ast}(x_1,\bm{x}_I).
	\end{equation}
\end{lem}

\begin{proof}
	Let $\mathfrak{m}$ be a $\partial_1$-special stuffed map with $2g - 2 + n > 0$. Since $\partial_1$ is adjacent to a single face, and since a boundary-face contribution can only occur for annular topology, this face is internal. If applying Tutte's procedure at $\partial_1$ changes topology, then $\mathfrak{m}$ is $\partial_1$-indecomposable. In the series $V_{g,n}(x_1;\bm{x}_I)$ the variable $x_1$ is carried by an unrooted boundary, while in $\mathcal{R}_{x_1}W_{g,n}^{\ast}$ it is carried by a rooted boundary; applying $\partial_{x_1}$ roots this boundary and gives the contribution $\partial_{x_1}V_{g,n}(x_1;\bm{x}_I)$. Otherwise, the topology is preserved. Then the internal face must be planar, one of its other boundary components carries a stuffed map of topology $(g,n)$, and the remaining boundary components are filled by stuffed discs. This is exactly the operation encoded by $\mathcal{O}_{x_1}W_{g,n}^{\ast}(x_1,\bm{x}_I)$.
\end{proof}

For $(g,n) = (0,1)$ or $(0,2)$  the same decomposition works provided the indecomposable term is replaced by something simpler. For discs, this is just the disc face contribution $T_{0,1}$; for annuli, this is the annulus without internal face.  We record the resulting formulae separately.

\begin{lem}
\label{lem:R1W:unstable}
	For $(g,n)=(0,1)$ and $(0,2)$,
	\begin{equation}
	\label{R1W:unstable}
	\begin{split}
		\mathcal{R}_{x_1}W_{0,1}^{\ast}(x_1)
		&=
		-
		x_1
		+
		T_{0,1}'(x_1)
		+
		\widetilde{\mathcal{O}}_{x_1}W^\ast_{0,1}(x_1), \\
		\mathcal{R}_{x_1}W_{0,2}^{\ast}(x_1,x_2)
		&=
		\frac{1}{(x_2-x_1)^2}
		+
		\mathcal{O}_{x_1}W^\ast_{0,2}(x_1,x_2), \\
		\mathcal{R}_{x_1}\widetilde{W}^{\ast}_{0,2}(x_1,x_2)
		&=
		\mathcal{O}_{x_1}W^{\ast}_{0,2}(x_1,x_2).
	\end{split}
	\end{equation}
	Here $\widetilde{\mathcal{O}}$ is defined as $\mathcal{O}$, with the annular potential $O$ replaced by $\widetilde{O}$.
\end{lem}

\begin{proof}
	For $(g,n) = (0,1)$, the special face is either a planar disc face, giving $T_{0,1}'(x_1)$, or a planar face with at least two boundary components, whose remaining boundaries are filled by stuffed discs. In both cases, the derivative with respect to $x_1$ roots the boundary glued to $\partial_1$. Since in the second case all the remaining boundaries play a symmetric role (instead of all but one in \Cref{lem:R1W:stable}), the contribution is $\widetilde{\mathcal{O}}_{x_1}W^\ast_{0,1}(x_1)$. The term $-x_1$ is the initial term in the definition of $\mathcal{R}_{x_1}W_{0,1}^{\ast}$.

	For $(g,n) = (0,2)$, $\partial_1$ may be adjacent to $\partial_2$. This gives
	\begin{equation}
		\sum_{\ell\geq 1} \ell \frac{x_1^{\ell-1}}{x_2^{\ell+1}}
		=
		\frac{1}{(x_2-x_1)^2},
	\end{equation}
	where the factor $\ell$ records the possible roots on the second boundary. The remaining contribution is obtained by gluing $\partial_1$ to a planar internal face, with one other boundary component carrying the annulus and all remaining ones filled by discs without preferred order. The derivative in the definition of $\mathcal{O}_{x_1}$ puts a root on the face glued to $\partial_1$, hence this contribution is $\mathcal{O}_{x_1}W^\ast_{0,2}(x_1,x_2)$. Excluding the boundary-face contribution gives the formula for $\mathcal{R}_{x_1}\widetilde{W}^{\ast}_{0,2}(x_1,x_2)$.
\end{proof}

It remains to make the indecomposable contribution explicit. By definition, a $\partial_1$-indecomposable stuffed map is a $\partial_1$-special map in which $\partial_1$ is adjacent to an internal face and Tutte's procedure at $\partial_1$ changes topology. Equivalently, in the notations of the internal-face case, the outer map is a disc, while the remaining components carry all the non-disc topology.

\begin{lem}
\label{lem:V:explicit}
	For $2g-2+n>0$,
	\begin{equation}
	\label{V:explicit}
		V_{g,n}(x_1;\bm{x}_I)
		=
		\sum_{\substack{h \geq 0 \\ m \geq 1}}
		\sum_{\substack{1 \le c \le m \\ (\bm{p},\bm{K},\bm{J}) \in \mathcal{T}_{g,n;h,m;c}^{\bullet}}}
		\Bigg\langle
			\frac{T_{h,m}(x_1,\xi_2,\ldots,\xi_m)}{(c-1)!}
			\prod_{b = 2}^{c}
				W^{\ast}_{p_b,\card{K_b}+\card{J_b}}(\bm{\xi}_{K_b},\bm{x}_{J_b})
		\Bigg\rangle_{\bm{\xi}},
	\end{equation}
	where $\mathcal{T}_{g,n;h,m;c}^{\bullet}$ is the subset of triples in $\mathcal{T}^{>}_{g,n;h,m;c}$ such that $(p_1,K_1,J_1)=(0,\emptyset,\emptyset)$.
\end{lem}

\begin{proof}
	The formula is a direct translation of the internal-face decomposition for the topology-changing cases. The condition $(p_1,K_1,J_1)=(0,\emptyset,\emptyset)$ says that the outer component, namely the one attached to the boundary containing $\partial_1$, is a disc. The remaining components record the non-disc part of the map, and the symmetry factor $(c-1)!$ accounts for the lack of preferred ordering among them.
\end{proof}

Combining \Cref{lem:R1W:stable,lem:R1W:unstable,lem:V:explicit}, the regular parts, hence the formal monodromies, and the indecomposable contribution in the excision formula of \Cref{thm:excision} are explicit. More precisely, they are expressed in terms of the stuffed potentials $T_{h,m}$ and stuffed generating series $W^{\ast}_{g',n'}$ with $2g'-2+n' < 2g-2+n$. Together with the previously obtained expression for $S^\ast$, this makes the recursion closed.

\subsubsection{Skin enumeration and formal discontinuity}
We finally rewrite the skin formula in a form closer to the usual expression of the kernel in topological recursion. To do so, it is convenient to introduce one further combination.

\begin{defn}
\label{def:dscnt}
	Given $i \in \set{1,\ldots,n}$, we define the \emph{formal discontinuity} by
	\begin{equation}
		\mathcal{D}_{x_i}W_{g,n}^{\ast}(x_1,\ldots,x_n)
		\coloneqq
		W_{g,n}^{\ast}(x_1,\ldots,x_n)
		+
		\frac{1}{2}\mathcal{R}_{x_i}W_{g,n}^{\ast}(x_1,\ldots,x_n).
	\end{equation}
\end{defn}

As for the formal monodromy, $\mathcal{D}_{x_i}$ is not an operator in the strict sense: it is only defined here on the generating series for which $\mathcal{R}_{x_i}$ has been defined. The terminology anticipates the analytic interpretation of the next section, where this expression is related to the discontinuity across the cut.

\begin{thm}[Skin enumeration, revisited]
\label{revisit:skin}
	The generating series of stuffed skin maps takes the form
	\begin{equation}
		S^{\ast}(x_1,x)
		=
		\frac{\displaystyle\int^{x}\mathcal{D}_{y}W^{\ast}_{0,2}(x_1,y)\,\dd y}{-\mathcal{D}_{x} W_{0,1}^{\ast}(x)},
	\end{equation}
	where the antiderivative is chosen so that
	\begin{equation}
		2\int^{x}\mathcal{D}_{y}W^{\ast}_{0,2}(x_1,y)\,\dd y
		=
		\int_{\infty}^{x}
		\left(
			2W_{0,2}^{\ast}(x_1,y)
			+
			\frac{1}{(x_1-y)^2}
		\right)\dd y
		+
		\big\langle O(x,\xi)W_{0,2}^{\ast}(x_1,\xi)\big\rangle_{\xi}
		-
		\partial_q W_{0,1}^{\ast}(x_1).
	\end{equation}
\end{thm}

\begin{proof}
	By \Cref{lem:R1W:unstable} and the definition of formal discontinuity, we have
	\begin{equation}
	\begin{split}
		2\mathcal{D}_{y}W_{0,2}^{\ast}(x_1,y)
		&=
		2W_{0,2}^{\ast}(x_1,y)
		+
		\frac{1}{(x_1-y)^2}
		+
		\mathcal{O}_{y}W_{0,2}^{\ast}(x_1,y), \\
		2\mathcal{D}_{x}W_{0,1}^{\ast}(x)
		&=
		2W_{0,1}^{\ast}(x)
		+
		T_{0,1}'(x)
		+
		\widetilde{\mathcal{O}}_{x}W_{0,1}^{\ast}(x)
		-
		x.
	\end{split}
	\end{equation}
	With the choice of antiderivative specified above, the first identity gives the numerator of \Cref{skin:enum:stuf}. The denominator in \Cref{skin:enum:stuf} is $-2\mathcal{D}_{x}W_{0,1}^{\ast}(x)$. Dividing numerator and denominator by $2$ gives the stated formula.
\end{proof}

\subsection{Excision as a Mirzakhani--McShane type identity}
\label{ssec:Mir}
As promised in the introduction, we now discuss the analogy with Mirzakhani's pair-of-pants decomposition for hyperbolic surfaces \cite{Mir07}.

Given a hyperbolic surface with geodesic boundaries, Mirzakhani starts from a point on the first boundary $\partial_1$ and shoots a geodesic $\gamma$ orthogonal to it. She shows that, outside a measure-zero set of starting points, $\gamma$ either intersects itself, returns to $\partial_1$, or hits or spirals around another boundary $\partial_i$. She terminates the geodesic the first time one of these events occurs, and distinguishes the following cases:
\begin{itemize}
	\item The geodesic $\gamma$ intersects either itself or $\partial_1$. Thickening $\partial_1\cup\gamma$ determines the homotopy class of an embedded pair of pants, one of whose boundaries is $\partial_1$. If $\gamma$ intersects itself, one of the two boundaries $\partial'$ is homotopic to the loop part of $\gamma$, while $\partial''$ is freely homotopic to the loop $\partial_1\cup\gamma$; see the left panel of~\Cref{fig:lasso}. If $\gamma$ ends on $\partial_1$, the other two boundaries $\partial',\partial''$ are freely homotopic to the loops obtained by concatenating either side of $\gamma$ with a segment of $\partial_1$; see the middle panel of~\Cref{fig:lasso}.

	\item The geodesic $\gamma$ hits or spirals around another boundary $\partial_i$. In this case, one instead defines the homotopy class of an embedded pair of pants by thickening $\partial_1\cup\gamma\cup\partial_i$. Its boundaries are $\partial_1$, $\partial_i$, and a third boundary $\partial''$ homotopy equivalent to the loop $\partial_1\cup\gamma\cup\partial_i$; see the right panel of~\Cref{fig:lasso}.
\end{itemize}
Mirzakhani then takes the representative of this homotopy class with geodesic boundaries. This decomposition yields a partition of unity, obtained by measuring the set of initial points giving rise to each homotopy class of pairs of pants, commonly called the ``Mirzakhani--McShane identity''. Integrating this identity over all hyperbolic structures, she derives a topological recursion for the Weil--Petersson volumes.

\begin{figure}[h!]
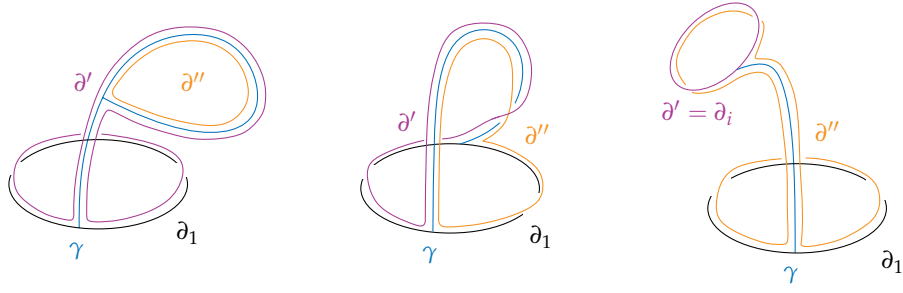

\begin{center}

	\caption{The boundaries of Mirzakhani's pair of pants, up to homotopy.}
	\label{fig:lasso}
\end{center}
\end{figure}

We now explain how the skinning process can be reinterpreted along the lines of Mirzakhani's construction. We discuss only ordinary maps and then briefly indicate how the picture adapts to stuffed maps. In our setting, we can likewise define a path $\gamma$ in the initial map $\mathfrak{m}$, emanating from $\partial_1$ and starting at the root edge, whose thickening produces a pair of pants $\mathfrak{p}$ realised as a submap of $\mathfrak{m}$, namely the same pair of pants as in \Cref{ssec:skinning:to:pop}. The path consists of the edges that carry the root of the first boundary at some stage of the skinning procedure. We now describe how to construct this path without referring to skinning. Recall that cutting along a sequence of edges is inverse to gluing. We say that two edges $e,e'$ are \emph{neighbours} if they are adjacent to a common internal or boundary face and the corresponding half-edges are consecutive along that face.

The path $\gamma$ is given by a sequence $e_1,\ldots,e_\ell$ of edges such that $e_i$ and $e_{i+1}$ are neighbours for each $i \in \set{1,\ldots,\ell-1}$. It starts with $e_1=r_1$, the root edge of $\partial_1$. The edge $e_{i+1}$ is obtained from $e_i$ as follows.
\begin{itemize}
	\item Assume that $e_i$ is adjacent to an internal face $\mathfrak{f}$ that has not yet been met by the path, i.e. it is not adjacent to any of $e_1,\ldots,e_{i-1}$. Then $e_{i+1}$ is the edge preceding $e_i$ in $\mathfrak{f}$. This is the analogue of case (I) of Tutte.
	
	\item Assume that $e_i$ is adjacent to two internal or boundary faces $\mathfrak{f},\mathfrak{f}'$ that have already been met by the path. Up to exchanging them, we may assume that $\mathfrak{f}$ is adjacent to $e_{i-1}$ and that $\mathfrak{f}'$ is adjacent to $e_j$, where $j \in \set{1,\ldots,i-1}$ is chosen maximal. Assume furthermore that cutting along the path $e_j,\ldots,e_i$ separates $\mathfrak{m}$ into a disc and a map of topology $(g,n)$. We then define $e_{i+1}$ to be the edge of $\mathfrak{f}$ neighbouring $e_i$ in the component of topology $(g,n)$. We use the same rule for $i=1$ if $r_1$ bounds $\partial_1$ on both sides and separates $\mathfrak{m}$ into a disc and a map of topology $(g,n)$. This is the analogue of case (D${}^=$) of Tutte.
\end{itemize}
The path terminates at $e_i=e_\ell$ in either of the following situations.
\begin{itemize}
	\item The edge $e_i$ is adjacent to a boundary $\partial_j\neq\partial_1$. This is the analogue of case (R) of Tutte.
	
	\item The edge $e_i$ is adjacent on both sides to a boundary or face that has already been met by the path, and the path $e_j,\ldots,e_i$, defined as above, separates $\mathfrak{m}$ into two maps, both of topology different from $(g,n)$. This includes the case $i=1$, when $r_1$ bounds $\partial_1$ on both sides and separates $\mathfrak{m}$ into two maps of topology different from $(g,n)$. This is the analogue of case (D${}^{>}$) of Tutte.
\end{itemize}

Note that both the skin $\mathfrak{s}$ of \Cref{ssec:skin} and the pair of pants $\mathfrak{p}$ of \Cref{ssec:skinning:to:pop} can be reconstructed as submaps of $\mathfrak{m}$ solely from the path $\gamma$. The set $\mathfrak{F}$ of faces of $\mathfrak{s}$ is obtained by thickening $\gamma$, that is, by collecting all faces adjacent to one of the edges $e_i$, together with all discs separated by some segment $e_i,\ldots,e_j$ of the path. The edges of $\mathfrak{s}$ consist of the edges adjacent to faces in $\mathfrak{F}$, together with the remaining edges of $\partial_1$. This describes $\mathfrak{s}$ as a submap of $\mathfrak{m}$.

To obtain $\mathfrak{s}$ itself, that is, to skin $\mathfrak{m}$, we cut the map along the edges in $\partial_1\setminus(\gamma\setminus\{e_\ell\})$ and along the edges adjacent to faces in $\mathfrak{F}$ but not belonging to $\gamma\setminus\{e_\ell\}$. This produces $\partial_2\mathfrak{s}$ and $\partial_1\hat{\mathfrak{m}}$ as new boundaries.

To obtain $\mathfrak{p}$, we proceed differently depending on the terminating step. As for $\mathfrak{s}$, its set of faces is $\mathfrak{F}$, and the edges of $\partial_1$ are included and form $\partial_1\mathfrak{p}$. In case (R), we also include the edges of $\partial_i$, which form $\partial'\mathfrak{p}$, while $\partial''\mathfrak{p}$ consists of the edges in $\partial_i\setminus\{e_\ell\}$ together with those in $\partial_2\mathfrak{s}\setminus\{e_\ell\}$. To excise the pair of pants, we simply cut $\mathfrak{m}$ along $\partial''\mathfrak{p}$. In case (D${}^{>}$), the edges of $\partial_2\mathfrak{s}$ form two loops connected by the last edge $e_\ell$. We declare $\partial'\mathfrak{p}$ and $\partial''\mathfrak{p}$ to correspond to these two loops, and obtain the pair-of-pants excision by cutting $\mathfrak{m}$ along these two boundaries. Notice that, in case (R), this does not include the trivial skin in $\check{\mathfrak{m}}$, unlike the construction of \Cref{ssec:skinning:to:pop}.

\begin{figure}[H]
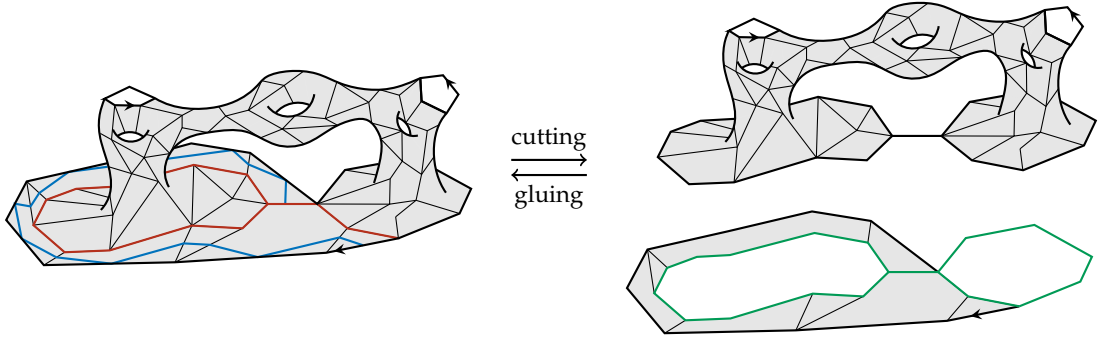

	\centering

	\caption{Example of the path $\gamma$, in blue, in a map terminating in case (D). The red edges are those adjacent to a face in $\mathfrak{F}$ but not belonging to $\gamma\setminus\set{e_\ell}$. The green edges form the union of this set and the set of edges of $\partial_1\setminus(\gamma\setminus\{e_\ell\})$.}
	\label{fig:Mir}
\end{figure}

The similarity with Mirzakhani's procedure is now apparent. We terminate the path in the same situations, namely when it reaches a different boundary $\partial_i$ (there is no spiralling phenomenon in the discrete setting), and when it meets either itself or $\partial_1$. Furthermore, the homotopy class of the pair of pants is obtained by the same thickening of $\gamma$. In our setting, we even obtain a distinguished representative, without passing to homotopy. The bijection underlying the pair-of-pants excision of \Cref{ssec:skinning:to:pop} is the analogue of the Mirzakhani--McShane identity. Integration is replaced by counting the elements of the corresponding sets of maps, and gives rise to the topological recursion of \Cref{intro:excision}.

To conclude, let us briefly discuss the case of stuffed maps. One can still define $\mathfrak{p}$ through the path $\gamma$ of root edges. Two additional situations may occur while developing the path. When the path meets the boundary of an internal face $\mathfrak{f}$ that is not a disc but is not topology-changing (in the sense that it has topology $(0,k)$ and is glued to $k-1$ stuffed discs), we extend the path across $\mathfrak{f}$ to an arbitrary point $o$ on the remaining boundary. The choice of $o$ is analogous to the arbitrary choice of a root on this boundary: different choices lead to isomorphic maps $\mathfrak{m}$, although they produce different pairs of pants. If the path meets the boundary $\delta$ of an internal face $\mathfrak{f}$ that is not a disc and is topology-changing, then we terminate it there. The pair of pants $\mathfrak{p}$ is then constructed from $\gamma$ exactly as in case (R), except that $\delta$ plays the role of $\partial_i$. The case in which $\partial_1$ is homotopic to $\delta$ is precisely our indecomposable case and has no analogue in hyperbolic geometry: by the Gauss--Bonnet theorem, hyperbolic cylinders with geodesic boundaries do not exist.

\section{The analytic topological recursion}
\label{sec:analytic}

When the generating series are sufficiently convergent and analytically well behaved, we show that the excision formulae of \Cref{thm:excision} imply the topological recursion of \cite{Eyn04,Eyn16} for maps, the topological recursion of \cite{BE11,BEO15} for maps with tubes, and the blobbed topological recursion of \cite{Bor14,BS17} for stuffed maps. As in \Cref{sec:pop}, we carry out the discussion at the general level of stuffed maps.

\subsection{A model case of analytic continuation}
We begin with a warm-up on a model case of analytic continuation. Let $[b,a] \subset \mathbb{R}$ be a compact interval with non-empty interior, and let $\widehat{\mathbb{C}}$ be the Riemann sphere. The Zhukovsky change of variable
\begin{equation}
	X(z)
	=
	\frac{a+b}{2}
	+
	\frac{a-b}{4}\bigg(z+\frac{1}{z}\bigg)
\end{equation}
defines a biholomorphic map from the open unit disc $\mathbb{D}$, and also from $\widehat{\mathbb{C}} \setminus \overline{\mathbb{D}}$, to $\widehat{\mathbb{C}} \setminus [b,a]$. The inverse maps are
\begin{equation}
	Z_{\pm}(x)
	=
	\frac{2}{a-b}
	\left(
		x-\frac{a+b}{2}
		\pm
		\sqrt{(x-a)(x-b)}
	\right).
\end{equation}
The images of the two copies of $\widehat{\mathbb{C}} \setminus [b,a]$ are exchanged by the involution $z \mapsto 1/z$. The map $X$ also extends continuously from the unit circle $\mathbb{U}$ to $[b,a]$. In particular, $z=1$ is mapped to $a$ and $z=-1$ is mapped to $b$, while $z=0$ and $z=\infty$ are mapped to $x=\infty$.

Consider a model situation in which we are given a holomorphic function $F \colon \mathbb{C} \setminus [b,a] \rightarrow \mathbb{C}$ such that
\begin{itemize}
	\item the limits $F(x^{\pm}) \coloneqq \lim_{\epsilon \rightarrow 0^+} F(x \pm \ii\epsilon)$ exist and depend continuously on $x \in (b,a)$;

	\item there exists an open neighbourhood $\Omega \subset \mathbb{C}$ of $[b,a]$ such that $R(x) \coloneqq F(x^+) + F(x^-)$ extends to a meromorphic function on $\Omega$, with possible poles located at $\set{-1,+1}$;

	\item there exists $\kappa > 0$ such that $|(x-a)(x-b)|^{\kappa}|F(x)|$ is bounded on $\Omega \setminus [b,a]$.
\end{itemize}
Then
\begin{equation}
	f(z)
	=
	\begin{cases}
		F(X(z)),
		& |z| \geq 1, \\
		R(X(z)) - F(X(z)),
		& |z| < 1 \text{ and } z \in X^{-1}(\Omega),
	\end{cases}
\end{equation}
extends to a meromorphic function on the domain $\Sigma \coloneqq (\mathbb{C} \setminus \mathbb{D}) \cup X^{-1}(\Omega)$, whose poles are contained in $\set{-1,+1}$, and satisfies
\begin{equation}
	\forall z \in X^{-1}(\Omega),
	\qquad
	f(z) + f\bigl(\tfrac{1}{z}\bigr) = R(X(z)).
\end{equation}
Equivalently, we may work with the $1$-form $F(x)\dd x$, which extends to the meromorphic $1$-form $f(z)\dd X(z)$ on $\Sigma$, with poles only at $z=\pm 1$. We call $\Sigma_{\textnormal{out}} \coloneqq \mathbb{C} \setminus \overline{\mathbb{D}}$ the \emph{outer sheet} and define $\Omega_{\textnormal{out}} \coloneqq X^{-1}(\Omega) \setminus \overline{\mathbb{D}}$. We call $\Sigma_{\textnormal{in}} \coloneqq \Omega_{\textnormal{in}} \coloneqq X^{-1}(\Omega) \cap \mathbb{D}$ the \emph{inner sheet}, see~\Cref{fig:Zhukovsky}.

\begin{figure}
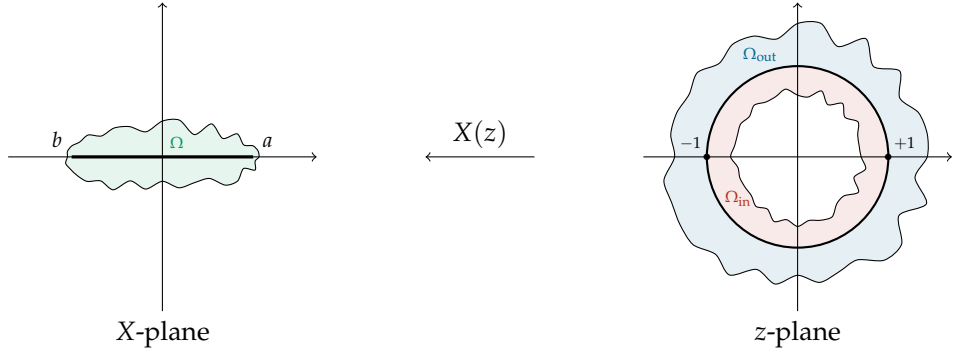

\centering

\caption{The Zhukovsky map, its two sheets, and the preimages of the domain $\Omega$.}
\label{fig:Zhukovsky}
\end{figure}

A toy model is $F(x)=\sqrt{(x-a)(x-b)}$, with a branch chosen to be univalued in $\mathbb{C}\setminus [b,a]$. Then $F(x^+) + F(x^-)=0$ on $(b,a)$ and, after the Zhukovsky change of variable, $f(z) = \frac{a-b}{4}(z-\frac{1}{z})$ becomes meromorphic in $z$. This illustrates how the parametrisation resolves a square-root branch cut.

We remark that
\begin{equation}
	X'(z)
	=
	\frac{1}{Z'(X(z))}
	\mathop{\sim}_{z \rightarrow \pm 1}
	\pm \frac{1}{\sqrt{(x-a)(x-b)}}.
\end{equation}

\subsection{Analytically good weights}
In the rest of this section, we restrict to assignments of weights with good analytic behaviour: the generating series converge, admit the required analytic continuations, and become meromorphic multidifferentials after the Zhukovsky change of variables. We make this precise in the following definition.

\begin{defn}
\label{def:anagood}
	We say that an assignment of non-negative values to the vertex and face weights $(q,\bm{t})$ is \emph{analytically good} if there exist $b < 0 < a$ and an open disc $\Omega$ centred at $0$ and containing $[b,a]$ such that the following properties hold.

	\smallskip

	\textit{(A) Convergence.}
	\begin{itemize}
		\item[(A1)] For any $g \geq 0$ and $n \geq 1$, the coefficients $W_{g;\ell_1,\ldots,\ell_n}^{\ast}$ are well-defined by absolutely convergent series, and the generating series $W_{g,n}^{\ast}(x_1,\ldots,x_n)$ in the boundary degrees has non-zero radius of convergence.

		\item[(A2)] The series $O(x,\xi)$ and $\widetilde{O}(x,\xi)$ defined in \Cref{ann:pot} have non-zero radius of convergence.

		\item[(A3)] For any $g \geq 0$ and $n \geq 1$ such that $2g-2+n>0$, the series $V_{g,n}(x_1;x_2,\ldots,x_n)$ defined in~\eqref{eq:indecomp} has non-zero radius of convergence.
	\end{itemize}

	\textit{(B) Analytic continuation and holomorphy.}
	\begin{itemize}
		\item[(B1)] For any $g \geq 0$ and $n \geq 1$, $W_{g,n}^{\ast}(x_1,\ldots,x_n)$ defines a holomorphic function of $x_i \in \mathbb{C} \setminus [b,a]$.

		\item[(B2)] The series $O(x,\xi)$ and $\widetilde{O}(x,\xi)$ define holomorphic functions on $\Omega^2$.

		\item[(B3)] For any $g \geq 0$ and $n \geq 1$ such that $2g - 2 + n > 0$, $V_{g,n}(x_1;x_2\ldots,x_n)$ defines a holomorphic function of $x_1 \in \Omega$ and $x_2,\ldots,x_n \in \mathbb{C} \setminus [b,a]$.
	\end{itemize}

	\textit{(C) Analytic behaviour in Zhukovsky variables.}
	\begin{itemize}
		\item[(C1)] For any $g \geq 0$ and $n \geq 1$, the multidifferential
		\begin{equation}
			\omega_{g,n}^{\ast}(z_1,\ldots,z_n)
			\coloneqq
			\bigg(
				W^{\ast}_{g,n}(X(z_1),\ldots,X(z_n))
				+
				\frac{\delta_{g,0}\delta_{n,2}}{(X(z_1) - X(z_2))^2}
			\bigg)
			\prod_{i = 1}^{n} \dd X(z_i),
		\end{equation}
		which is defined for $z_1,\ldots,z_n \in \mathbb{C} \setminus \overline{\mathbb{D}}$ by \textnormal{(B1)}, admits an analytic continuation to a meromorphic multidifferential on $\Sigma^n$, where $\Sigma \coloneqq (\mathbb{C} \setminus \mathbb{D}) \cup X^{-1}(\Omega)$. Moreover, for $2g-2+n>0$, its only singularities in $\Sigma^n$ are poles located at $z_i=\pm 1$, for $i=1,\ldots,n$.

		\item[(C2)] The $1$-form $\omega_{0,1}^{\ast}(z)$ is holomorphic in $X^{-1}(\Omega)$, and the only zeros of $\omega_{0,1}^{\ast}(z)-\omega_{0,1}^{\ast}(1/z)$ in $X^{-1}(\Omega)$ are double zeros at $z=\pm 1$.

		\item[(C3)] The bidifferential $\omega_{0,2}^{\ast}(z_1,z_2)-\frac{\dd z_1 \dd z_2}{(z_1-z_2)^2}$ is holomorphic on $(X^{-1}(\Omega))^2 \subset \Sigma^2$.
	\end{itemize}
\end{defn}


There are large classes of assignments of weights which are analytically good; see \cite{Eyn16} for maps, \cite{BEO15} for maps with tubes, and \cite{Bor14} for stuffed maps. For instance, one may assign sufficiently small real values to finitely many weights and set all the others to zero.

For maps, the model situation above holds for the disc generating series $W_{0,1}(x)$, for off-critical weights, by the one-cut lemma; see \cite[Section 6]{BBG12b}. For maps with tubes, the same conclusion holds for the disc generating series $W_{0,1}^{\circ}(x)$, since the substitution approach of \cite{BBG12a,BBG12b} identifies it with $W_{0,1}(x)$ for renormalised weights satisfying a fixed-point equation. A similar substitution approach also applies to the stuffed disc \cite[Section 3.2]{Bor14}, with a fully nonlinear fixed-point equation.

For maps, the phase diagram, and in particular the off-critical region, in the space of Boltzmann weights is relatively well studied. In contrast, it is very difficult to describe the phase diagram for maps with tubes, and even more so for stuffed maps; it is well understood only for a few families of models, such as self-avoiding rigid loops on quadrangulations \cite{Bud18} or self-avoiding loops on triangulations \cite{Kho22}. For this reason, we prefer to take as assumptions the analytic properties \textnormal{(B2)} and \textnormal{(B3)}, together with the off-criticality condition \textnormal{(C1)}. The latter means that
\begin{equation}
	\rho(x)
	\coloneqq
	\frac{W_{0,1}^{\ast}(x^-) - W_{0,1}^{\ast}(x^+)}{2\pi\ii}
\end{equation}
has no zeros in $(b,a)$, and vanishes exactly like a square root as $x \rightarrow a$ or $x \rightarrow b$. In \cite{Bor14}, the properties in \Cref{def:anagood} were combined with Tutte's equation, reproduced here as \Cref{Tutte:lemma:stuffed}, to prove the analytic continuation in \textnormal{(B1)}, to locate the singularities as in \textnormal{(C2)} and \textnormal{(C3)}, and to establish the linear loop equations, namely \Cref{lem:lloop} below.

Determining optimal assumptions on the weight values that guarantee \Cref{def:anagood} is an interesting but difficult problem. It is more demanding than determining the phase diagram, which only concerns the analytic behaviour of $W_{0,1}(x)$. In what follows, we take analytically good weights as input and explain how the formal excision formulae imply blobbed topological recursion.

\subsection{Analytic properties of the formal monodromy and the skin generating series}
Since we shall need them shortly, we recall the linear loop equations for the reader's convenience.

\begin{lem}[Linear loop equations {\cite[Lemma~3.3 and Theorem~4.2]{Bor14}}]
\label{lem:lloop}
	Fix analytically good real values for the face and vertex weights. Then, for any $g \geq 0$ and $n \geq 1$, the following relations hold for $x_1 \in (b,a)$ and $x_2,\ldots,x_n \in \mathbb{C} \setminus [b,a]$.

	For $(g,n) = (0,1)$, we have
	\begin{equation}
	\label{lloop1}
		W_{0,1}^{\ast}(x_1^+) + W_{0,1}^{\ast}(x_1^-)
		+
		\widetilde{\mathcal{O}}_{x_1}W_{0,1}^{\ast}(x_1)
		=
		x_1 - \partial_{x_1} T_{0,1}(x_1).
	\end{equation}
	For $(g,n) \neq (0,1)$, we have
	\begin{equation}
		W_{g,n}^{\ast}(x_1^+,\bm{x}_I)
		+
		W_{g,n}^{\ast}(x_1^-,\bm{x}_I)
		+
		\frac{\delta_{g,0}\delta_{n,2}}{(x_2 - x_1)^2}
		+
		\mathcal{O}_{x_1}W_{g,n}^{\ast}(x_1,\bm{x}_I)
		+
		\partial_{x_1}V_{g,n}(x_1;\bm{x}_I)
		=
		0.
	\end{equation}
	By \Cref{def:anagood}, all quantities in these equations are well-defined in the specified range.\footnote{
		In \cite{Bor14}, the term $\frac{1}{(x_2 - x_1)^2}$ is absorbed into the definition of the corresponding $(0,2)$ term, whereas in our conventions it is written separately.
	}
\end{lem}

Comparing the linear loop equations with \Cref{lem:R1W:stable,R1W:unstable} and \Cref{def:mndrm,def:dscnt} leads to three observations, which justify the names \emph{regular part}, \emph{formal monodromy}, and \emph{formal discontinuity}.
\begin{itemize}
	\item The analytic continuation of the regular part $\mathcal{R}_{x_1}W_{g,n}^{\ast}(x_1,\bm{x}_I)$ to $x_1 \in (b,a)$ and $x_2,\ldots,x_n \in \mathbb{C} \setminus [b,a]$ coincides with
	\begin{equation}
		-\bigl(W_{g,n}^{\ast}(x_1^+,\bm{x}_I) + W_{g,n}^{\ast}(x_1^-,\bm{x}_I)\bigr).
	\end{equation}

	\item The analytic continuation of the formal monodromy $\mathcal{M}_{x_1}W_{g,n}^{\ast}(x_1,\bm{x}_I)$ in the same range coincides with
	\begin{equation}
		-W_{g,n}^{\ast}(x_1^-,\bm{x}_I).
	\end{equation}

	\item The analytic continuation of the formal discontinuity $\mathcal{D}_{x_1}W_{g,n}^{\ast}(x_1,\bm{x}_I)$ in the same range coincides with
	\begin{equation}
		\frac{1}{2}\bigl(W_{g,n}^{\ast}(x_1^+,\bm{x}_I) - W_{g,n}^{\ast}(x_1^-,\bm{x}_I)\bigr).
	\end{equation}
\end{itemize}

Let $w_{g,n}^{\ast}(z_1,\ldots,z_n) \coloneqq W_{g,n}^{\ast}(X(z_1),\ldots,X(z_n))$ for $z_1,\ldots,z_n$ in the outer sheet $\Sigma_{\textnormal{out}}$. By assumption \textnormal{(C1)}, this admits an analytic continuation to a meromorphic function for $z_1,\ldots,z_n \in \Sigma$. We can reformulate the previous observations as follows. Let $\bm{z}_I \in \Sigma_{\textnormal{out}} \setminus \overline{X^{-1}(\Omega)}$.
\begin{itemize}
	\item For $z_1 \in X^{-1}(\Omega)$, we have
	\begin{equation}
		\mathcal{R}_{x_1}W_{g,n}^{\ast}(X(z_1),X(\bm{z}_I))
		=
		-\bigl(w^{\ast}_{g,n}(z_1,\bm{z}_I) + w_{g,n}^{\ast}(1/z_1,\bm{z}_I)\bigr).
	\end{equation}

	\item The function $\mathcal{M}_{x_1}W_{g,n}^{\ast}(X(z_1),X(\bm{z}_I))$, initially defined for $z_1 \in \Omega_{\textnormal{out}}$, admits the analytic continuation
	\begin{equation}
		-w_{g,n}^{\ast}(1/z_1,\bm{z}_I)
	\end{equation}
	for $z_1 \in X^{-1}(\Omega)$.

	\item The function $\mathcal{D}_{x_1}W_{g,n}^{\ast}(X(z_1),X(\bm{z}_I))$, initially defined for $z_1 \in \Omega_{\textnormal{out}}$, admits the analytic continuation
	\begin{equation}
		\frac{1}{2}\bigl(w_{g,n}^{\ast}(z_1,\bm{z}_I) - w_{g,n}^{\ast}(1/z_1,\bm{z}_I)\bigr)
	\end{equation}
	for $z_1 \in X^{-1}(\Omega)$.
\end{itemize}
In other words, after analytic continuation in the Zhukovsky variable, we have the correspondence
\begin{equation}
\label{intermono}
	F(x) \longmapsto \mathcal{M}_{x}F(x)
	\qquad \longleftrightarrow \qquad
	f(z) \longmapsto -f(1/z).
\end{equation}
This principle is valid for the assignment $W_{g,n}^{\ast} \mapsto \mathcal{M}_{x_i}W_{g,n}^{\ast}$ with respect to any variable $x_i$. For $\widetilde{W}_{0,2}^{\ast}$, there is a shift; see~\Cref{def:mndrm}. Namely, after analytic continuation in the Zhukovsky variables, we have
\begin{equation}
\label{monotilde}
	\widetilde{W}_{0,2}^{\ast}(x_1,x_2)
	\longmapsto
	\mathcal{M}_{x_2}\widetilde{W}_{0,2}^{\ast}(x_1,x_2)
	\qquad
	\longleftrightarrow
	\qquad
	w_{0,2}^{\ast}(z_1,z_2)
	\longmapsto
	-\left(
		w_{0,2}^{\ast}(z_1,1/z_2)
		+
		\frac{1}{(x_1 - x_2)^2}
	\right).
\end{equation}
This perspective clarifies the formula of \Cref{revisit:skin} for the skin generating series and allows us to determine its analytic properties.

\begin{lem}
\label{kernel:an}
	Assume that the face and vertex weights $(q,\bm{t})$ are analytically good. Then, for all $x_1 \in \mathbb{C} \setminus [b,a]$, we have
	\begin{equation}
		\bigg(\int_{\ii 0^+}^{\infty} + \int_{\ii 0^-}^{\infty}\bigg)
		W_{0,2}^{\ast}(x_1,y)\dd y
		=
		\frac{1}{x_1} - \partial_q W_{0,1}^{\ast}(x_1),
	\end{equation}
	independently of the choice of paths of integration in $\widehat{\mathbb{C}} \setminus [b,a]$. Moreover, $S^{\ast}(x_1,x)$ defines a holomorphic function of $x_1 \in \mathbb{C} \setminus \overline{\Omega}$ and $x \in \Omega$.

	Besides, $s^{\ast}(z_1,z) \coloneqq S^{\ast}(X(z_1),X(z))$, initially defined for $z_1 \in \Sigma_{\textnormal{out}}$ and $z \in X^{-1}(\Omega)$, admits an analytic continuation to a meromorphic function of $z_1 \in \Sigma$ and $z \in X^{-1}(\Omega)$ given by
	\begin{equation}
		s^{\ast}(z_1,z)
		=
		\frac{
			-\displaystyle\int_{1/z}^{z} w_{0,2}^{\ast}(z_1,\zeta)\dd X(\zeta)
		}{
			w_{0,1}^{\ast}(z) - w_{0,1}^{\ast}(1/z)
		}.
	\end{equation}
	This function is invariant under $z \mapsto 1/z$ and has no singularity when $z_1 \in \Sigma \setminus \overline{X^{-1}(\Omega)}$ and $z \in X^{-1}(\Omega)$.
\end{lem}

\begin{proof}
	Set
	\begin{equation}
	\begin{split}
		N^{\ast}(x_1,x)
		&\coloneqq
		\int_{\infty}^{x} 2W_{0,2}^{\ast}(x_1,y)\dd y
		+
		\frac{1}{x_1 - x}
		+
		\big\langle O(x,\xi) W_{0,2}^{\ast}(x_1,\xi)\big\rangle_{\xi}
		-
		\partial_q W_{0,1}^{\ast}(x_1),
		\\
		D^{\ast}(x)
		&\coloneqq
		x - \partial_{x} T_{0,1}(x)
		-
		2W_{0,1}^{\ast}(x)
		-
		\bigl\langle \partial_{x}\widetilde{O}(x,\xi) W_{0,1}^{\ast}(\xi)\bigr\rangle_{\xi}.
	\end{split}
	\end{equation}
	According to \Cref{skin:enum:stuf}, we have the equality of formal series, with negative powers of $x_1$ and all integer powers of $x$,
	\begin{equation}
		S^{\ast}(x_1,x)D^{\ast}(x) = N^{\ast}(x_1,x).
	\end{equation}
	By \textnormal{(B1)} and \textnormal{(B2)}, the series defining $S^{\ast}(x_1,x)D^{\ast}(x)$ and $D^{\ast}(x)$ converge to holomorphic functions of $x_1 \in \mathbb{C} \setminus \overline{\Omega}$ and $x \in \Omega \setminus [b,a]$. The linear loop equation for $(g,n) = (0,1)$ yields the relation between boundary values
	\begin{equation}
	\label{discoD}
		\forall x \in (b,a),
		\qquad
		D^{\ast}(x^+) + D^{\ast}(x^-) = 0.
	\end{equation}
	Assumption \textnormal{(C2)} guarantees that these boundary values do not vanish on $(b,a)$, and that they vanish at most like a square root near $a$ and $b$.

	To study the boundary values of $N^*$, we use the linear loop equation for $(g,n) = (0,2)$, which gives, for $x \in (b,a)$,
	\begin{equation}
		W_{0,2}^{\ast}(x_1,x^+)
		+
		W_{0,2}^{\ast}(x_1,x^-)
		+
		\frac{1}{(x_1 - x)^2}
		+
		\big\langle \partial_x O(x,\xi) W_{0,2}^{\ast}(\xi,x_1)\big\rangle_{\xi}
		=
		0.
	\end{equation}
	Integrating from $x=0$, we obtain, for $x \in (b,a)$,
	\begin{equation}
	\label{lloop02}
		\bigg(\int_{\ii 0^+}^{\infty} + \int_{\ii 0^-}^{\infty}\bigg)
		W_{0,2}^{\ast}(x_1,y)\dd y
		+
		\bigg(\int_{\infty}^{x^+} + \int_{\infty}^{x^-}\bigg)
		W_{0,2}^{\ast}(x_1,y)\dd y
		+
		\frac{1}{x_1 - x}
		-
		\frac{1}{x_1}
		+
		\big\langle O(x,\xi)W_{0,2}^{\ast}(\xi,x_1)\big\rangle_{\xi}
		=
		0.
	\end{equation}
	Since all terms except the integral of $W_{0,2}^{\ast}$ are holomorphic functions of $x \in \Omega \setminus [b,a]$, we get
	\begin{equation}
		\frac{N^{\ast}(x_1,x^+) + N^{\ast}(x_1,x^-)}{2}
		=
		\frac{1}{x_1}
		-
		\partial_q W_{0,1}^{\ast}(x_1)
		-
		\bigg(\int_{\ii 0^+}^{\infty} + \int_{\ii 0^-}^{\infty}\bigg)
		W_{0,2}^{\ast}(x_1,y)\dd y
		\eqqcolon
		C(x_1).
	\end{equation}
	This implies
	\begin{equation}
		S^{\ast}(x_1,x^+) - S^{\ast}(x_1,x^-)
		=
		\frac{2C(x_1)}{D^{\ast}(x^+)}.
	\end{equation}
	Since $0$ belongs to $(b,a)$ by \Cref{def:anagood}, the function $x \mapsto S^{\ast}(x_1,x)$ is continuous in a neighbourhood of $0$ if and only if $C(x_1)=0$. But $S^{\ast}(x_1,x)$ must have non-zero radius of convergence, because the set of stuffed skin maps is contained in the set of stuffed annular maps, and assumptions \textnormal{(A1)}--\textnormal{(A2)} require the generating series enumerating the latter to have non-zero radius of convergence. Thus $C(x_1)=0$. We have made the argument for $x_1 \in \mathbb{C} \setminus \overline{\Omega}$ to avoid discussing the simple pole at $x_1=x$ in the numerator of $S^{\ast}$, but the vanishing extends to the whole domain $\mathbb{C} \setminus [b,a]$, where $C$ is analytic.

	At this point, we have shown that $S^{\ast}(x_1,x)$ is continuous across $(b,a)$, leaving only isolated singularities at $a$ and $b$ in $\Omega$. Assumption \textnormal{(C3)} implies that $\sqrt{(x-a)(x-b)} W_{0,2}^{\ast}(x_1,x)$ remains bounded as $x$ approaches $a$ or $b$. Therefore, $\int_{\infty}^{x} 2W_{0,2}^{\ast}(x_1,y)\dd y$ remains bounded as $x$ approaches $a$ or $b$. Taking into account the vanishing of the denominator, we see that $\sqrt{(x-a)(x-b)}S^{\ast}(x_1,x)$ remains bounded as $x$ approaches $a$ or $b$, so these singularities are removable. Thus, $S^{\ast}(x_1,x)$ is a holomorphic function of $x \in \Omega$ and $x_1 \in \mathbb{C} \setminus \overline{\Omega}$, and it extends to a meromorphic function of $x \in \Omega$ and $x_1 \in \mathbb{C} \setminus [b,a]$ because of the simple pole at $x_1=x$ in the numerator.

	By assumptions \textnormal{(C2)} and \textnormal{(C3)}, the meromorphic function $s^{\ast}(z_1,z) = S^{\ast}(X(z_1),X(z))$, initially defined for $(z_1,z) \in \Omega_{\textnormal{in}}^2$, extends to a meromorphic function of $(z_1,z) \in \Sigma^2$. Moreover, for $(z_1,z) \in \Sigma_{\textnormal{out}} \times X^{-1}(\Omega)$, we have
	\begin{equation}
	\label{preformS}
		s^{\ast}(z_1,z)
		=
		\frac{
			\displaystyle\int_{\infty}^{z} 2w_{0,2}^{\ast}(z_1,\zeta)\dd X(\zeta)
			+
			\frac{1}{X(z_1) - X(z)}
			+
			\big\langle O(X(z),\xi) W_{0,2}^{\ast}(\xi,X(z_1))\big\rangle_{\xi}
			-
			\partial_q W_{0,1}^{\ast}(X(z_1))
		}{
			X(z)
			-
			\partial_{x} T_{0,1}(X(z))
			-			-
			\big\langle\partial_{x}\widetilde{O}(X(z),\xi) W_{0,1}^{\ast}(\xi)\big\rangle_{\xi} -
			2w_{0,1}^{\ast}(z)
		}.
	\end{equation}
	In the Zhukovsky variable, the linear loop equation \eqref{lloop1} reads, for $z \in \mathbb{U}$,
	\begin{equation}
		w_{0,1}^{\ast}(z)
		+
		w_{0,1}^{\ast}(1/z)
		-
		X(z)
		+
		\partial_{x} T_{0,1}(X(z))
		+
		\big\langle \partial_{x}\widetilde{O}(X(z),\xi) W_{0,1}^{\ast}(\xi)\big\rangle_{\xi}
		=
		0.
	\end{equation}
	The integrated linear loop equation \eqref{lloop02}, taking into account that $C(x_1)=0$, reads, for $z \in \mathbb{U}$,
	\begin{equation}
		\bigg(\int_{\infty}^{z} + \int_{\infty}^{1/z}\bigg)
		w_{0,2}^{\ast}(\zeta,z_1)\dd X(\zeta)
		+
		\frac{1}{(X(z_1) - X(z))^2}
		+
		\big\langle \partial_{x} O(X(z),\xi) W_{0,2}^{\ast}(\xi,X(z_1))\big\rangle_{\xi}
		=
		\partial_q W_{0,1}^{\ast}(X(z_1)).
	\end{equation}
	These are identities between meromorphic functions of $z$, and hence they hold for all $z \in \Sigma$. Using them, we can simplify \eqref{preformS} to
	\begin{equation}
		s^{\ast}(z_1,z)
		=
		\frac{
			-\displaystyle\int_{1/z}^{z} w_{0,2}^{\ast}(\zeta,z_1)\dd X(\zeta)
		}{
			w_{0,1}^{\ast}(z) - w_{0,1}^{\ast}(1/z)
		},
	\end{equation}
	where the fact that $w_{0,2}^{\ast}(\zeta,z_1)\dd X(\zeta)$ has no residues allows us to choose any path of integration between $1/z$ and $z$ that avoids $z_1$. Since this formula holds for all $z_1 \in \Sigma_{\textnormal{out}}$, it also holds for all $z_1 \in \Sigma$. The invariance under the involution $z \mapsto 1/z$ is manifest.
\end{proof}

\subsection{Deriving topological recursion}
\label{ssec:analytic:TR}
We now show that, once analytically good weights are fixed, the excision formula of \Cref{thm:excision} can be transformed into the usual form of blobbed topological recursion. The first ingredient is the interpretation of the pairing $\langle F(x)G(x)\rangle_x$ as a contour integral. The proof is immediate from Cauchy's residue theorem, so we omit it.

\begin{lem}
\label{lem:pairing}
	Let $\Omega' \subset \Omega$ be two open discs centred at $0$, and let
	\begin{equation}
		F(x) = \sum_{\ell \geq 1} F_{\ell}x^{\ell - 1},
		\qquad
		G(x) = \sum_{\ell \in \mathbb{Z}} G_{\ell}x^{-\ell}
	\end{equation}
	be two formal series. Assume that $F(x)$ converges absolutely for $x \in \Omega$, and that $G(x)$ converges absolutely for $x$ in the annulus $\Omega \setminus \overline{\Omega'}$. Then the pairing $\langle F(x)G(x) \rangle_x \coloneqq \sum_{\ell \geq 1} F_{\ell}G_{\ell}$ converges absolutely and is given by
	\begin{equation}
		\langle F(x)G(x)\rangle_x
		=
		\oint_\Gamma F(x)G(x)\frac{\dd x}{2\pi\ii}
	\end{equation}
	for any circle $\Gamma \subset \Omega \setminus \overline{\Omega'}$ with counterclockwise orientation.
\end{lem}

In the setup above, we now use the Zhukovsky change of variables in the integral representation. The functions $F(X(z))$ and $G(X(z))$ are initially defined and holomorphic for $z \in \Omega_{\textnormal{out}} \setminus \overline{\Omega'_{\textnormal{out}}}$. The domains $X^{-1}(\Omega') \subset X^{-1}(\Omega)$ are neighbourhoods of the unit circle, and are topological annuli. The contour $X^{-1}(\Gamma)$ is the union of two topological circles: $\Gamma_{\textnormal{out}} \subset \mathbb{C} \setminus \overline{\mathbb{D}}$, with counterclockwise orientation, and $\Gamma_{\textnormal{in}} \subset \mathbb{D}$, with clockwise orientation. The involution $z \mapsto 1/z$ sends $\Gamma_{\textnormal{out}}$ to $\Gamma_{\textnormal{in}}$. To proceed, we need additional assumptions on $F$ and $G$ in the Zhukovsky variable.

\begin{lem}
\label{lem:pairing:residue}
	Adopt the assumptions of \Cref{lem:pairing}. Assume that $f(z) = F(X(z))$ and $g(z) = G(X(z))$ admit analytic continuations as meromorphic functions in $X^{-1}(\Omega)$, with poles only at $z = \pm 1$, and that $f(z) = f(1/z)$ and $g(z) = g(1/z)$ for $z \in X^{-1}(\Omega)$. Then
	\begin{equation}
		\big\langle F(x)G(x) \big\rangle_x
		=
		\frac{1}{2}
		\Res_{z= \pm 1}
		f(z)g(z)\dd X(z).
	\end{equation}
\end{lem}

\begin{proof}
	We write
	\begin{equation}
		\big\langle F(x)G(x) \big\rangle_x
		=
		\oint_{\Gamma_{\textnormal{out}}}
		f(z)g(z)\frac{\dd X(z)}{2\pi\ii}
		=
		\frac{1}{2}
		\left(
			\oint_{\Gamma_{\textnormal{out}}}
			f(z)g(z)\frac{\dd X(z)}{2\pi\ii}
			+
			\oint_{\Gamma_{\textnormal{in}}}
			f(1/z)g(1/z)\frac{\dd X(1/z)}{2\pi\ii}
		\right).
	\end{equation}
	Since $X$, $f$ and $g$ are invariant under the involution $z \mapsto 1/z$, we obtain
	\begin{equation}
		\big\langle F(x)G(x) \big\rangle_x
		=
		\frac{1}{2}
		\oint_{\Gamma_{\textnormal{out}} + \Gamma_{\textnormal{in}}}
		f(z)g(z)\frac{\dd X(z)}{2\pi\ii}.
	\end{equation}
	We conclude by observing that, in the domain $X^{-1}(\Omega) \setminus \{\pm 1\}$ where the integrand is holomorphic, and the contour $\Gamma_{\textnormal{out}} + \Gamma_{\textnormal{in}}$ is homologous to the sum of two small counterclockwise circles centered at $+1$ and $-1$; see~\Cref{fig:contour:residue}.
\end{proof}

\begin{figure}
	\begin{tikzpicture}
		\draw [NavyBlue,->] (0,1.1) arc (90:-90:1.1);
		\draw [NavyBlue,->] (0,-1.1) arc (270:90:1.1);
		\node [NavyBlue] at (135:.7) {\small$\Gamma_{\textup{in}}$};

		\draw (1.5,0) arc (0:360:1.5);
		\node [BrickRed] at (1.5,0) {\tiny$\bullet$};
		\node [BrickRed] at (-1.5,0) {\tiny$\bullet$};

		\draw [NavyBlue,->] (0,-1.9) arc (-90:90:1.9);
		\draw [NavyBlue,->] (0,1.9) arc (90:270:1.9);
		\node [NavyBlue] at (135:2.3) {\small$\Gamma_{\textup{out}}$};


		\begin{scope}[xshift=5cm]
			\draw [NavyBlue,->] (0,-1.6) arc (-90:90:1.6);
			\draw [NavyBlue,->] (0,1.6) arc (90:270:1.6);

			\draw [NavyBlue,->] (0,1.4) arc (90:-90:1.4);
			\draw [NavyBlue,->] (0,-1.4) arc (270:90:1.4);

			\fill[white] (1.5,0) circle (.4);
			\draw [NavyBlue,->] (1.9,0) arc (0:180:.4);
			\draw [NavyBlue,->] (1.1,0) arc (180:360:.4);

			\fill[white] (-1.5,0) circle (.4);
			\draw [NavyBlue,->] (-1.1,0) arc (0:180:.4);
			\draw [NavyBlue,->] (-1.9,0) arc (180:360:.4);

			\draw [white,line width=1.13ex] (1.5,0) arc (0:20:1.5);
			\draw [white,line width=1.13ex] (1.5,0) arc (0:-20:1.5);
			\draw [white,line width=1.13ex] (-1.5,0) arc (180:160:1.5);
			\draw [white,line width=1.13ex] (-1.5,0) arc (180:200:1.5);

			\draw (1.5,0) arc (0:360:1.5);
			\node [BrickRed] at (1.5,0) {\tiny$\bullet$};
			\node [BrickRed] at (-1.5,0) {\tiny$\bullet$};

		\end{scope}

		\begin{scope}[xshift=10cm]
			\draw [NavyBlue,->] (1.9,0) arc (0:180:.4);
			\draw [NavyBlue,->] (1.1,0) arc (180:360:.4);

			\draw [NavyBlue,->] (-1.1,0) arc (0:180:.4);
			\draw [NavyBlue,->] (-1.9,0) arc (180:360:.4);

			\draw (1.5,0) arc (0:360:1.5);
			\node [BrickRed] at (1.5,0) {\tiny$\bullet$};
			\node [BrickRed] at (-1.5,0) {\tiny$\bullet$};
		\end{scope}
	\end{tikzpicture}
	\caption{\label{fig:contour:residue}
		Transformation of the contour integral into residues at $z=\pm1$.
	}
\end{figure}
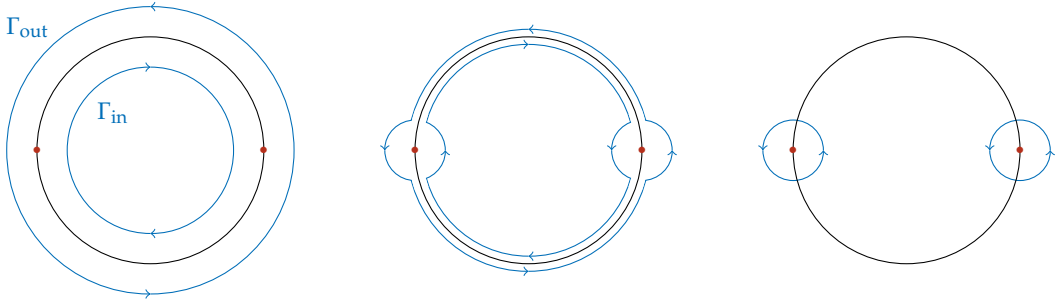

The conclusion of the proof of blobbed topological recursion is then close to the original proofs in \cite{Eyn04,Eyn16,BEO15,Bor14}. It uses the analytic realisation of the formal monodromy in~\eqref{intermono} and~\eqref{monotilde}, together with the analytic properties of the recursion kernel provided here by \Cref{kernel:an}. We nevertheless include the proof for completeness, and to correct the spurious factor of $\frac{1}{2}$ in the blob term in \cite{Bor14}.

\begin{thm}[Blobbed topological recursion]
\label{thm:blobbed:TR}
	Assume that the face and vertex weights $(q,\bm{t})$ are analytically good. Then the multidifferentials $\omega_{g,n}^{\ast}$ satisfy the blobbed topological recursion: for $2g-2+n>0$,
	\begin{equation}
	\label{eq:blobbed:TR}
	\begin{split}
		\omega_{g,n}^\ast(z_1,\bm{z}_I)
		&=
		\Res_{z=\pm1}
		K^\ast(z_1,z)
		\Bigg(
			\omega_{g-1,n+1}^\ast(z,1/z,\bm{z}_I)
			+
			\sum_{\substack{h' + h'' = g \\ J' \sqcup J'' = I}}^{\nodisc}
			\omega^\ast_{h',1+\card{J'}}(z,\bm{z}_{J'})
			\omega^\ast_{h'',1+\card{J''}}(1/z,\bm{z}_{J''})
		\Bigg) \\
		&\quad
		+
		\Phi_{g,n}(z_1,\bm{z}_I),
	\end{split}
	\end{equation}
	where the recursion kernel and the blob term are
	\begin{equation}
		K^{\ast}(z_1,z)
		\coloneqq
		\frac{
			\dfrac{1}{2}\displaystyle\int_{1/z}^z \omega_{0,2}^{\ast}(z_1,\cdot)
		}{
			\omega_{0,1}^{\ast}(z) - \omega_{0,1}^{\ast}(1/z)
		},
		\qquad
		\Phi_{g,n}(z_1,\ldots,z_n)
		\coloneqq
		\frac{1}{2\pi\ii}
		\oint_{\mathbb{U}}
		\omega_{0,2}^{\ast}(z_1,\zeta)
		V_{g,n}(X(\zeta);\bm{z}_I).
	\end{equation}
\end{thm}

\begin{proof}
	The excision formula of \Cref{thm:excision} is the identity of formal series
	\begin{equation}
	\label{eq:W:SQV}
		W^{\ast}_{g,n}(x_1,\bm{x}_I)
		=
		\big\langle
			S^{\ast}(x_1,\xi)
			Q^{\ast}_{g,n}(\xi,\xi,\bm{x}_I)
		\big\rangle_{\xi}
		+
		\big\langle
			W_{0,2}^{\ast}(x_1,\xi)
			V_{g,n}(\xi;\bm{x}_I)
		\big\rangle_{\xi},
	\end{equation}
	where
	\begin{equation}
	\label{eq:Q:def}
	\begin{split}
		Q^{\ast}_{1,1}(\xi',\xi'')
		&\coloneqq
		\mathcal{M}_{\xi''}\widetilde{W}_{0,2}^{\ast}(\xi',\xi''),
		\\
		Q^{\ast}_{g,n}(\xi',\xi'',\bm{x}_I)
		&\coloneqq
		\mathcal{M}_{\xi''}
		\Bigg(
			W_{g-1,n+1}^{\ast}(\xi',\xi'',\bm{x}_I)
			+
			\sum_{\substack{h'+h'' = g \\ J' \sqcup J'' = I}}^{\nodisc}
			W_{h',1+\card{J'}}^{\ast}(\xi',\bm{x}_{J'})
			W_{h'',1+\card{J''}}^{\ast}(\xi'',\bm{x}_{J''})
		\Bigg)
	\end{split}
	\end{equation}
	for $(g,n)\neq(1,1)$.

	Let $\Omega' \subset \Omega$ be an open disc centered at $0$ and containing $[b,a]$. After specialising the vertex and face weights to the given analytically good values, the tuple $\bm{x} \coloneqq (x_1,\ldots,x_n) \in (\mathbb{C} \setminus \overline{\Omega})^n$ and $(\xi',\xi'') \in (\Omega \setminus \overline{\Omega'})^2$ lies in the domain of absolute convergence of the following functions:
	\begin{itemize}
		\item $S^{\ast}(x_1,\xi)$ by \Cref{kernel:an}, and $W_{0,2}^{\ast}(x_1,\xi)$ by assumption \textnormal{(B1)};

		\item $Q^{\ast}_{g,n}(\xi',\xi'',\bm{x}_I)$, by assumptions \textnormal{(B1)}--\textnormal{(B2)} and the definition of the monodromy operator, together with \Cref{R1W:stable,R1W:unstable};

		\item $V_{g,n}(\xi;\bm{x}_I)$, by assumption \textnormal{(B3)}.
	\end{itemize}
	Thus, we can apply the lemmata above and compute $W_{g,n}^{\ast}$ as a residue or contour integral involving only generating series of lower complexity.

	We first treat the blob term. By \Cref{lem:pairing},
	\begin{equation}
	\label{eq:blob:pairing}
		\big\langle
			W_{0,2}^{\ast}(x_1,\xi)
			V_{g,n}(\xi;\bm{x}_I)
		\big\rangle_{\xi}
		=
		\oint_{\Gamma}
			W_{0,2}^{\ast}(x_1,\xi)
			V_{g,n}(\xi;\bm{x}_I)
			\frac{\dd \xi}{2\pi\ii}.
	\end{equation}
	Using assumption \textnormal{(C3)}, we rewrite this in terms of Zhukovsky variables as
	\begin{equation}
	\label{eq:blob:z}
		\big\langle
			W_{0,2}^{\ast}(x_1,\xi)
			V_{g,n}(\xi;\bm{x}_I)
		\big\rangle_{\xi}
		\dd X(z_1)
		=
		\oint_{\Gamma_{\textnormal{out}}}
		\left(
			\omega_{0,2}^{\ast}(z_1,\zeta)
			-
			\frac{\dd X(z_1)\dd X(\zeta)}{(X(z_1)-X(\zeta))^2}
		\right)
		V_{g,n}(X(\zeta);\bm{x}_I),
	\end{equation}
	where $z_1 = Z_+(x_1)$. Since $\omega_{0,2}^{\ast}(z_1,\zeta)$ has no singularity for $z_1$ outside $\overline{\Omega_{\textnormal{out}}}$ and $\zeta \in \Omega$, we can move the contour to the unit circle. The double-pole term is holomorphic for $\zeta \in \Omega$, because $x_1 \notin \overline{\Omega}$, so it does not contribute. Rewriting the result purely in terms of differential forms gives the announced formula for $\Phi_{g,n}$.

	We now transform the pairing involving $Q_{g,n}^{\ast}$ into a sum of residues. For this, we verify the assumptions of \Cref{lem:pairing:residue}. We first treat the case $(g,n)\neq(1,1)$. Fix $\bm{z} = (z_1,\ldots,z_n)$ in $\Sigma_{\textnormal{out}} \setminus \overline{\Omega_{\textnormal{out}}}$ and set $x_i = X(z_i)$. By the analytic interpretation of the formal monodromy in~\eqref{intermono}, based on the linear loop equations, the function
	\begin{equation}
		q^{\ast}_{g,n}(\zeta',\zeta'',\bm{z}_I)
		\coloneqq
		Q^{\ast}_{g,n}(X(\zeta'),X(\zeta''),X(\bm{z}_I)),
	\end{equation}
	initially defined for $\zeta',\zeta'' \in \Omega_{\textnormal{out}}$, admits a continuation as a meromorphic function of $\zeta',\zeta'' \in X^{-1}(\Omega)$ given by
	\begin{equation}
	\label{eq:q:analytic}
		q_{g,n}^{\ast}(\zeta',\zeta'',\bm{z}_I)
		=
		-\Bigg(
			w^{\ast}_{g-1,n+1}(\zeta',1/\zeta'',\bm{z}_I)
			+
			\sum_{\substack{h' + h'' = g \\ J' \sqcup J'' = I}}^{\nodisc}
			w^{\ast}_{h',1+\card{J'}}(\zeta',\bm{z}_{J'})
			w^{\ast}_{h'',1+\card{J''}}(1/\zeta'',\bm{z}_{J''})
		\Bigg).
	\end{equation}
	In particular, $q_{g,n}^{\ast}(\zeta,\zeta,\bm{z}_I)$ is invariant under $\zeta \mapsto 1/\zeta$, and assumptions \textnormal{(C2)}--\textnormal{(C3)} imply that its poles in $X^{-1}(\Omega)$ are located at $\zeta=\pm1$. By \Cref{kernel:an}, the function $s^{\ast}(z_1,\zeta)=S^{\ast}(X(z_1),X(\zeta))$, initially defined for $\zeta \in \Omega_{\textnormal{out}}$, admits an analytic continuation as a holomorphic function of $\zeta \in X^{-1}(\Omega)$ which is invariant under $\zeta \mapsto 1/\zeta$. Hence, by \Cref{lem:pairing:residue}, the first term of \eqref{eq:W:SQV} is equal to
	\begin{equation}
	\label{eq:SQ:residue}
		\frac{1}{2}
		\Res_{\zeta = \pm 1}
		s^{\ast}(z_1,\zeta)
		q_{g,n}^{\ast}(\zeta,\zeta,\bm{z}_I)
		\dd X(\zeta).
	\end{equation}

	Let us rewrite this in terms of differential forms. Set
	\begin{equation}
	\label{eq:kernel:Q}
	\begin{split}
		K^{\ast}(z_1,\zeta)
		&\coloneqq
		-\frac{1}{2}
		\frac{s^{\ast}(z_1,\zeta)\dd X(z_1)}{\dd X(\zeta)}
		=
		\frac{
			\dfrac{1}{2}\displaystyle\int_{1/\zeta}^{\zeta}
			\omega^{\ast}_{0,2}(z_1,\cdot)
		}{
			\omega_{0,1}^{\ast}(\zeta) - \omega_{0,1}^{\ast}(1/\zeta)
		},
		\\
		\mathcal{Q}^{\ast}_{g,n}(\zeta',\zeta'',\bm{z}_I)
		&\coloneqq
		-q_{g,n}^{\ast}(\zeta',\zeta'',\bm{z}_I)
		\dd X(\zeta')\dd X(1/\zeta'')
		\prod_{i = 2}^{n}\dd X(z_i).
	\end{split}
	\end{equation}
	Since $\dd X(\zeta)=\dd X(1/\zeta)$, \eqref{eq:SQ:residue} gives
	\begin{equation}
	\label{eq:quasi:TR}
		\omega_{g,n}^{\ast}(\bm{z})
		=
		\Res_{\zeta = \pm 1}
		K^{\ast}(z_1,\zeta)
		\mathcal{Q}_{g,n}^{\ast}(\zeta,\zeta,\bm{z}_I)
		+
		\Phi_{g,n}(\bm{z}).
	\end{equation}

	Returning to \eqref{eq:q:analytic}, we can express $\mathcal{Q}_{g,n}^{\ast}$ in terms of the multidifferentials $\omega$ of \Cref{def:anagood}, taking into account the double-pole shift in the definition of $\omega_{0,2}$:
	\begin{multline}
	\label{eq:Q:omega}
		\mathcal{Q}_{g,n}^{\ast}(\zeta,\zeta,\bm{z}_I)
		=
		\omega_{g-1,n+1}^{\ast}(\zeta,1/\zeta,\bm{z}_I)
		+
		\sum_{\substack{h' + h'' = g \\ J' \sqcup J'' = I}}^{\nodisc}
		\omega_{h',1+\card{J'}}^{\ast}(\zeta,\bm{z}_{J'})
		\omega_{h'',1+\card{J''}}^{\ast}(1/\zeta,\bm{z}_{J''})
		\\
		-
		\sum_{i = 2}^{n}
		\Bigl(
			\omega_{g,n-1}^{\ast}(\zeta,\bm{z}_{I_i})
			+
			\omega_{g,n-1}^{\ast}(1/\zeta,\bm{z}_{I_i})
		\Bigr)
		\frac{\dd X(z_i)(\dd X(\zeta))^2}{(X(z_i)-X(\zeta))^2}.
	\end{multline}
	The $i$-th term in the second line is holomorphic at $\zeta=\pm1$ by the linear loop equation of \Cref{lem:lloop}. Moreover, the quadratic differential $(\dd X(\zeta))^2$ has a double zero at $\zeta=\pm1$, compensating the simple pole in $K^{\ast}(z_1,\zeta)$ created by the division by $\dd X(\zeta)$ in~\eqref{eq:kernel:Q}. Hence, the second line of \eqref{eq:Q:omega} does not contribute to the residue in~\eqref{eq:quasi:TR}, and we obtain the claimed formula.

	For $(g,n)=(1,1)$, the discussion of the analytic properties is similar. Taking \eqref{monotilde} into account, we obtain~\eqref{eq:quasi:TR} with
	\begin{equation}
	\begin{split}
		\mathcal{Q}^{\ast}_{1,1}(\zeta',\zeta'')
		&=
		-q_{1,1}^{\ast}(\zeta',\zeta'')
		\dd X(\zeta')\dd X(\zeta'')
		\\
		&=
		\left(
			w_{0,2}^{\ast}(\zeta',1/\zeta'')
			+
			\frac{1}{(X(\zeta')-X(\zeta''))^2}
		\right)
		\dd X(\zeta')\dd X(\zeta'')
		\\
		&=
		\omega_{0,2}^{\ast}(\zeta',1/\zeta'')
	\end{split}
	\end{equation}
	for $\zeta',\zeta'' \in X^{-1}(\Omega)$. This gives the expected result.
\end{proof}

\begin{rem}
	The result of \Cref{thm:blobbed:TR} remains valid for complex, rather than non-negative, weight assignments whose absolute values define analytically good weights. Indeed, all generating series are analytic in their domains of convergence, and the identity of \Cref{thm:blobbed:TR}, valid for non-negative parameters, extends to the whole domain of analyticity by analytic continuation.
\end{rem}

\appendix
\section{Formal-series identities}

We record here a simple identity for triple products of formal series, which gives alternative ways to present
\begin{equation}
	\sum_{\substack{i,j,k \geq 0 \\ i - j + k = 0}} k \, F_{i} G_{j} H_{k}.
\end{equation}

\begin{lem}
\label{lem:triple:prod}
	Let $F(x)$ and $H(x)$ be formal power series with only non-negative powers of $x$, and let $G(x)$ be a formal series with only negative powers of $x$:
	\begin{equation}
		F(x) = \sum_{i\geq 0} G_i x^{i},
		\quad
		G(x) = \sum_{j\geq 0} G_j x^{-j-1},
		\quad
		H(x) = \sum_{k\geq 0} H_k x^{k}.
	\end{equation}
	Assume that all products appearing below are well defined. Then
	\begin{equation}
	\Bigg\langle
		F(x)\Bigg\langle
			\frac{G(\xi)}{x-\xi}\,\partial_\xi H(\xi)
		\Bigg\rangle_\xi
	\Bigg\rangle_x
	=
	\Bigg\langle
		\Bigg\langle
			\frac{F(\xi)G(\xi)}{(x-\xi)^2}
		\Bigg\rangle_\xi H(x)
	\Bigg\rangle_x
	=
	\langle F(x)G(x)\partial_x H(x)\rangle_x.
	\end{equation}
\end{lem}

\begin{proof}
	Denote the left-hand side by $[F,G,H]$. We decompose it as
	\begin{equation}
	\label{triplebra}
		[F,G,H]
		=
		\Bigg\langle
			\Bigg\langle
				G(\xi) \, \frac{F(x)\partial_\xi H(\xi)-F(\xi)\partial_x H(x)}{x-\xi}
			\Bigg\rangle_\xi
		\Bigg\rangle_x
		+
		\Bigg\langle
			\Bigg\langle
				\frac{F(\xi)G(\xi)}{x-\xi}
			\Bigg\rangle_\xi
			\partial_x H(x)
		\Bigg\rangle_x.
	\end{equation}
	In the first term, the inner coefficient extraction produces a formal power series with non-negative powers in $x$, so this term does not contribute to the outer bracket $\langle\cdot\rangle_x$. For the second term, formal integration by parts in the outer bracket gives
	\begin{equation}
		[F,G,H]
		=
		\Bigg\langle
			\Bigg\langle
				\frac{F(\xi)G(\xi)}{(x-\xi)^2}
			\Bigg\rangle_\xi H(x)
		\Bigg\rangle_x.
		\end{equation}
		On the other hand, using the expansion \eqref{x1:x2:expns}, we find
		\begin{equation}
			\Bigg\langle
				\frac{F(\xi)G(\xi)}{x-\xi}
			\Bigg\rangle_\xi
			=
			[F(x)G(x)]_{-,x}.
	\end{equation}
	Since in~\eqref{triplebra} this series is paired with the formal power series $\partial_x H(x)$, replacing $[F(x)G(x)]_{-,x}$ by $F(x)G(x)$ does not change the value of the outer bracket. This yields the second rewriting and completes the proof.
\end{proof}

\section{Maps without internal faces}

For maps with arbitrary $\bm{t}$, \Cref{skinning:maps} gives, for $2g-2+n>0$, a recursion of the form
\begin{equation}\label{TR:GUE}
	W_{g;\ell_1,\ldots,\ell_n}
	=
	\sum_{i=2}^{n}
	\sum_{\ell \geq 1}
	B_{\ell_1,\ell_i,\ell}\,
	W_{g;\ell,\bm{\ell}_{I\setminus\{i\}}}
	+
	\sum_{\ell',\ell'' \geq 1}
	C_{\ell_1,\ell',\ell''}
	\bigg(
		W_{g-1;\ell',\ell'',\bm{\ell}_I}
		+
		\sum_{\substack{h' + h'' = g \\ J' \sqcup J'' = I}}^{\textup{stable}}
		W_{h';\ell',\bm{\ell}_{J'}}
		W_{h'';\ell'',\bm{\ell}_{J''}}
	\bigg),
\end{equation}
where ``stable'' means that we exclude the terms containing generating series of discs or annuli. The cases with $2g-2+n=1$ may be regarded as initial data for the recursion. The recursion coefficients are obtained from skin maps and discs as
\begin{equation}
	B_{\ell_1,\ell_2,\ell}
	\coloneqq
	\sum_{m \geq 1} 2S_{\ell_1,\ell+m+2}W_{0;m,\ell_2}
	+
	\ell_2 S_{\ell_1,\ell-\ell_2+2},
	\qquad
	C_{\ell_1,\ell',\ell''}
	\coloneqq
	S_{\ell_1,\ell'+\ell''+2}.
\end{equation}
In other words, the initial data, together with the series $B$ and $C$, uniquely determine all higher topologies.

For maps without internal faces, that is, for $\bm{t}=0$, we can make these series completely explicit. Without loss of generality, we set the vertex weight to $q=1$. We also remark that this case is not new: maps without internal faces are already enumerated by Catalan numbers in disc topology, by Tutte's slicing formula \cite{Tut62} in the planar case with arbitrarily many boundaries, and by formulae of Lehman--Walsh \cite{WL72} or Harer--Zagier \cite{HZ86} in positive genus with a single boundary component. What we want to emphasise here is that, in this situation, the excision and skinning formulae become completely explicit and provide an effective, bijective way of computing the generating series $W_{g;\ell_1,\ldots,\ell_n}$ for arbitrary topology.

Let us define
\begin{equation}
	c_{\ell}
	\coloneqq
	\frac{\ell!}{\lfloor \frac{\ell}{2} \rfloor!\,\lfloor \frac{\ell-1}{2} \rfloor!}
	=
	\begin{cases}
		k \binom{2k}{k} & \text{if }\ell=2k, \\
		(2k+1)\binom{2k}{k} & \text{if }\ell=2k+1.
	\end{cases}
\end{equation}
Then we have
\begin{equation}\label{Catalan}
	W_{0;\ell}
	=
	\frac{4\,c_\ell}{\ell(\ell+2)}\,\delta_{\ell,\textup{even}},
	\qquad
	W_{0;\ell_1,\ell_2}
	=
	\frac{2c_{\ell_1}c_{\ell_2}}{\ell_1+\ell_2}\,\delta_{\ell_1+\ell_2,\textup{even}},
\end{equation}
for discs and annuli, and
\begin{equation}\label{ic}
	W_{0;\ell_1,\ell_2,\ell_3}
	=
	c_{\ell_1}c_{\ell_2}c_{\ell_3}\,\delta_{\ell_1+\ell_2+\ell_3,\textup{even}},
	\qquad
	W_{1;\ell}
	=
	\frac{\ell-2}{24}\,c_\ell\,\delta_{\ell,\textup{even}}.
\end{equation}
for pairs of pants and one-holed tori. To compute $B$ and $C$, we need to enumerate skin maps. To this end, observe that, by comparing \Cref{Tutte:lemma} and \Cref{skinning:maps}, the operator $f(x)\mapsto \langle S(x_1,x)f(x)\rangle_x$ is the inverse of
\begin{equation}
	f(x)\longmapsto \mathcal{L}_{x}f
	=
	f(x)\bigl(x-2W_{0,1}(x)\bigr)
	=
	f(x)\sqrt{x^2-4}.
\end{equation}
The formula for $x-2W_{0,1}(x)$ is easily obtained from \Cref{Catalan} and the generating series of Catalan numbers. Hence, for $\ell_1,\ell \geq 1$,
\begin{equation}
	S_{\ell_1,\ell}
	\coloneqq
	\big\langle S(x_1,x)x_1^{\ell_1}x^{-\ell}\big\rangle_{x_1,x}
	=
	\Bigg\langle \frac{x_1^{\ell_1-\ell}}{\sqrt{x_1^2-4}}\Bigg\rangle_{x_1}
	=
	\binom{\ell_1-\ell}{(\ell_1-\ell)/2}\,\delta_{\ell_1+\ell,\textup{even}}
	=
	\frac{\ell_1-\ell+2}{2}\text{Cat}((\ell_1-\ell)/2).
\end{equation}
In particular, $S_{\ell_1,\ell}=0$ unless $\ell_1 \geq \ell$. This gives
\begin{equation}\label{BC}
\begin{aligned}
	B_{\ell_1,\ell_2,\ell}
	&=
	\begin{cases}
		\ell_2 \binom{\ell_1+\ell_2-\ell-2}{(\ell_1+\ell_2-\ell-2)/2}
		&
		\text{if }\ell_1-\ell_2 \leq \ell+2 \leq \ell_1
		\text{ and }
		\ell_1+\ell_2+\ell \text{ is even},
		\\
		2^{\ell_1-\ell-1} c_{\ell_2}
		&
		\text{if }\ell_1>\ell+2
		\text{ and }
		\ell_1+\ell_2+\ell \text{ is even},
		\\
		0
		&
		\text{otherwise},
	\end{cases} \\
	C_{\ell_1,\ell',\ell''}
	&=
	\binom{\ell_1-\ell'-\ell''-2}{(\ell_1-\ell'-\ell'')/2-1}\,\delta_{\ell_1+\ell'+\ell'',\textup{even}}.
\end{aligned}
\end{equation}
Indeed, the formula for $C$ is immediate. As for $B$, write $\beta_{\ell_1,\ell_2,\ell} \coloneqq \sum_{m \geq 1} 2S_{\ell_1,\ell+m+2}W_{0;m,\ell_2}$ for the first term in $B_{\ell_1,\ell_2,\ell}$. If $\ell_1-\ell_2 \leq \ell+2 \leq \ell_1$, then the term $\beta_{\ell_1,\ell_2,\ell}$ vanishes, and therefore
\begin{equation}
	B_{\ell_1,\ell_2,\ell}
	=
	\ell_2 S_{\ell_1,\ell-\ell_2+2}
	=
	\ell_2 \binom{\ell_1+\ell_2-\ell-2}{(\ell_1+\ell_2-\ell-2)/2},
\end{equation}
provided that $\ell_1+\ell_2+\ell$ is even.

Assume now that $\ell_1>\ell+2$. If $\ell_2=2k_2$ is even, only the terms $m=2p$ contribute to $\beta$, and $\ell_1$ and $\ell$ must have the same parity. Writing $d = \frac{\ell_1-\ell-2}{2}$, we obtain
\begin{equation}
\begin{split}
	\beta_{\ell_1,\ell_2,\ell}
	&=
	2k_2 \binom{2k_2}{k_2}
	\sum_{p=1}^{d}
	\frac{p}{p+k_2}
	\binom{2p}{p}
	\binom{2d-2p}{d-p} \\
	&=
	2k_2
	\Bigg[
		4^d \binom{2k_2}{k_2}
		-
		2\binom{2d+2k_2-1}{d+k_2}
	\Bigg] \\
	&=
	2^{\ell_1-\ell-1}c_{\ell_2}
	-
	\ell_2
	\binom{\ell_1+\ell_2-\ell-2}{(\ell_1+\ell_2-\ell-2)/2}.
\end{split}
\end{equation}
If instead $\ell_2=2k_2+1$ is odd, only the terms $m=2p+1$ contribute to $\beta$, and $\ell_1$ and $\ell$ must have opposite parity. Writing $d = \frac{\ell_1-\ell-3}{2}$, we obtain
\begin{equation}
\begin{split}
	\beta_{\ell_1,\ell_2,\ell}
	&=
	2(2k_2+1)\binom{2k_2}{k_2}
	\sum_{p=0}^{d}
	\frac{2p+1}{p+k_2+1}
	\binom{2p}{p}
	\binom{2d-2p}{d-p} \\
	&=
	2(2k_2+1)
	\Bigg[
		2^{2d+1}\binom{2k_2}{k_2}
		-
		\binom{2d+2k_2+1}{d+k_2}
	\Bigg] \\
	&=
	2^{\ell_1-\ell-1}c_{\ell_2}
	-
	\ell_2
	\binom{\ell_1+\ell_2-\ell-2}{(\ell_1+\ell_2-\ell-2)/2}.
\end{split}
\end{equation}
In both cases, the sum in the second line was evaluated using the identities
\begin{equation}
	\sum_{p=0}^d \binom{2p}{p}\binom{2d-2p}{d-p}=4^d,
	\qquad
	r\binom{2r}{r}\sum_{p=0}^d \frac{1}{p+r}\binom{2p}{p}\binom{2d-2p}{d-p}
	=
	\binom{2d+2r}{d+r},
\end{equation}
applied with $r=k_2$ and $r=k_2+1$, respectively.

The final expression is independent of the parity of $\ell_2$, and the subtracted term cancels with $\ell_2 S_{\ell_1,\ell-\ell_2+2}$. This proves the claimed formula. The recursion \labelcref{TR:GUE}, together with the initial data \labelcref{ic} and the explicit expressions \labelcref{BC} for $B$ and $C$, provides an effective computation of the generating series of maps without internal faces.

\printbibliography

\end{document}